\documentclass{siamart171218}
\usepackage[top=1in, bottom=1in, left=1in, right=1in]{geometry}

\usepackage{a4wide}

\newcommand{\mat}[1]{\mathbf{#1}}
\newcommand{\vc}[1]{\mathbf{#1}}
\newcommand{\truMat}{\mat{A}}
\newcommand{\sampMat}{\hat{\truMat}}
\newcommand{\rhs}{\vc{b}}
\newcommand{\truSol}{\vc{x}}
\newcommand{\sampSol}{\hat{\truSol}}
\newcommand{\dimension}{n}
\newcommand{\R}{\mathbb{R}}

\newcommand{\N}{\mathbb{N}}
\newcommand{\Rdim}{\R^{\dimension}}
\newcommand{\Rdimdim}{\R^{\dimension \times \dimension}}
\newcommand{\impSol}{\tilde{\truSol}}
\newcommand{\impSolAt}[1]{\impSol_{#1}}
\newcommand{\pdcone}{S_+(\Rdim)}
\newcommand{\err}[2]{\mathcal{E}_{#2}(#1)}
\newcommand{\errId}[1]{\mathcal{E}(#1)}
\newcommand{\errBayes}[2]{\mathcal{E}_{#2}^{\text{Bayes}}(#1)}

\newcommand{\E}{\mathbb{E}}
\newcommand{\normMat}{\mat{B}}
\newcommand{\id}{\mat{I}}
\newcommand{\errMat}{\hat{\mat{Z}}}
\newcommand{\augMat}{\hat{\mat{K}}}
\newcommand{\augFac}{\beta}
\newcommand{\optAugFac}{\augFac^*}

\newcommand{\rhsPrior}{\mathcal{B}}

\newcommand{\autoCor}{\mat{R}}
\newcommand{\modMat}{\mat{C}}
\newcommand{\modMatTwo}{\mat{D}}
\newcommand{\combA}{\mat{L}}
\newcommand{\trvec}{\hat{\vc{v}}}
\newcommand{\truncFac}[1]{\augFac_{#1}}
\newcommand{\window}{\omega}
\newcommand{\windowFunc}[2]{\window^{#1}(#2)}
\newcommand{\windowDenom}{\omega_*}
\newcommand{\windowFuncDenom}[2]{\windowDenom^{#1}(#2)}
\newcommand{\basepointfactor}{\alpha}
\newcommand{\basepointfactorFunc}{\basepointfactor(\sampMat)}
\newcommand{\basepointfactorFuncOf}[1]{\basepointfactor(#1)}
\newcommand{\genWindow}{\Omega}
\newcommand{\genWindowDenom}{\Omega_*}
\newcommand{\genWindowPoly}[1]{\genWindow^{#1}}
\newcommand{\genWindowPolyDenom}[1]{\genWindowDenom^{#1}}
\newcommand{\xMat}{\hat{\mat{X}}}
\newcommand{\sMat}{\mat{S}}
\newcommand{\primmPoly}[1]{\Delta^{#1}}
\newcommand{\primmPolyDenom}[1]{\Delta^{#1}_*}
\newcommand{\basepointerrMat}{\errMat_{\basepointfactor}}
\newcommand{\basepointxMat}{\xMat_{\basepointfactor}}
\newcommand{\windowFuncBP}[2]{\window^{#1}_\basepointfactor(#2)}
\newcommand{\windowFuncBPDenom}[2]{\window^{#1}_{\basepointfactor,*}(#2)}
\newcommand{\genWindowBP}{\Omega_\alpha}
\newcommand{\genWindowDenomBP}{\Omega_{\alpha,*}}
\newcommand{\genWindowBPPoly}[1]{\genWindowBP^{#1}}
\newcommand{\genWindowDenomBPPoly}[1]{\genWindowDenomBP^{#1}}
\newcommand{\primmPolyBP}[1]{\Delta_\basepointfactor^{#1}}
\newcommand{\primmPolyDenomBP}[1]{\Delta_{\basepointfactor,*}^{#1}}
\newcommand{\transformedMean}{\eta}
\newcommand{\sampleTrVec}{\hat{\vc{q}}}
\newcommand{\sampSampleTrVec}[1]{\hat{\vc{q}}_{#1}}

\newcommand{\bootstrapAugFac}{\tilde{\augFac}^*}
\newcommand{\bootstrapSampMat}{\sampMat_b}
\newcommand{\sampBootstrapSampMat}[1]{\sampMat_{b, #1}}
\newcommand{\bootstrapMonteCarloAugFac}{\hat{\augFac}^*}
\newcommand{\accelTruncFac}[1]{\bar{\augFac}_{#1}}
\newcommand{\accelWindowFuncBP}[2]{\window^{#1}_{\basepointfactorFunc}(#2)}
\newcommand{\accelWindowFuncBPDenom}[2]{\window^{#1}_{\basepointfactorFunc,*}(#2)}
\newcommand{\accelBasepointerrMat}{\errMat_{\basepointfactorFunc}}
\newcommand{\bootAccelBasepointerrMat}[1]{\errMat_{b,\basepointfactorFuncOf{#1}}}
\newcommand{\bootAccelWindowFuncBP}[3]{\window^{#1}_{\basepointfactorFuncOf{#3}}(#2)}
\newcommand{\bootAccelWindowFuncBPDenom}[3]{\window^{#1}_{\basepointfactorFuncOf{#3},*}(#2)}
\newcommand{\bootstrapTruncAugFac}[1]{\tilde{\augFac}_{#1}}
\newcommand{\bootstrapMonteCarloTruncAugFac}[1]{\hat{\augFac}_{#1}}
\newcommand{\incidenceMat}{\mat{E}}
\newcommand{\weightMat}{\mat{W}}
\newcommand{\laplacian}{\mat{L}}
\newcommand{\sampLaplacian}{\hat{\laplacian}}
\newcommand{\sampWeightMat}{\hat{\weightMat}}
\newcommand{\truEdge}{w_e}
\newcommand{\sampTruEdge}{\hat{w}_e}
\newcommand{\params}{\omega}
\newcommand{\paramSpace}{\Omega}
\newcommand{\paramDist}[1]{\mathbb{P}_{#1}}
\newcommand{\paramToMat}{\mathcal{M}}
\newcommand{\truParam}{\params^*}
\newcommand{\truMatDist}{D}
\newcommand{\matDistAt}[1]{\truMatDist_{#1}}
\newcommand{\matDist}{\matDistAt{\truParam}}
\newcommand{\sampParam}{\hat{\params}}
\newcommand{\bootstrapDist}{\matDistAt{\sampParam}}

\DeclareMathOperator*{\argmin}{arg\,min}
\DeclareMathOperator*{\var}{var}
\DeclareMathOperator*{\cov}{cov}
\DeclareMathOperator*{\tr}{tr}
\DeclareMathOperator*{\supp}{supp}
\DeclareMathOperator*{\bias}{bias}

\newcommand\independent{\protect\mathpalette{\protect\independenT}{\perp}}
\def\independenT#1#2{\mathrel{\rlap{$#1#2$}\mkern2mu{#1#2}}}

\usepackage{lipsum}
\usepackage{amsfonts, amssymb}
\usepackage{graphicx}
\usepackage{epstopdf}
 \usepackage{algpseudocode}
\usepackage{comment}
\usepackage{cleveref}
\usepackage{thmtools}
\usepackage{dsfont}
\usepackage{thm-restate}

\ifpdf
  \DeclareGraphicsExtensions{.eps,.pdf,.png,.jpg}
\else
  \DeclareGraphicsExtensions{.eps}
\fi

\newsiamremark{remark}{Remark}
\newsiamremark{hypothesis}{Hypothesis}
\crefname{hypothesis}{Hypothesis}{Hypotheses}
\newsiamthm{claim}{Claim}

\headers{Operator Shifting for Noisy Elliptic Systems}{P.A. Etter, L. Ying}

\title{Operator Shifting for Noisy Elliptic Systems}

\author{Philip A. Etter \thanks{Institute for Computational and Mathematical Engineering, Stanford University} (\email{paetter@stanford.edu}) \and Lexing Ying \thanks{Department of Mathematics, Stanford University} (\email{lexing@stanford.edu})}

\usepackage{amsopn}

\makeatletter
\newcommand*{\addFileDependency}[1]{
  \typeout{(#1)}
  \@addtofilelist{#1}
  \IfFileExists{#1}{}{\typeout{No file #1.}}
}
\makeatother

\begin{document}

\maketitle

\begin{abstract}
In the computational sciences, one must often estimate model parameters from data subject to noise and uncertainty, leading to inaccurate results. In order to improve the accuracy of models with noisy parameters, we consider the problem of reducing error in an elliptic linear system with the operator corrupted by noise. We assume the noise preserves positive definiteness, but otherwise, we make no additional assumptions the structure of the noise. Under these assumptions, we propose the \emph{operator shifting} framework, a collection of easy-to-implement algorithms that augment a noisy inverse operator by subtracting an additional shift term. In a similar fashion to the James-Stein estimator, this has the effect of drawing the noisy inverse operator closer to the ground truth by reducing both bias and variance. We develop bootstrap Monte Carlo algorithms to estimate the required shift magnitude for optimal error reduction in the noisy system. To improve the tractability of these algorithms, we propose several approximate polynomial expansions for the operator inverse, and prove desirable convergence and monotonicity properties for these expansions. We also prove theorems that quantify the error reduction obtained by operator shifting. In addition to theoretical results, we provide a set of numerical experiments on four different graph and grid Laplacian systems that all demonstrate effectiveness of our method. 
\end{abstract}

\begin{keywords}
Operator shifting, Random Matrices, Monte Carlo, Polynomial Expansion, Elliptic Systems.
\end{keywords}

\begin{AMS}
Linear and multilinear algebra; matrix theory. Statistics. Computer Science.
\end{AMS}

\section{Introduction}

There are a plethora of different situations in the natural, mathematical, and computer sciences that necessitate computing the solution to a linear system of equations given by
\begin{equation} \label{eq:base_system}
    \truMat \truSol = \rhs \,,
\end{equation}
where $\truMat \in \Rdimdim$ and $\truSol, \rhs \in \Rdim$ for $\dimension \in \N$. When both the matrix $\truMat$ and $\rhs$ are known, there are many decades of research on how to solve the system \cref{eq:base_system} efficiently. Unfortunately, for a variety of reasons, it is often the case that the true matrix $\truMat$ is not known exactly, and must be estimated from data (see \cite{soize2005comprehensive, palmer2005representing}). In this situation, there is an error between the unobserved true matrix $\truMat$ and the matrix $\sampMat$ one constructs from data. The discrepancy between $\truMat$ and $\sampMat$ is often referred to as \emph{model uncertainty}, as it stems from incomplete or inaccurate information about the underlying system. This model uncertainty means that with naive application of the inverse of the observed matrix $\sampMat$, one is not solving the desired system \cref{eq:base_system}, but rather, the system 
\begin{equation} \label{eq:noisy_system}
    \sampMat \sampSol = \rhs \,,
\end{equation}
where $\sampSol = \sampMat^{-1} \rhs \in \Rdim$ is the solution we observe when we solving the observed system naively. Often, we will write
\begin{equation}
    \sampMat = \truMat + \errMat\,,
\end{equation}
where one can think of the matrix $\errMat$ as constituting the noise or sampling error in our measurements of the system \cref{eq:base_system}. Hence, the sampling error $\errMat$ between $\truMat$ and $\sampMat$ translates into error between the true solution $\truSol$ and the naively estimated solution $\sampSol$. 

The question of interest in this paper is whether, using the information available to us, we can find a better approximation $\impSol$ for the true solution $\truSol$ by modifying how we solve the sampled system \cref{eq:noisy_system}. ``Better'' here means in the sense of average error measured in the norm of some symmetric positive definite matrix $\normMat$, i.e., that we have
\begin{equation}
    \err{\impSol}{\normMat} < \err{\sampSol}{\normMat} \,,
\end{equation}
where the error functional $\err{\cdot}{\normMat}$ is defined as
\begin{equation} \label{eq:errDef}
    \err{\sampSol}{\normMat} \equiv \E[\|\sampSol - \truSol\|_{\normMat}^2] = \E[\| \sampMat^{-1} \rhs - \truMat^{-1} \rhs\|_\normMat^2 ] \,,
\end{equation}
where the norm $\|\cdot\|_\normMat$ is defined $\| \vc{x} \|_\normMat^2 = \vc{x}^T \normMat \vc{x}$. The two norms of particular interest to us are the $L^2$ norm (for obvious reasons), i.e., $\normMat = \id$, as well as the energy norm, i.e., $\normMat = \truMat$, as the latter is an important metric of error in many physical problems. 

Many traditional techniques approach this problem by imposing Bayesian regularization conditions on the sampled solution $\sampSol$ (e.g., Tikhonov regularization \cite{tikhonov1963solution}) or applying post-processing on $\sampSol$. In this paper, we take a fundamentally different tact. Instead of thinking about the problem of improving the individual estimates $\sampSol$ of solutions $\truSol$, we propose herein a framework for thinking about the problem in terms of linear operators. We content that this paradigm shift is quite useful --- as it is often the case that one may be interested in solving more than just one system of the form \cref{eq:noisy_system} given a single estimate $\sampMat$ of the matrix $\truMat$. In this situation, it often makes more sense to think of improving the estimator $\sampMat^{-1}$ rather than improving individual estimators $\sampSol$, although the two are obviously related. In light of this, we will amend our earlier objective \cref{eq:errDef} slightly. Namely, instead of achieving low error on just a single right-hand side $\rhs$, we want to simultaneously perform well on a whole collection of possible right-hand sides of interest. For this reason, we suppose that $\rhs$ is sampled from a distribution $\rhsPrior$ and that our goal is to reduce the average error over this distribution,
\begin{equation}
    \E_{\rhs \sim \rhsPrior} \E_{\sampMat} [\| \sampMat^{-1} \rhs - \truMat^{-1} \rhs\|_\normMat^2 ] \,.
\end{equation}

In the interest of building out this new perspective, we propose a novel method we call \emph{operator shifting}. The idea of operator shifting is to add an augmenting term to the sampled inverse operator $\sampMat^{-1}$, yielding a family of operators
\begin{equation}
    \sampMat^{-1} - \augFac \augMat(\sampMat^{-1})
\end{equation}
parameterized by an shift factor $\augFac \in \R$, for a choice of shift operator $\augMat(\sampMat^{-1}) \in \Rdimdim$ depending on the problem setting. Note that the shift operator is a function of the sampled matrix $\sampMat$. Our new approximation for $\truSol$ is then given by
\begin{equation} \label{eq:opAug}
    \impSolAt{\augFac} =  (\sampMat^{-1} - \augFac \augMat) \rhs
\end{equation}
Through judicious selection of the shift operator $\augMat$, we show that one can estimate a $\augFac$ that will reduce error by a factor that depends on the variance of the naive solution $\sampSol$. As we will see, the power of operator shifting lies in the fact that the technique works under very minimal assumptions on the randomness structure of $\sampMat$; in general, the only assumption we need to guarantee error reduction is that $\sampMat$ is an unbiased estimator of $\truMat$, and even this assumption can be relaxed.

The most obvious choice of shift operator is perhaps to shift the naive estimate $\sampMat^{-1}$ towards the origin by taking $\augMat(\sampMat) = \sampMat^{-1}$. There are two fundamental reasons why one might expect this to be a good choice of shift --- the first concerns the bias of the estimate $\sampMat^{-1}$ and the second concerns the variance.

\begin{enumerate}
    \item \textbf{Bias}: For symmetric positive definite matrices, the matrix inversion operation is convex with respect to the L\"owner order\footnote{Recall the definition of the L\"owner order: $A \preceq B$ when $x^T A x \leq x^T B x$ for all vectors $x \in \Rdim$.}. A matrix analogue of Jensen's inequality therefore suggests that, depending on the variance in $\sampMat$, $\sampMat^{-1}$ will substantially overestimate $\truMat^{-1}$ on average (i.e., $\E [ \sampMat^{-1}] \succeq \truMat^{-1}$). Hence, it makes to shift $\sampMat^{-1}$ towards the origin in order to reduce the bias in $\sampMat^{-1}$. We provide an illustration of this bias in \cref{fig:jensen}.
    \item \textbf{Variance}: Shrinking the estimate towards a fixed point (i.e., the origin) simultaneously has the effect of reducing variance in the estimator. This is analogous to the seminal work of James and Stein \cite{james1992estimation} that demonstrated the standard mean estimator is inadmissible, as shrinking the estimator slightly towards the origin always reduces average error.
\end{enumerate}
Therefore, the confluence of these two factors suggest that we should expect a reduction in both bias and variance, and hence a more accurate estimator as a result. Indeed, in this paper we prove that, with only minimal assumptions on the randomness of $\sampMat$, that the optimal reduction in error always comes from a shift towards the origin. 

\begin{figure}
    \centering
    \includegraphics[scale=0.25]{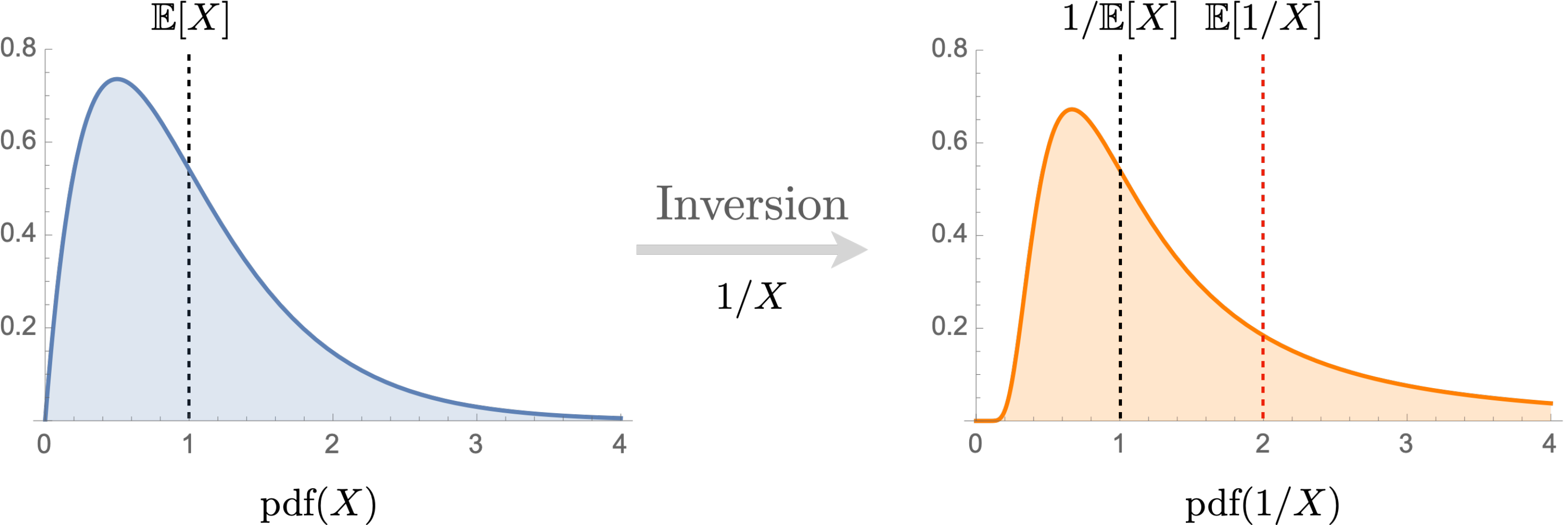}
    \caption{An example of the overshooting effect. If we take a single sample of the of the scalar random variable $X \sim \Gamma(2, 1/2)$, and invert it, the pdf of the inverted $1/X$ has an expectation that is significantly larger (x2) than the inversion $1/\E[X]$. This means that naively trying to estimate $1/\E[X]$ with only a single sample will likely give a significant overestimate. The same principle also applies when $X$ is a random matrix.}
    \label{fig:jensen}
\end{figure}

We structure the remainder of the paper as follows. In \cref{sec:opAugmentation}, we present the basic operator shifting formalism and examine the special case where $\rhs$ is deterministic to motivate the full operator shifting technique. In \cref{sec:semiBayesian}, we present the full version of operator shifting in the setting where the vector $\rhs$ is now also drawn from a probability distribution $\rhsPrior$. We prove bounds that quantify how much error our technique can reduce. In \cref{sec:energyNorm}, we consider operator shifting in the aforementioned energy norm, i.e., when $\normMat = \truMat$, and prove similar bounds. In \cref{sec:taylorApprox}, we show that the energy norm has special monotonicity properties that are immensely useful for efficiently computing a good choice of $\augFac$. Finally, in \cref{sec:numerics}, we present numerical experiments to verify the theoretical results in this paper.

Note that we consider only elliptic systems in this paper, i.e., requiring that $\truMat$ is symmetric positive definite and $\sampMat$ is symmetric positive definite almost surely --- however, one could theoretically apply the techniques we present herein to asymmetric systems as well, but we do not provide any theoretical guarantees in the asymmetric case.

Finally, in order to help readers quickly implement our method without getting caught up in all of the surrounding mathematical details, we provide the quick start \cref{sec:quickstart} to give readers an alternate entry point to the algorithm we present in this paper. For the accompanying source code for this paper, please see \cref{sec:source}.

\section{Related Work}

The spirit of our approach is heavily influenced by James-Stein Estimation \cite{james1992estimation}. In Stein's original paper, \cite{stein1956inadmissibility}, he demonstrated the (at the time) shocking phenomenon that the standard mean estimator is actually \emph{inadmissible} for the quadratic loss in dimensions $\geq 3$. The reason behind this has to do with the fact that one can always advantageously trade off bias for a reduction in variance by shrinking the estimator towards any fixed point. At a fundamental level, one can frame our work as taking this idea and applying it to the novel setting of matrices corrupted by noise.

Some particularly relevant work pertains to debiasing distributed second-order optimization. Second-order optimization methods often rely on solving a symmetric linear system involving the Hessian of an object (positive-definite if the objective is strongly convex). However, in many machine learning applications, the objective is composed of a summation of terms over a massive corpus of data, such that computing the true Hessian is extremely expensive. Instead, practitioners often turn to stochastic optimization methods that subsample the objective and its derivatives by using only a small section of the corpus at a time. However, for an optimization problem given by
\begin{equation}
    \min_x \sum_{i = 0}^m f_i(x) \,,
\end{equation}
the tru Hessian and approximated Hessian are given as follows:
\begin{equation} \label{eq:hessianestimate}
    \mat{H} = \sum_{i = 0}^m \mat{H}_i \, \qquad \hat{\mat{H}} = \sum_{i = 0}^m \frac{\hat{p}_i}{\E[\hat{p}_i]} \mat{H}_i \,,
\end{equation}
where $\mat{H}_i$ is the Hessian of $f_i$ and $p_i \in \{0, 1\}$ is a random variable that determines if the $i$-th item in the corpus is in the current mini-batch. The naive estimator $\hat{\mat{H}}$ has an upward bias and there has been work in the literature on how to de-bias the estimator using determinantal averaging \cite{derezinski2019distributed}. However, this approach is fundamentally limited to matrix ensembles of the form \cref{eq:hessianestimate}. In this paper, the types of noise we consider are far more general. 

Other relevant work has been done in the field of matrix sketching. Matrix sketching is a technique to reduce the complexity of a least-squares/linear problem by using random sketches of the rows/columns of the matrix. This process can likewise produce estimates that are substantially biased. One can attempt to address this bias by modifying regularization or other problem parameters \cite{derezinski2020debiasing}. This can be applied for the aforementioned second order optimization problem by using a Hessian sketch. However, again, the technique is tied to a very specific type of matrix noise. 

Beyond the world of James-Stein estimation and operator de-biasing, there are a number of immediate connections between the work done herein and previous work in the field of statistical inverse problems. In various inverse problems, one is interested in estimation from noisy or incomplete measurements. For example, \emph{semi-blind deconvolution} involves trying to reconstruct a function convolved with a kernel where the kernel is known, but with some uncertainty. Note that this is distinct from fully \emph{blind deconvolution} where one has no information about the kernel. In the sense that the measurement operator is corrupted by noise or uncertainty, and the goal is to recover the underlying object by inverting a linear system, this setting is quite similar to our own and hence worth mentioning. 

A common approach to these problems is to induce regularization on both the operator and the recovery target. For example, Total Least Squares algorithms as pioneered by Golub and Van Loan \cite{golub1980analysis} optimize over small perturbations to the noisy operator as well as the linear regression weights. Similar approaches specific to semi-blind deconvolution includes introducing a free estimate of the underlying kernel with regularization to match the observed data \cite{buccini2018semiblind}. Another technique in semi-blind deconvolution is to treat the full operator as a free variable and introduce optimization constraints to make sure that the operator and the observations do not deviate by too much \cite{bleyer2013double}.

Unfortunately, these types of techniques that operator over the operator suffer from a number of flaws. The most obvious is that introducing $\sim \dimension^2$ additional free variables into an optimization problem also introduces a substantial additional computational cost. Along with this computational cost also comes a much more severe chance of over-fitting unless regularization is handled appropriately. Furthermore, these regularization techniques implicitly depend on good Bayesian priors for what the underlying target and the operator should look like. In the absence of good priors, this optimization avenue may not be as viable. In contrast, all optimizations performed in the operator shifting framework we present here are only over a single variable $\augFac$, and hence are not subject to these concerns.

Other situtations in the statistical inverse problem literature that involve noisy or uncertain operators include circumstance where the forward operator may be far too expensive to apply directly, and hence must be replaced by a learned proxy for efficient computation \cite{lunz2021learned}. Another setting in the literature is when one has a set of noisy input-output pairs of the underlying operator. Work has been done on using these input-output pairs to construct regularizers for solving the inverse problem \cite{aspri2020data}.  Nonetheless, this approaches and settings are quite different from the approach and setting we present in this paper.  

Beyond the field of statistical inverse problems, a pertinent area of the literature related to our work is \emph{model uncertainty}. Quantifying and representing model uncertainty is important in many different fields of computational science, ranging from structural dynamics \cite{soize2005comprehensive} to weather and climate prediction \cite{palmer2005representing}. However, work relating to model or parameter uncertainty is usually domain specific and focuses more on establishing a model for uncertainty than it does on trying to reduce error in the resulting predictions. In contrast, our work focuses entirely on reducing error, rather than quantifying it. Our work is also not restricted to a particular domain, class of problems, or randomness structure, as long as those problems are linear. 

We note that our setting shares some similarities with the problem of uncertainty quantification (UQ). However, the problem we face here is different from the standard uncertainty quantification setting in a subtle but very important way. In UQ, one is usually given a distribution $\mathcal{P}$ and a map $T$ and asked to estimate statistics about the pushforward distribution $T_* \mathcal{P}$ (i.e., expectation, standard error, etc.). Practitioners typically accomplish this task via Monte Carlo techniques \cite{marzouk2016introduction} or some form of stochastic Galerkin projection \cite{xiu2002wiener} or collocation method \cite{xiu2005high}. However, for our purposes, we are more interested in the image of the statistic $\E[\sampMat] = \truMat$ under matrix inversion, rather than quantifying the pushforward of the distribution of $\sampMat$ under matrix inversion.

The central problem in this paper is also not dissimilar to the setting of matrix completion seen in \cite{candes2010matrix, keshavan2009matrix}. In matrix completion, one usually seeks to recover a low-rank ground truth matrix $\mat{M}_{ij}$ from observations that have been corrupted by additive noise, e.g., $\mat{N} = \mat{M} + \mat{Z}$. If $\mathcal{P}_{\Omega}$ denotes the subset sampling operator on matrix space, then one is trying to recover $\mat{M}$ from
\begin{equation}
    \mathcal{P}_{\Omega}(\mat{N}) = \mathcal{P}_{\Omega}(\mat{M}) + \mathcal{P}_{\Omega}(\mat{Z}) \,.
\end{equation}
However, the operator shifting and matrix completion settings are subtly different. The matrix completion analogue of $\truMat$ is the actual linear operator $\mathcal{P}_{\Omega}$, and not the matrix $\mat{M}$. Morally, one may think of the matrix competition problem as solving the under-determined linear system
\begin{equation}
    \mathcal{P}_\Omega (\mat{M}) = \mathcal{P}_{\Omega} (\mat{N}) \,
\end{equation}
by assuming a low-rank regularity on $\mat{M}$. The randomness in this problem lies completely in the right hand side $\mat{N}$, and not in the actual linear operator $\mathcal{P}_{\Omega}$.

We also draw attention to the related field of perturbation matrix analysis. In this setting, one is usually interested in proving results about how various properties of matrices change under a perturbation to the elements of the matrix. A seminal example of work in this field is the Davis-Kahan Theorem \cite{davis1970rotation}, which quantifies the extent to which the invariant sub-spaces of a matrix change under perturbations. In a similar vein, work in backward stability analysis revolves around understanding the behavior of the solution of a linear system under perturbations to the matrix. However, backward stability analysis typically adopts a worst-case mentality in analysis. In contrast, we care about average case error --- and more importantly, how one can reduce it.

We should briefly mention that the mathematical branch of random matrix theory (RMT) studies the spectral properties of random matrix ensembles \cite{anderson2010introduction, tao2012topics}. However, RMT results usually apply only when the entries of the random matrices are independent and in the large matrix limit. We find these assumptions to be too stringent for the problem at hand.

In addition to these tangentially related settings, we also call attention to the similarity of some of our techniques to those in harmonic analysis. It is well known that the method of summation of an infinite series can affect the conditions under which it convergence, as well as the quality of the convergence. For example, the Fourier series of a continuous function $f$ on the unit interval $[0, 1]$ may not converge pointwise to $f$ if summed naively. But Fej\'er's Theorem (see  \cite{stein2011fourier}) states that C\'esaro and Abel sums of the Fourier series of an integrable function $f$ converge uniformly to $f$ at any point of continuity. Our work takes on a similar favor in that it revolves heavily around the convergence properties of partial sums of infinite series expansion of the matrix function $f(\mat{A}) = \mat{A}^{-1}$. These partial sums are critical to accelerating an otherwise expensive Monte Carlo computation, hence we develop methods of partial summation that have desirable properties --- such as convergence and monotonicity.

In conclusion, we do not believe that the setting we introduce in this paper, where the operator is noisy, has been studied in the proposed fashion before. There is little precedent in the literature for the operator shifting method we present herein.

\section{Basic Assumptions and Notation} \label{sec:assumptions}

For the sake of transparency, before we go any further, we will make a number of assumptions on the nature of randomness on $\sampMat$ --- as this will help clarify the setting. We will use $\truMatDist$ to denote the distribution of $\sampMat$. Throughout this paper, we will use $\pdcone$ to denote the set of symmetric positive definite matrices in $\Rdimdim$. We make the following extremely lax assumptions about the randomness of $\sampMat$:

\begin{enumerate}
    \item \emph{Almost-Surely Positive Definite}: We assume that $\sampMat \in \pdcone$ almost surely. We believe this is a very reasonable assumption, if $\sampMat$ is generated from an elliptic problem whose parameters are subject to noise, it is extremely unlikely that any value of the underlying problem parameters will destroy ellipticity.
    \item \emph{Unbiased, or Downward-Biased}: We assume that $\sampMat$ is an unbiased estimate of $\truMat$, i.e., that $\E[\sampMat] = \truMat$. More generally, all of the machinery applies equally well when $\E[\sampMat] \preceq \truMat$.
    \item \emph{Finiteness of the Inverse Second Moment}: We assume that $\E[\sampMat^{-2}] \prec \infty$. Note that this is necessary to ensure that our measure of error \cref{eq:errDef} actually exists for arbitrary choice of $\rhs$. 
\end{enumerate}
We note that these assumptions are surprisingly lax. Most importantly, we \emph{do not} assume that entries of $\sampMat$ are independent. In the context of the theory to be presented herein, this assumption is irrelevant and not needed. Moreover, for all of the numerical examples we present, the entries of $\sampMat$ will in fact be correlated random variables. We believe this helps reinforce the generality of the operator shifting framework.

\section{Warm-up: Deterministic Right-Hand Side} \label{sec:opAugmentation}

Let us suppose that $\sampMat \sim \truMatDist$ satisfies the conditions outlined in the previous section. A substantial amount of the theory in the subsequent sections is simply a generalization of the case where $\rhs$ is a deterministic vector. We also only consider the $L^2$ error norm for now, i.e., $\normMat = \id$, and simply write $\errId{\cdot}$ for the $L^2$ error $\err{\cdot}{\id}$. We hope these simplified assumptions will help us easily communicate the core idea of the proof we use in the subsequent sections.

Supposing the above, operator shifting operates by finding a good choice of shift factor $\augFac$ to minimize error. Indeed, the choice of $\augFac$ we would like to make is the minimizer of the error,
\begin{equation}
    \optAugFac = \argmin_{\augFac \in \R} \errId{(\sampMat^{-1} - \augFac \augMat)\rhs } \,.
\end{equation}
In particular, note that under the choice of 1-parameter family in \cref{eq:opAug}, the error functional becomes a quadratic in $\augFac$,
\begin{equation}
\begin{split}
    \errId{(\sampMat^{-1} - \augFac \augMat)\rhs} &=  \E \| (\sampMat^{-1} - \augFac \augMat) \rhs - \truMat^{-1} \rhs \|_2^2 \\
    &= \E \| \sampMat^{-1} \rhs - \truMat^{-1} \rhs \|_2^2 - 2 \augFac \E \left[ \rhs^T \augMat^T (\sampMat^{-1} - \truMat^{-1}) \rhs \right] + \augFac^2 \E \| \augMat \rhs \|_2^2 \\
    &= \errId{\sampMat^{-1} \rhs} - 2 \augFac \E \left[ \rhs^T \augMat^T (\sampMat^{-1} - \truMat^{-1}) \rhs \right] + \augFac^2 \E \| \augMat \rhs \|_2^2
\end{split}
\end{equation}
Hence, the optimal choice of $\augFac$ is given by
\begin{equation} \label{eq:optAugFac}
    \optAugFac = \frac{\E \left[ \rhs^T \augMat^T (\sampMat^{-1} - \truMat^{-1}) \rhs \right]}{ \E \left[ \| \augMat \rhs \|_2^2 \right]} \,.
\end{equation}
And under this optimal choice of $\optAugFac$, we see a reduction of error given by
\begin{equation} \label{eq:optErrRed}
    \errId{(\sampMat^{-1} - \augFac \augMat)\rhs} = \errId{\sampMat^{-1} \rhs} - \frac{\E \left[ \rhs^T \augMat^T (\sampMat^{-1} - \truMat^{-1}) \rhs \right]^2}{ \E \left[ \| \augMat \rhs \|_2^2 \right]} \,,
\end{equation}
where note that $\errId{\sampMat^{-1} \rhs}$ is the error when solving the system naively using $\sampMat$. Unfortunately, this choice of $\optAugFac$ depends on knowledge of the true inverse matrix $\truMat^{-1}$ and hence cannot be implemented exactly. We will address this issue later by using a bootstrap Monte Carlo technique -- for now, let us try to develop bounds for $\optAugFac$ as well as the optimal reduction in error.

In accordance with the intuition we presented in the introduction, in practice one observes that the sampling error in the matrix $\sampMat$ causes the $\sampSol$ to have a tendency to overshoot the true solution $\truSol$. Hence, a reasonable thing one might try to correct for this is to uniformly scale the entries of the estimated solution $\sampSol$ by some scalar value in the range $[0, 1]$. Note that this would correspond to a choice of shift operator and factor given by
\begin{equation}
    \augMat = \sampMat^{-1}, \qquad \augFac \geq 0 \,.
\end{equation}
Indeed, under this choice of $\augMat$, the numerator of \cref{eq:optAugFac} becomes
\begin{equation} \label{eq:upper}
    \E \left[ \rhs^T \augMat^T (\sampMat^{-1} - \truMat^{-1}) \rhs \right] = \rhs^T ( \E [ \sampMat^{-T} \sampMat ] - \E [\sampMat^{-T}] \truMat^{-1}) \rhs \,.
\end{equation}
This expansion above is strongly suggestive of the moment formula for covariance matrices, given by
\begin{equation}
    \E [\hat{\mat{Y}}^T \hat{\mat{Y}}] - \E[\hat{\mat{Y}}]^T \E[\hat{\mat{Y}}] = \cov(\hat{\mat{Y}}) \succeq 0 \,,
\end{equation}
one might suspect that it would therefore be possible to lower bound the troublesome term in \cref{eq:upper} by something that depends on the variance of the estimated solution $\sampMat^{-1} \rhs = \sampSol$ and not directly on $\truMat$. Unfortunately, the asymmetry of \cref{eq:upper} and the fact that $\E[\sampMat^{-1}] \neq \truMat^{-1}$ makes this difficult.

However, there is --- fortunately --- something we can say about the quantity $\E[\sampMat^{-1}]$ under mild assumptions, namely,

\begin{restatable}[L\"owner Order Inversion]{lemma}{loewnerinversionlemma} \label{lem:loewner}
Suppose that $\truMat \in \pdcone$ and $\sampMat \in \pdcone$ almost surely. Moreover, suppose that, $\truMat$ spectrally dominates $\sampMat$ in expectation, i.e.,
\begin{equation}
    \E[ \sampMat ] \preceq \truMat \,,
\end{equation}
then, matrix inversion inverts the expected L\"owner order, i.e.,
\begin{equation}
    \E[ \sampMat^{-1} ] \succeq \truMat^{-1}
\end{equation}
\end{restatable}

The proof of this fact is given in the appendix and relies on the fact that the function $\vc{u}^T \truMat^{-1} \vc{u}$ for arbitrary $\vc{u} \in \Rdim$ is convex in $\truMat$ when $\truMat \in \pdcone$. We would like to use this fact that $\truMat^{-1} \preceq \E[\sampMat^{-1}]$ and say that the right hand side of \cref{eq:upper} is bounded below by $\rhs^T \cov(\sampMat) \rhs > 0$ --- however, this would be incorrect. Indeed, the fact that the matrix $\E [\sampMat^{-T}] \truMat^{-1}$ in \cref{eq:upper} is not symmetric makes this line of inquiry difficult. 

Therefore, instead of shrinking the entire estimated solution $\sampSol$ (a strategy we will return to later), let us consider simply shrinking the vector $\sampSol$ along the $\rhs$ component (for reasons that will soon be apparent), i.e.,
\begin{equation}
    \impSolAt{\augFac} = \sampSol - \augFac \frac{1}{\|\rhs\|_2^2} \rhs \rhs^T \sampSol \,.
\end{equation}
This move corresponds to choosing
\begin{equation}
     \augMat = \frac{1}{\|\rhs\|_2^2} \rhs \rhs^T \sampMat^{-1}\,.
\end{equation}
Note that, from \cref{eq:optAugFac}, the optimal choice of $\augFac$ is now given by
\begin{equation}
    \optAugFac = \frac{\E [(\rhs^T \sampMat^{-1} \rhs)^2 ] - \E[\rhs^T \sampMat^{-1} \rhs] (\rhs^T \truMat^{-1} \rhs)}{\E[(\rhs^T \sampMat^{-1} \rhs)^2] } \equiv \frac{\xi(\sampMat^{-1})}{\E[(\rhs^T \sampMat^{-1} \rhs)^2]}\,,
\end{equation}
with corresponding error from \cref{eq:optErrRed},
\begin{equation} \label{eq:errRed1}
    \errId{\sampMat^{-1} - \optAugFac \augMat} = \errId{\sampMat^{-1}} - \frac{1}{\| \rhs\|_2^2 } \frac{\xi(\sampMat^{-1})^2}{ \E[(\rhs^T \sampMat^{-1} \rhs)^2]} \,,
\end{equation}
By the unbiased-ness assumption $\E[\sampMat] = \truMat$ (or $\E[\sampMat] \preceq \truMat$), so from \cref{lem:loewner}, it immediately follows that
\begin{equation} \label{eq:flipbound}
    \rhs^T \truMat^{-1} \rhs \leq \E[ \rhs^T \sampMat^{-1} \rhs] \,.
\end{equation}
And therefore,
\begin{equation}
   -\E[ \rhs^T \sampMat^{-1} \rhs]  (\rhs^T \truMat^{-1} \rhs) + \E[ \rhs^T \sampMat^{-1} \rhs]^2 \geq 0 \,.
\end{equation}
Thus, adding the left hand of the above inequality to $\xi$ gives us that:
\begin{equation}
\begin{split}
    \xi(\sampMat^{-1}) &= \E [(\rhs^T \sampMat^{-1} \rhs)^2 ] - \E[\rhs^T \sampMat^{-1} \rhs] (\rhs^T \truMat^{-1} \rhs) \\
    &\geq \E [(\rhs^T \sampMat^{-1} \rhs)^2 ] - 2 \, \E[\rhs^T \sampMat^{-1} \rhs] (\rhs^T \truMat^{-1} \rhs) + \E[ \rhs^T \sampMat^{-1} \rhs]^2 \\
    &= \E[(\rhs^T \sampMat^{-1} \rhs - \rhs^T \truMat^{-1} \rhs)^2] \\
    &= \errId{\rhs^T \sampMat^{-1} \rhs}
\end{split}
\end{equation}
where here $\errId{\rhs^T \sampMat^{-1} \rhs}$ denotes the mean-squared error of the estimator $\rhs^T \sampMat^{-1} \rhs$ for the quantity $\rhs^T \truMat^{-1} \rhs$. 

Ergo, we immediately have that
\begin{equation}
    1 \geq \optAugFac \geq \frac{\errId{\rhs^T \sampMat^{-1} \rhs}}{\E[(\rhs^T \sampMat^{-1} \rhs)^2]} \geq 0 \,,
\end{equation}

Applying the same inequality to \cref{eq:errRed1} we get a bound on the optimal reduction in error,
\begin{equation} \label{eq:errRed2}
    \errId{(\sampMat^{-1} - \optAugFac \augMat)\rhs} \leq \errId{\sampMat^{-1}\rhs} - \frac{1}{\|\rhs\|^2} \frac{\errId{\rhs^T \sampMat^{-1} \rhs}^2 }{ \E[(\rhs^T \sampMat^{-1} \rhs)^2]} \,,
\end{equation}

But before discussing the computation of $\optAugFac$, let us restate the above result in a theorem:

\begin{theorem} \label{thm:simpleAug}
Under the assumptions in \cref{sec:assumptions}, consider the operator shifting algorithm with the choice $\augMat = \frac{1}{\|\rhs\|_2^2} \rhs \rhs^T \sampMat^{-1}$. We have that the optimal shift always satisfies
\begin{equation}
    1 \geq \optAugFac \geq \frac{\errId{\rhs^T \sampMat^{-1} \rhs}}{\E[(\rhs^T \sampMat^{-1} \rhs)^2]} \geq 0 \,,
\end{equation}
and furthermore the optimal reduction in relative error satisfies
\begin{equation}
    \max_{\augFac \in \R} \frac{\errId{\sampMat^{-1}\rhs} - \errId{(\sampMat^{-1} - \augFac \augMat)\rhs}}{\errId{\sampMat^{-1}\rhs}} \geq \left(\frac{\errId{\rhs^T \sampMat^{-1} \rhs}}{\E[(\rhs^T \sampMat^{-1} \rhs)^2] } \right) \left(\frac{\errId{\rhs^T \sampMat^{-1} \rhs / \|\rhs\|_2}}{\errId{\sampMat^{-1} \rhs}} \right) \,,
\end{equation}
where $\errId{\hat{\vc{x}}}$ denotes the mean-squared error for an estimator $\hat{\vc{x}}$.
\end{theorem}

\subsection{Discussion}

The theorem above serves a mostly illustrative purpose to provide a simplified version of the proofs we will present in the subsequent section. However, we believe it still warrants some discussion. The result about the optimal reduction in error is what we would likely expect given our earlier comments about shrinkage reducing both variance and bias, indeed, one can use the standard bias-variance decomposition of error to write:
\begin{equation}
    \begin{split}
    &\frac{\errId{\sampMat^{-1}\rhs} - \errId{(\sampMat^{-1} - \optAugFac \augMat)\rhs}}{\errId{\sampMat^{-1}\rhs}} \geq \\
    &\qquad \left(\frac{\var(\rhs^T \sampMat^{-1} \rhs) + \bias(\rhs^T \sampMat^{-1} \rhs)^2}{\E[(\rhs^T \sampMat^{-1} \rhs)^2] } \right) \left(\frac{\errId{\rhs^T \sampMat^{-1} \rhs / \|\rhs\|_2}}{\errId{\sampMat^{-1} \rhs}} \right) \,,
    \end{split}
\end{equation}
we therefore clearly see that the larger the variance or the bias of the $\rhs$ component of the solution $\sampMat^{-1} \rhs$, the larger one expects the reduction in error to be --- clearly in line with our expectations.

Moreover, note that both $\frac{\errId{\rhs^T \sampMat^{-1} \rhs}}{\E[(\rhs^T \sampMat^{-1} \rhs)^2] }$ and $\frac{\errId{\rhs^T \sampMat^{-1} \rhs / \|\rhs\|_2}}{\errId{\sampMat^{-1} \rhs}}$ are dimensionless quantities and have clear interpretations. Both are ratios between $0$ and $1$. The first one is the error in the $\rhs$ component of the solution relative to the second moment, and the second one is the ratio of the total error in the $\rhs$ component of the solution to the solution as a whole. The second term is simply a penalty we pay for only addressing error in the $\rhs$ component --- as the error in all other directions remains unaddressed. This means we will likely to expect the second term to be on the order of $\sim 1/\dimension$. 

Naturally, this means this technique of shifting along only one direction is likely not a very good algorithm to use in practice. The obvious solution to this $\sim 1/\dimension$ factor is to try to shrink $\sampMat^{-1}$ along multiple directions at the same time, and not just along the single direction $\rhs$. Furthermore, there is the other issue that the shifted operator is no longer symmetric with the choice of $\augMat = \frac{1}{\|\rhs\|_2^2} \rhs \rhs^T \sampMat^{-1}$. In practice, symmetry typically corresponds to important physical properties (i.e., reversibility), so there are good reasons why one may want an operator shift that maintains symmetry and doesn't just shrink one component of the solution. 

\section{Operator Shifting in Operator Inner Product Norms} \label{sec:semiBayesian}

To address the issues with shrinking along a single component of the solution in the section above, we will pivot to thinking about the problem at hand as an operator estimation problem. In practice, one may not be simply interested in a single right-hand side $\rhs$, but rather, producing a good inverse operator for a wide variety of potential right-hand sides $\rhs$. As discussed in our introduction, we encode this desire by changing our error metric to have $\rhs$ be sampled from a known distribution $\rhsPrior$ and then measuring the average error under this distribution,
\begin{equation}
    \E_{\rhs \sim \rhsPrior} \E_{\sampMat} [\| \sampMat^{-1} \rhs - \truMat^{-1} \rhs\|_\normMat^2 ] \,.
\end{equation}
When $\rhs$ is made into a random variable, the result actually induces a metric on the space of operators $\Rdimdim$. To see this, let $\autoCor \equiv \E[\rhs \rhs^T]$ denote the second moment matrix of the distribution $\rhsPrior$ and consider the following manipulations,
\begin{equation}
\begin{split}
    \E_{\rhs \sim \rhsPrior} \E_{\sampMat \sim \truMatDist} [\| \sampMat^{-1} \rhs - \truMat^{-1} \rhs\|_\normMat^2 ] &=  \E_{\rhs \sim \rhsPrior} \E_{\sampMat \sim \truMatDist} [\rhs^T (\sampMat^{-1}- \truMat^{-1})^T \normMat (\sampMat^{-1}- \truMat^{-1}) \rhs] \\
    &= \E_{\rhs \sim \rhsPrior} \E_{\sampMat \sim \truMatDist} \tr [\rhs^T (\sampMat^{-1}- \truMat^{-1})^T \normMat (\sampMat^{-1}- \truMat^{-1}) \rhs] \\
    &= \E_{\rhs \sim \rhsPrior} \E_{\sampMat \sim \truMatDist} \tr [(\sampMat^{-1}- \truMat^{-1})^T \normMat(\sampMat^{-1}- \truMat^{-1}) \rhs \rhs^T] \\
    &= \E_{\sampMat \sim \truMatDist} \tr [(\sampMat^{-1}- \truMat^{-1})^T \normMat(\sampMat^{-1}- \truMat^{-1}) \, \E_{\rhs \sim \rhsPrior}(\rhs \rhs^T)] \\
    &= \E_{\sampMat \sim \truMatDist} \tr [(\sampMat^{-1}- \truMat^{-1})^T \normMat(\sampMat^{-1}- \truMat^{-1}) \autoCor] \\
    &= \E_{\sampMat \sim \truMatDist} \tr [\autoCor^{1/2} (\sampMat^{-1}- \truMat^{-1})^T \normMat(\sampMat^{-1}- \truMat^{-1}) \autoCor^{1/2}] \\
\end{split}
\end{equation}
The natural metric and norm on operator space that corresponds to this notion of error is therefore defined by:
\begin{equation}
    \begin{split}
    \langle \mat{X}, \mat{Y} \rangle_{\normMat, \autoCor} &\equiv \tr [\autoCor^{1/2} (\sampMat^{-1}- \truMat^{-1})^T \normMat(\sampMat^{-1}- \truMat^{-1}) \autoCor^{1/2}] \\
    \|\mat{X}\|_{\normMat, \autoCor}^2 &\equiv \langle \mat{X}, \mat{X} \rangle_{\normMat, \autoCor} \,,
    \end{split}
\end{equation}
where $\mat{X}, \mat{Y} \in \Rdimdim$. Note that the often-used Frobenius norm $\|\cdot\|_F$ is a special case of this class of norms that we obtain when $\normMat = \autoCor = \id$.

Therefore, the pivot from thinking about obtaining lower error in a specific $\rhs$ to obtaining lower error on a collection of $\rhs$ essentially changes our problem to an estimation problem for $\sampMat^{-1}$ in the $\|\cdot\|_{\normMat, \autoCor}$ norm. Corresponding to this change in outlook, we will use the notation
\begin{equation}
    \err{\sampMat^{-1}}{\normMat, \autoCor} \equiv \E \|\sampMat^{-1} - \truMat^{-1} \|_{\normMat, \autoCor}^2 = \E_{\rhs \sim \rhsPrior} \E_{\sampMat \sim \truMatDist} [\| \sampMat^{-1} \rhs - \truMat^{-1} \rhs\|_\normMat^2 ] \,,
\end{equation}
to denote the $(\normMat,\autoCor)$-error of the estimator $\sampMat^{-1}$. 

Now, let us again introduce an operator shift to the operator $\sampMat^{-1}$,
\begin{equation}
    \sampMat^{-1} - \augFac \augMat(\sampMat^{-1}) \,.
\end{equation}
A quick dimensional analysis of the above quantity suggests that $\augMat(\sampMat^{-1})$ should be a linear function of $\sampMat^{-1}$. Therefore, it makes sense to study operator shifts of the form
\begin{equation}
    \augMat(\sampMat^{-1}) = \modMat \sampMat^{-1} \modMatTwo \,,
\end{equation}
where $\modMat, \modMatTwo$ are matrices. In service of a similar analysis to the one in the previous section, we note that there is a compatibility constraint on $\modMat$ and $\modMatTwo$ that forces the result to play especially nice with the $\langle \cdot, \cdot \rangle_{\normMat, \autoCor}$ inner product, namely,
\begin{equation} \label{eq:compatConditions}
    \autoCor \modMatTwo^T = \modMat^T \normMat, \qquad (\autoCor \modMatTwo^T) = (\autoCor \modMatTwo^T)^T, \qquad \autoCor \modMatTwo^T \succeq \mat{0} \,.
\end{equation}
We will see soon why this is the case. Some examples when this may be the case are as follows:
\begin{enumerate}
    \item The trivial case where $\autoCor = \normMat = \modMatTwo = \modMat = \id$.
    \item The case where $\autoCor = \normMat$ and $\modMatTwo = \modMat = \id$. 
    \item The case where $\modMat = \autoCor$, $\modMatTwo = \normMat$ and $[\normMat, \autoCor] = 0$.
    \item The case where $\modMat = \normMat^{-1}$ and $\modMatTwo = \autoCor^{-1}$.
\end{enumerate}

With this choice, we can essentially repeat the theorem of the previous section, but now with an eye towards the operator estimation viewpoint,

\begin{theorem} \label{thm:fullOpAug_revision}
    Under the assumptions in \cref{sec:assumptions}, consider operator shifting in the $\|\cdot\|_{\normMat, \autoCor}$-norm. Any operator shift $\augMat = \modMat \sampMat^{-1} \modMatTwo$ such that $\modMat, \modMatTwo \in \Rdimdim$ satisfy the compatibility conditions \cref{eq:compatConditions} has an optimal shift factor that satisfies:
    \begin{equation}
        \sqrt{\frac{\err{\sampMat^{-1}}{ \normMat, \autoCor}}{\E\|\sampMat^{-1}\|^2_{\modMat^T \normMat \modMat, \modMatTwo \autoCor \modMatTwo^T}}} \geq \optAugFac \geq \frac{\err{\sampMat^{-1}}{\modMat^T \normMat, \autoCor \modMatTwo^T}}{\E\|\sampMat^{-1}\|^2_{\modMat^T \normMat \modMat, \modMatTwo \autoCor \modMatTwo^T}} \geq 0 \,.
    \end{equation}
    And the corresponding optimal reduction in error is given by
    \begin{equation}
        \max_{\augFac \in \R} \frac{\err{\sampMat^{-1}}{\normMat, \autoCor} - \err{\sampMat^{-1} - \augFac \augMat}{\normMat, \autoCor}}{\err{\sampMat^{-1}}{\normMat, \autoCor}} \geq \frac{\err{\sampMat^{-1}}{\modMat^T \normMat, \autoCor \modMatTwo^T}^2}{\E\|\sampMat^{-1}\|^2_{\modMat^T \normMat \modMat, \modMatTwo \autoCor \modMatTwo^T} \, \err{\sampMat^{-1}}{\normMat, \autoCor}} \,,
    \end{equation}
    where $\err{\hat{\mat{X}}}{\normMat, \autoCor}$ is the mean squared error of matrix estimator $\hat{\mat{X}}$ in the $\|\cdot\|_{\normMat, \autoCor}$-norm.
\end{theorem}

\begin{proof}
    We would like to repeat the results of the previous section, except now we want to choose the shift factor $\augFac$ that optimizes the $(\normMat,\autoCor)$-error. Just like the previous section, we obtain
    \begin{equation}
        \err{\sampMat^{-1} - \augFac \augMat}{\normMat, \autoCor} = \err{\sampMat^{-1}}{\normMat, \autoCor} - 2 \augFac \, \E \langle \augMat, \sampMat^{-1} - \truMat^{-1} \rangle_{\normMat, \autoCor} + \augFac^2 \, \E \| \augMat \|_{\normMat, \autoCor}^2 \,,
    \end{equation}
    and hence, the optimal shift factor $\optAugFac$ is given by
    \begin{equation} \label{eq:itfollows1}
        \optAugFac = \frac{\E \langle \augMat, \sampMat^{-1} - \truMat^{-1} \rangle_{\normMat, \autoCor}}{\E \| \augMat \|_{\normMat, \autoCor}^2 } \,,
    \end{equation}
    and the corresponding optimal error is
    \begin{equation} \label{eq:itfollows2}
        \err{\sampMat^{-1} - \optAugFac \augMat}{\normMat, \autoCor} = \err{\sampMat^{-1}}{\normMat, \autoCor}  - \frac{(\E \langle \augMat, \sampMat^{-1} - \truMat^{-1} \rangle_{\normMat, \autoCor})^2}{\E \| \augMat \|_{\normMat, \autoCor}^2 } \,,
    \end{equation}

    Let us expand the quantity
    \begin{equation} \label{eq:wanttocompletesquares}
        \begin{split}
        \E \langle \augMat, \sampMat^{-1} - \truMat^{-1} \rangle_{\normMat, \autoCor} &= \E \langle \modMat \sampMat^{-1} \modMatTwo, \sampMat^{-1} - \truMat^{-1} \rangle_{\normMat, \autoCor} \\
        &= \E \langle \sampMat^{-1} , \sampMat^{-1} - \truMat^{-1} \rangle_{\modMatTwo^T \normMat, \autoCor \modMat^T} 
        \end{split}
    \end{equation}

    We want to repeat the argument of the theorem in the previous section. Namely, we would like to have
    \begin{equation}
        \E \langle \truMat^{-1}, \sampMat^{-1} - \truMat^{-1} \rangle_{\modMatTwo^T \normMat, \autoCor \modMat^T} \geq 0 \,,
    \end{equation}
    so that we can complete the square in \cref{eq:wanttocompletesquares}. To prove this fact, we will use $\mat{M} = \mat{M}^T$ to denote $\modMatTwo^T \normMat = \autoCor \modMat^T \succeq \mat{0}$. Now, we simply need to do some manipulations inside the trace,
    \begin{equation}
        \begin{split}
            \E \langle \truMat^{-1}, \sampMat^{-1} - \truMat^{-1} \rangle_{\mat{M}, \mat{M}} &= \E \tr(\mat{M} \truMat^{-1} \mat{M} (\sampMat^{-1} - \truMat^{-1})) \\
            &= \E \tr(\mat{M} \truMat^{-1} \mat{M} \sampMat^{-1}) - \tr(\mat{M} \truMat^{-1} \mat{M} \truMat^{-1}) \\
            &= \E \tr(\truMat^{-1/2} \mat{M} \sampMat^{-1} \mat{M} \truMat^{-1/2}) - \tr( \truMat^{-1/2} \mat{M} \truMat^{-1} \mat{M} \truMat^{-1/2}) \\
            &=  \tr(\truMat^{-1/2} \mat{M} \, \E[\sampMat^{-1}] \mat{M} \truMat^{-1/2}) - \tr( \truMat^{-1/2} \mat{M} \truMat^{-1} \mat{M} \truMat^{-1/2}) \,.
        \end{split}
    \end{equation}
    Since $\E[\sampMat^{-1}] \succeq \truMat^{-1}$ by \cref{lem:loewner}, it follows that:
    \begin{equation}
        \E \langle \truMat^{-1}, \sampMat^{-1} - \truMat^{-1} \rangle_{\mat{M}, \mat{M}} \geq 0 \,.
    \end{equation}
    Using this fact and returning to \cref{eq:wanttocompletesquares}, we obtain
    \begin{equation}
        \begin{split}
            \E \langle \augMat, \sampMat^{-1} - \truMat^{-1} \rangle_{\normMat, \autoCor} &= \E \langle \sampMat^{-1} , \sampMat^{-1} - \truMat^{-1} \rangle_{\modMatTwo^T \normMat, \autoCor \modMat^T} \\
            &\geq \E \langle \sampMat^{-1} , \sampMat^{-1} - \truMat^{-1} \rangle_{\modMatTwo^T \normMat, \autoCor \modMat^T} - \E \langle \truMat^{-1}, \sampMat^{-1} - \truMat^{-1} \rangle_{\modMatTwo^T \normMat, \autoCor \modMat^T} \\
            &= \E \langle \sampMat^{-1} - \truMat^{-1}, \sampMat^{-1} - \truMat^{-1} \rangle_{\modMatTwo^T \normMat, \autoCor \modMat^T} \\
            &= \err{\sampMat^{-1}}{\modMat^T \normMat, \autoCor \modMatTwo^T} \,.
        \end{split}
    \end{equation}
    Similarly, an expansion of the term $\E \| \augMat \|_{\normMat, \autoCor}^2$ gives:
    \begin{equation}
        \begin{split}
            \E \| \augMat \|_{\normMat, \autoCor}^2 &= \E \tr(\autoCor^{1/2} \modMatTwo^T \sampMat^{-1} \modMat^T \normMat \modMat \sampMat^{-1} \modMatTwo \autoCor^{1/2}) \\
            &= \E\|\sampMat^{-1}\|^2_{\modMat^T \normMat \modMat, \modMatTwo \autoCor \modMatTwo^T}
        \end{split}
    \end{equation}
    For a bound in the opposite direction, we simply invoke Cauchy-Schwartz:
    \begin{equation}
        \begin{split}
            \E \langle \augMat, \sampMat^{-1} - \truMat^{-1} \rangle_{\normMat, \autoCor} &\leq \sqrt{\E \|\augMat\|_{\normMat, \autoCor}^2 \, \E \|\sampMat^{-1} - \truMat^{-1}\|_{\normMat, \autoCor}^2} \\
            &= \sqrt{\E\|\sampMat^{-1}\|^2_{\modMat^T \normMat \modMat, \modMatTwo \autoCor \modMatTwo^T} \, \err{\sampMat^{-1}}{\normMat, \autoCor} }
        \end{split}
    \end{equation}
    Therefore, the desired result follows immediately from \cref{eq:itfollows1} and \cref{eq:itfollows2}. 
\end{proof}

A particularly nice corollary of this theorem comes from specializing to the Frobenius norm:

\begin{corollary}
    Under the assumptions in \cref{sec:assumptions}, consider operator shifting in the Frobenius norm $\|\cdot\|_F$. The operator shift $\augMat =\sampMat^{-1}$ has an optimal shift factor that satisfies:
    \begin{equation}
        1 \geq \sqrt{\frac{\err{\sampMat^{-1}}{F}}{\E\|\sampMat^{-1}\|^2_{F}}}  \geq \optAugFac \geq \frac{\err{\sampMat^{-1}}{F}}{\E\|\sampMat^{-1}\|^2_{F}} \geq 0 \,.
    \end{equation}
    And the corresponding optimal reduction in relative error is given by
    \begin{equation}
        \max_{\augFac \in \R} \frac{\err{\sampMat^{-1}}{F} - \err{\sampMat^{-1} - \augFac \augMat}{F}}{\err{\sampMat^{-1}}{F}} \geq \frac{\err{\sampMat^{-1}}{F}}{\E\|\sampMat^{-1}\|^2_{F}} \,,
    \end{equation}
    where $\err{\hat{\mat{X}}}{F}$ is the mean squared error of matrix estimator $\hat{\mat{X}}$ in the Frobenius norm.
\end{corollary}

This theorem tells us that if we approximate a good shift factor that comes close to $\optAugFac$ we should expect a reduction in error that is proportional to the Frobenius error relative to the average squared Frobenius norm. Moreover, it tells us roughly how large we should expect the optimal shift factor to be. If one already has a good estimate of the ratio $\err{\sampMat^{-1}}{F} / \E\|\sampMat^{-1}\|^2_{F}$, one could use this as an approximate shift factor.

Alternatively, another method to approximate $\optAugFac$ is to try to bootstrap it using synthetic samples of $\sampMat^{-1}$. Naturally, one cannot draw additional samples from the distribution $\truMatDist$; however, it is usually the case that by observing $\sampMat$, we have some ideas of the parameters that generate the distribution $\truMatDist$ and hence can draw synthetic samples from an approximate distribution $\truMatDist'$ that can be used to build a Monte Carlo estimate for $\optAugFac$. However, we will table this discussion until later in the paper when we talk about algorithmic implementations of operator shifting. For now, let us focus primarily on theoretical results.

\section{Operator Shifting in the Energy Norm} \label{sec:osenergynorm}

The previous section represents a class of operator shifts that one might use when the norm $\normMat$ is actually known; however, for many elliptic problems, the norm defined by the true matrix $\truMat$ itself is an important error norm. For example, in many physical problems, $\vc{x}^T \truMat \vc{x}$ measures the energy of a state $\vc{x}$ and hence can be even more important as a metric than $L^2$. Moreover, the case of $\truMat$ is special because the optimal shift factor reads 
\begin{equation}
    \begin{split}
    \optAugFac &= \frac{\E \langle \augMat, \sampMat^{-1} - \truMat^{-1} \rangle_{\truMat, \autoCor}}{\E \| \augMat \|_{\truMat, \autoCor}^2 } \,,
    \end{split}
\end{equation}
and hence the $\truMat^{-1}$ in the numerator will cancel with the $\truMat$ in the $\langle \cdot, \cdot \rangle_{\truMat, \autoCor}$-inner product.

This means that the possible operator shifts we can make and the conditions they must satisfy are slightly different. Indeed, for the energy norm, we consider only shifts of the form
\begin{equation}
    \augMat = \sampMat^{-1} \modMat \,,
\end{equation}
where $\modMat$ satisfies the compatibility conditions:
\begin{equation} \label{eq:compatConditions2}
    ( \autoCor \modMat^T ) = ( \autoCor \modMat^T)^T, \qquad \autoCor \modMat^T \succeq \mat{0} \,.
\end{equation}
This type of shift gives the following theorem:

\begin{theorem} \label{thm:fullOpAugEnergy}
    Under the assumptions in \cref{sec:assumptions}, consider operator shifting in energy norm $\|\cdot\|_{\truMat, \autoCor}$. Any operator shift $\augMat = \sampMat^{-1} \modMat$ such that $\modMat$ satisfies the compatibility conditions \cref{eq:compatConditions2} has an optimal shift factor that satisfies:
    \begin{equation}
        1 \geq \sqrt{\frac{\err{\sampMat^{-1}}{\truMat, \autoCor}}{\E\|\sampMat^{-1}\|^2_{\truMat, \modMat^T \autoCor \modMat}}}  \geq \optAugFac \geq \frac{\err{\sampMat^{-1}}{\truMat, \autoCor \modMat^T}}{\E\|\sampMat^{-1}\|^2_{\truMat, \modMat^T \autoCor \modMat}} \geq 0 \,.
    \end{equation}
    And the corresponding optimal reduction in relative error is given by
    \begin{equation}
        \max_{\augFac \in \R} \frac{\err{\sampMat^{-1}}{\truMat, \autoCor} - \err{\sampMat^{-1} - \augFac \augMat}{\truMat, \autoCor}}{\err{\sampMat^{-1}}{\truMat, \autoCor}} \geq \frac{\err{\sampMat^{-1}}{\truMat, \autoCor \modMat^T}^2}{\E\|\sampMat^{-1}\|^2_{\truMat, \modMat^T \autoCor \modMat} \, \err{\sampMat^{-1}}{\truMat, \autoCor}}
    \end{equation}
    where $\err{\sampMat^{-1}}{\truMat, \autoCor}$ is the mean squared error of matrix estimator $\hat{\mat{X}}$ in the $\|\cdot\|_{\truMat, \autoCor}$-norm.
\end{theorem}

\begin{proof}
This proof is more or less a carbon copy of the proof of \cref{thm:fullOpAug_revision}. The only difference is when lower bounding
\begin{equation}
    \begin{split}
    \E \langle \augMat, \sampMat^{-1} - \truMat^{-1} \rangle_{\truMat, \autoCor} &= \E \langle \sampMat^{-1} \modMat, \sampMat^{-1} - \truMat^{-1} \rangle_{\truMat, \autoCor} \\
        &= \E \langle \sampMat^{-1} , \sampMat^{-1} - \truMat^{-1} \rangle_{\truMat, \autoCor \modMat^T} 
    \end{split}
\end{equation}
The crucial inequality we need to complete the square like in the previous proof is
\begin{equation}
    \E \langle \truMat^{-1}, \sampMat^{-1} - \truMat^{-1} \rangle_{\truMat, \autoCor \modMat^T} \geq 0 \,.
\end{equation}
expanding the quantity on the left hand side
\begin{equation}
    \begin{split}
        &\E \langle \truMat^{-1}, \sampMat^{-1} - \truMat^{-1} \rangle_{\truMat, \autoCor \modMat^T} \\
        &\qquad = \E \tr((\autoCor \modMat^T)^{1/2} \sampMat^{-1} (\autoCor \modMat^T)^{1/2}) - \tr((\autoCor \modMat^T)^{1/2} \truMat^{-1} (\autoCor \modMat^T)^{1/2}) \\
        &\qquad = \tr((\autoCor \modMat^T)^{1/2} \, \E [\sampMat^{-1}] (\autoCor \modMat^T)^{1/2}) - \tr((\autoCor \modMat^T)^{1/2} \truMat^{-1} (\autoCor \modMat^T)^{1/2}) \geq 0 \,.
    \end{split}
\end{equation}
Thus, the result follows as in \cref{thm:fullOpAug_revision}.
\end{proof}

Specializing the above theorem to the case where $\modMat = \id$, we obtain a particularly nice corollary,

\begin{corollary}
    Under the assumptions in \cref{sec:assumptions}, consider operator shifting in energy norm $\|\cdot\|_{\truMat, \autoCor}$. The operator shift $\augMat = \sampMat^{-1}$ has an optimal shift factor that satisfies:
    \begin{equation}
        1 \geq \sqrt{\frac{\err{\sampMat^{-1}}{\truMat, \autoCor}}{\E\|\sampMat^{-1}\|^2_{\truMat, \autoCor}}}  \geq \optAugFac \geq \frac{\err{\sampMat^{-1}}{\truMat, \autoCor}}{\E\|\sampMat^{-1}\|^2_{\truMat, \autoCor}} \geq 0 \,.
    \end{equation}
    And the corresponding optimal reduction in relative error is given by
    \begin{equation}
        \max_{\augFac \in \R} \frac{\err{\sampMat^{-1}}{\truMat, \autoCor} - \err{\sampMat^{-1} - \augFac \augMat}{\truMat, \autoCor}}{\err{\sampMat^{-1}}{\truMat, \autoCor}} \geq \frac{\err{\sampMat^{-1}}{\truMat, \autoCor}}{\E\|\sampMat^{-1}\|^2_{\truMat, \autoCor}}
    \end{equation}
    where $\err{\sampMat^{-1}}{\truMat, \autoCor}$ is the mean squared error of matrix estimator $\hat{\mat{X}}$ in the $\|\cdot\|_{\truMat, \autoCor}$-norm.
\end{corollary}

\section{Bootstrap Formalism}

To be able to approximate the optimal shift factor $\optAugFac$ using Bootstrap Monte Carlo and write down a final algorithm for the operator shifting ideas presented above, we must first establish a formalism that allows one to generate synthetic samples of $\sampMat^{-1}$.

To build the formalism, we assume that there exists an underlying parameter space $\paramSpace$ (with sigma algebra $\Sigma$), where the parameters $\params \in \paramSpace$ contain a description of the system that produces the matrices above (e.g., $\params$ may be measurements of a scattering background, edge weights, vertex positions, etc.). We suppose the relationship between parameters and matrices is given by a measurable map 
\begin{equation}
    \paramToMat : \paramSpace \longrightarrow \pdcone \,.
\end{equation}
For example, $\params \in \paramSpace$ may be a weighted graph, and $\paramToMat(\params) \in \pdcone$ may denote a minor of its Laplacian. We suppose that there exist some \emph{unobserved} true system parameters $\truParam \in \paramSpace$ that produce the true matrix $\truMat = \paramToMat(\truParam)$. We also suppose that there exists a known family of distributions $\paramDist{\params}$ over $\paramSpace$ indexed by $\params \in \paramSpace$ that describes the observed randomness in the system if $\params$ were to be the true system parameters. It is this the relationship between $\truParam$ and the distribution $\paramDist{\truParam}$ that we assume is known as part of the model (but not the true system parameters $\truParam$ themselves). Once this family has been specified, the distribution of $\sampMat$ is given by $\paramToMat_\# \paramDist{\truParam}$, where $\paramToMat_\#$ denotes the pushforward. We define $\matDist \equiv \paramToMat_\# \paramDist{\truParam}$. Note that $\truMatDist$ as used before and after this section refers to the distribution $\matDist$.

To frame the full problem, we assume that we are given a single sample $\sampParam$ from $\paramDist{\truParam}$ with corresponding matrix $\sampMat = \paramToMat(\sampParam)$ and we would like to use operator shifting to obtain a more accurate estimate of the inverse operator $\truMat = \paramToMat(\truParam)$. This, of course, necessitates estimating the optimal shift factor,
\begin{equation} \label{eq:optaugfacexpr2}
    \optAugFac = \frac{\E_{\sampMat \sim \truMatDist} \langle \augMat, \sampMat^{-1} - \truMat^{-1} \rangle_{\normMat, \autoCor}}{\E_{\sampMat \sim \truMatDist} \| \augMat \|_{\normMat, \autoCor}^2 } \,.
\end{equation}
Naturally, it is not possible for us to estimate this quantity directly with Monte Carlo, as we do not know the true parameters $\truParam$ and hence cannot draw synthetic samples from $\truMatDist$.

However, while $\truMatDist = \matDist$ is unknown, we assume that the family of distributions $\paramDist{\params}$ itself is known --- that is, given a $\params$, we can sample synthetic data from the distribution $\paramDist{\params}$. This means that to approximate the optimal shift factor, we can try to approximate $\optAugFac$ by drawing approximate Monte Carlo samples from the approximate distribution $\paramDist{\sampParam}$. We will give all the details of this algorithm in the next section.

\section{Estimating the Optimal Shift Factor} \label{sec:genaugalg}

To convert the above into a general algorithm, we need to first do two things. The first is to convert $\optAugFac$ into a form that is more amenable to Monte Carlo evaluation. Obviously, computing the trace of a $\dim \times \dim$ matrix is too expensive in most settings, therefore, we evaluate traces by using the probabilistic form of the trace, i.e., if $\autoCor \in \pdcone$, then
\begin{equation} \label{eq:trprob}
    \tr(\autoCor^{1/2} \mat{X} \autoCor^{1/2}) = \E_{\sampleTrVec}[\sampleTrVec^T \mat{X} \sampleTrVec],
\end{equation}
where $\sampleTrVec$ is sampled from any distribution with second moment matrix $\autoCor$. We will use the notation that $\langle\cdot, \cdot \rangle_{\normMat}$ for $\normMat \in \pdcone$ denotes the $\normMat$ vector norm,
\begin{equation}
    \langle \vc{x}, \vc{y} \rangle_{\normMat} \equiv \vc{x}^T \normMat \vc{y} \,.
\end{equation}
With \cref{eq:trprob}, we can evaluate matrix inner products in the $\langle \cdot, \cdot \rangle_{\normMat, \autoCor}$ by using expectations of the corresponding $\langle\cdot, \cdot\rangle_{\normMat}$ vector norm,
\begin{equation}
    \langle \mat{X}, \mat{Y} \rangle_{\normMat, \autoCor} = \E_{\sampleTrVec \sim \mathcal{N}(\vc{0}, \autoCor)} \langle \mat{X} \sampleTrVec, \mat{Y} \sampleTrVec \rangle_{\normMat} \,.
\end{equation}

With this, we can rewrite the expression \cref{eq:optaugfacexpr2} as
\begin{equation} 
    \optAugFac = \frac{\E_{\sampMat \sim \truMatDist, \sampleTrVec \sim \mathcal{N}(\vc{0}, \autoCor)} \langle \augMat \sampleTrVec, (\sampMat^{-1} - \truMat^{-1}) \sampleTrVec \rangle_{\normMat}}{\E_{\sampMat \sim \truMatDist, \sampleTrVec \sim \mathcal{N}(\vc{0}, \autoCor)} \| \augMat \sampleTrVec \|_{\normMat}^2 } \,,
\end{equation}
where the normal distribution $\mathcal{N}(\vc{0}, \autoCor)$ can always be substituted for any other distribution with the same second moment. We note that the above quantity is impossible to compute outright because we do not know the ground truth $\truMat$ or the distribution $\matDist$. To work around this limitation, we approximate $\optAugFac$ by bootstrapping the above quantity with observed data $\sampMat$, and replacing $\truMat$ with an observed $\sampMat$ and the distribution $\matDist$ with $\bootstrapDist$. This nets us the approximation
\begin{equation}
    \bootstrapAugFac(\sampMat) = \frac{\E_{\bootstrapSampMat \sim \bootstrapDist, \sampleTrVec \sim \mathcal{N}(\vc{0}, \autoCor)} \langle \augMat(\bootstrapSampMat) \sampleTrVec, (\bootstrapSampMat^{-1} - \sampMat^{-1}) \sampleTrVec \rangle_{\normMat}}{\E_{\sampMat \sim \bootstrapDist, \sampleTrVec \sim \mathcal{N}(\vc{0}, \autoCor)} \| \augMat(\bootstrapSampMat) \sampleTrVec \|_{\normMat}^2 } \,,
\end{equation}
where $\bootstrapSampMat$ denotes a bootstrapped sample from the distribution $\bootstrapDist$. Since bootstrapping tends to work well when estimating scalar quantities, we believe that this approximation step is justified. Now, the above can be estimated with Monte Carlo,
\begin{equation} \label{eq:montecarloaugfac}
    \bootstrapMonteCarloAugFac(\sampMat) = \frac{\sum_{i = 0}^M \langle \augMat(\sampBootstrapSampMat{i}) \sampSampleTrVec{i}, (\sampBootstrapSampMat{i}^{-1} - \sampMat^{-1}) \sampSampleTrVec{i} \rangle_{\normMat}}{\sum_{i = 0}^M \| \augMat(\sampBootstrapSampMat{i}) \sampSampleTrVec{i} \|_{\normMat}^2 } \,,
\end{equation}
where
\begin{equation}
\begin{aligned}[c]
\sampBootstrapSampMat{1}, ..., \sampBootstrapSampMat{M} &\sim \bootstrapDist\\
\sampSampleTrVec{1}, ..., \sampSampleTrVec{M} &\sim \mathcal{N}(0, \autoCor)
\end{aligned}
\qquad
\begin{aligned}[c]
&\text{ i.i.d.}, \\
&\text{ i.i.d.}
\end{aligned}
\end{equation}
This gives us our general purpose operator shifting algorithm, which we give in full detail in \cref{alg:genSemiBayesian}.

\begin{algorithm}[t] 
\caption{Operator Shifting (\textbf{GS})} \label{alg:two_sided_operator_augmentation}
\hspace*{\algorithmicindent} \textbf{Input}: A right hand side $\rhs$, an operator sample $\sampMat \sim \matDist$ with corresponding parameters $\sampParam \in \paramSpace$, a choice of second moment matrix $\autoCor$, a choice of norm $\normMat$, sample count $M$. \\
\hspace*{\algorithmicindent} \textbf{Output}: An estimate $\impSol$ of $\truMat^{-1} \rhs$. \\
\begin{algorithmic}[1] \label{alg:genSemiBayesian}
\State Draw $M$ i.i.d. bootstrap samples $\sampBootstrapSampMat{1}, ..., \sampBootstrapSampMat{M} \sim \bootstrapDist$.
\State Draw $M$ i.i.d. bootstrap samples $\sampSampleTrVec{1}, ..., \sampSampleTrVec{M} \sim \mathcal{N}(\vc{0}, \autoCor)$.
\State Assign
\begin{equation*}
    \bootstrapMonteCarloAugFac(\sampMat) = \frac{\sum_{i = 1}^M \langle \augMat(\sampBootstrapSampMat{i}) \sampSampleTrVec{i}, (\sampBootstrapSampMat{i}^{-1} - \sampMat^{-1}) \sampSampleTrVec{i} \rangle_{\normMat}}{\sum_{i = 1}^M \| \augMat(\sampBootstrapSampMat{i}) \sampSampleTrVec{i} \|_{\normMat}^2 } \,,
\end{equation*}
\State Assign $\impSol \gets (\sampMat^{-1} - \bootstrapMonteCarloAugFac \autoCor \sampMat^{-1} \normMat) \rhs$
\State Return $\impSol$.
\end{algorithmic}
\end{algorithm}

\section{Efficient Estimation using Truncated Expansions} \label{sec:energyNorm}

The reader will note that an implementation of operator shifting will involve applying a different $M$ Monte Carlo samples in \cref{eq:montecarloaugfac}. Naturally, this can be quite expensive for very large operators. Hence, in this section we turn to the problem of making Monte Carlo samples more efficient. Fortunately, the energy norm has a number of properties that make it particularly attractive when it comes to efficient computations. In particular, under certain assumptions on the distribution of the randomness in $\sampMat$, we will prove that $\augFac$ can be approximated effectively by using a modified $2k$-th order Taylor expansion for $\sampMat^{-1}$. This means that one can perform Monte-Carlo  computation of $\augFac$ effectively without needing to invert a full linear system for each sample.

We will operate in the framework of \cref{sec:osenergynorm}, but specialize our discussion to the operator shift given by
\begin{equation}
    \augMat = \sampMat^{-1} \,,
\end{equation}
Repeating the computation done in the previous two sections, we have that the optimal shift factor is given by
\begin{equation}
    \begin{split}
    \optAugFac &= \frac{\E \langle \sampMat^{-1}, \sampMat^{-1} - \truMat^{-1} \rangle_{\truMat, \autoCor}}{\E \| \sampMat^{-1} \|_{\truMat, \autoCor}^2 } \,,
    \end{split}
\end{equation}
For brevity of notation, we introduce a shorthand for the expected $\autoCor$-modulated trace,
\begin{equation}
    \langle \hat{\mat{X}} \rangle_{\autoCor} = \E \tr (\autoCor^{1/2} \hat{\mat{X}} \autoCor^{1/2}) \,.
\end{equation}
With this notation, we have:
\begin{equation} \label{eq:modulatedtracebetastar}
    \optAugFac = \frac{\langle \sampMat^{-1} \truMat \sampMat^{-1} \rangle_{\autoCor} - \langle \sampMat^{-1} \rangle_{\autoCor}}{\langle \sampMat^{-1} \truMat \sampMat^{-1} \rangle_{\autoCor} } \,,
\end{equation}
The properties that makes this setting amenable for computation are related to the Taylor series of the numerator and denominator of the above expression. To demonstrate, we can expand the numerator and denominator term using the Taylor expansion of $\sampMat^{-1}$ about base-point $\truMat^{-1}$,
\begin{equation}
\begin{split}
        \sampMat^{-1} &\sim \truMat^{-1} - \truMat^{-1} \errMat \truMat^{-1} + \truMat^{-1} \errMat \truMat^{-1} \errMat \truMat^{-1} -  \truMat^{-1} \errMat \truMat^{-1} \errMat \truMat^{-1} \errMat \truMat^{-1} + ... \\
    &= \truMat^{-1/2} \left[ \sum_{k = 0}^\infty (-\truMat^{-1/2} \errMat \truMat^{-1/2})^k \right] \truMat^{-1/2}  \,.
\end{split}
\end{equation}
However, note that for this infinite Taylor series to converge, one must restrict the domain of $\sampMat$. Just like in the single variable case, the Taylor series only converges absolutely on the event $\{ \sampMat \prec 2 \truMat \}$. We prove this in a lemma,
\begin{restatable}{lemma}{convergencelemma} \label{lem:regionofconvergence}
Let $\hat{\mat{X}} \in \pdcone$ be a random matrix such that $\E [\hat{\mat{X}}^{-2}]$ exists and $\hat{\mat{X}} \preceq (2 - \varepsilon) \mat{Y}$ almost surely for $\mat{Y} \in \pdcone$ and $\varepsilon > 0$. Consider the infinite Taylor series for $\hat{\mat{X}}^{-1}$ and $\hat{\mat{X}}^{-2}$ respectively about base-point $\mat{Y}$, i.e.,
\begin{equation}
\begin{split}
    \hat{\mat{X}}^{-1} &\sim \mat{Y}^{-1/2} \left[ \sum_{k = 0}^\infty (-\mat{Y}^{-1/2} (\hat{\mat{X}} - \mat{Y}) \mat{Y}^{-1/2})^k \right] \mat{Y}^{-1/2} \,, \\
    \hat{\mat{X}}^{-2} &\sim \mat{Y}^{-1/2} \left[ \sum_{k = 0}^\infty (k + 1) (-\mat{Y}^{-1/2} (\hat{\mat{X}} - \mat{Y}) \mat{Y}^{-1/2})^k \right] \mat{Y}^{-1/2} \,.
\end{split}
\end{equation}
Both series converge in mean-squared Frobenius norm to their respective limits.
\end{restatable}
A proof of this fact is relegated to the appendix. This places a damper on our ability to use the Taylor expansion of $\sampMat^{-1}$ with impunity over all of $\pdcone$. For simplicity, however, we will assume for now that the true distribution $\truMatDist$ is supported on the event $\{ \sampMat \prec (2 - \varepsilon) \truMat \}$. It turns out, as we will show in \cref{sec:variableBasepoint}, that one can remove this assumption by instead expanding about a variable base-point $\basepointfactorFunc \truMat$ for some large enough factor $\basepointfactorFunc \in \R$.

Therefore, when we have $\supp(\truMatDist) \subset \{ \sampMat \prec (2 - \varepsilon)\truMat \}$, we can expand $\langle \sampMat^{-1} \rangle_{\autoCor}$,
\begin{equation} \label{eq:subfactor}
    \begin{split}
        \langle \sampMat^{-1} \rangle_{\autoCor} &= \left \langle \truMat^{-1/2}  \left( \sum_{k = 0}^\infty (-\truMat^{-1/2} \errMat \truMat^{-1/2})^k \right) \truMat^{-1/2} \right \rangle_{\autoCor} \\
        &= \sum_{k = 0}^\infty \left \langle \truMat^{-1/2} (-\truMat^{-1/2} \errMat \truMat^{-1/2})^k \truMat^{-1/2} \right \rangle_{\autoCor} \\
        &= \sum_{k = 0}^\infty \left \langle (-\truMat^{-1/2} \errMat \truMat^{-1/2})^k \right \rangle_{\truMat^{-1/2} \autoCor \truMat^{-1/2}} \\
        &= \sum_{k = 0}^\infty \langle \xMat^k \rangle_{\sMat} \,,
    \end{split} 
\end{equation}
where we have defined 
\begin{equation}
    \xMat \equiv - \truMat^{-1/2} \errMat \truMat^{-1/2}\,, \qquad \sMat \equiv \truMat^{-1/2} \autoCor \truMat^{-1/2} \,.
\end{equation}
Note quickly that the assumption that $\E[\sampMat] \preceq \truMat$ implies $\E[\xMat] \succeq 0$. The assumption that $\mat{0} \prec \sampMat \prec (2 - \varepsilon) \truMat$ gives us that
\begin{equation}
    -(1 - \varepsilon) \id \prec \xMat \prec \id \,.
\end{equation}
From line one to two in \cref{eq:subfactor}, we may interchange the $\langle \cdot \rangle_{\autoCor}$ operator and the infinite sum by virtue of the fact that $\langle \cdot \rangle_{\autoCor}$ is continuous with respect to the mean squared Frobenius norm,
\begin{equation}
    \langle \hat{\mat{X}} \rangle_{\autoCor} = \E \tr(\autoCor^{1/2} \hat{\mat{X}} \autoCor^{1/2}) = \E \tr(\autoCor \hat{\mat{X}}) \leq \| \autoCor\|_F \, \sqrt{ \E \|\hat{\mat{X}} \|_F^2 } \,,
\end{equation}
where the inequality above is by Cauchy-Schwartz.

We can similarly expand $\langle \sampMat^{-1} \truMat \sampMat^{-1} \rangle_{\autoCor}$,
\begin{equation}
\begin{split}
    &\langle \sampMat^{-1} \truMat \sampMat^{-1} \rangle_{\autoCor} \\
    &= \left \langle \truMat^{-1/2}  \left( \sum_{k = 0}^\infty (-\truMat^{-1/2} \errMat \truMat^{-1/2})^k\ \right) \left( \sum_{k = 0}^\infty (-\truMat^{-1/2} \errMat \truMat^{-1/2})^k\ \right) \truMat^{-1/2} \right \rangle_{\autoCor} \\
    &= \left \langle \truMat^{-1/2} \left(  \sum_{k = 0}^\infty (k + 1) (- \truMat^{-1/2} \errMat \truMat^{-1/2})^k \right) \truMat^{-1/2} \right \rangle_{\autoCor} \\
    &= \sum_{k = 0}^\infty (k + 1) \left \langle \truMat^{-1/2} (- \truMat^{-1/2} \errMat \truMat^{-1/2})^k \truMat^{-1/2} \right \rangle_{\autoCor} \,, \\
    &= \sum_{k = 0}^\infty (k + 1) \langle \xMat^{k} \rangle_{\sMat} 
\end{split}
\end{equation}
where on line two to three we have used the property that $\left(\sum_k z^k\right) \left(\sum_k z^k \right) \sim \sum_k (k + 1) z^k$. \cref{lem:regionofconvergence} tells us the above series converges in the mean Frobenius norm and the fact that $\langle \cdot \rangle_{\sMat}$ is continuous with respect to the expected squared Frobenius norm lets us interchange summation and the $\langle \cdot \rangle_{\sMat}$ operator.

Thus, plugging everything into \cref{eq:modulatedtracebetastar}, we obtain that
\begin{equation} \label{eq:normA_optBeta}
    \optAugFac = \frac{\sum_{k = 0}^\infty k \, \langle \xMat^k \rangle_{\sMat}}{\sum_{k = 0}^\infty (k + 1) \, \langle \xMat^k \rangle_{\sMat}}
\end{equation}

The form \cref{eq:normA_optBeta} suggests a possible way of avoiding the need to invert a linear system for every Monte Carlo sample involved in approximating $\optAugFac$. Instead of attempting to approximate the quantity $\optAugFac$ directly, one can truncate the series in \cref{eq:normA_optBeta} with an appropriate windowing function to obtain a series of truncated shift factors, defined as
\begin{equation} \label{eq:truncFacDef}
    \truncFac{N} \equiv \frac{\sum_{k = 0}^\infty \windowFunc{N}{k} \, \langle \xMat^k \rangle_{\sMat}}{\sum_{k = 0}^\infty \windowFuncDenom{N}{k}  \, \langle \xMat^k \rangle_{\sMat}} \,,
\end{equation}
where $\windowFunc{N}{k} : \mathbb{Z}_{\geq 0} \longrightarrow \R$ and $\windowFuncDenom{N}{k} : \mathbb{Z}_{\geq 0} \longrightarrow \R$ are two appropriately defined collections of discrete windowing functions, each with \emph{bounded support}, such that the collection has the property that $\windowFunc{N}{k} \rightarrow k$ and $\windowFuncDenom{N}{k} \rightarrow k + 1$ as $N \rightarrow \infty$.
It turns out, as we will discuss in the next section, that regardless of the randomness structure of the distribution $\truMatDist$ (as long as it is bounded), one can choose an appropriate series of windowing functions $\windowFunc{N}{k}, \windowFuncDenom{N}{k}$ such that
\begin{equation}
    0 \leq \truncFac{1} \leq \truncFac{2} \leq ... \leq \truncFac{N} \leq ... \leq \optAugFac \leq 1 \,,
\end{equation}
which means that using any of the truncated shift factors $\truncFac{N}$ underestimates the value of $\optAugFac$ and hence still decreases the value of the objective $\err{(1 - \augFac) \sampMat^{-1}}{\truMat, \autoCor}$ from its base value of $\err{\sampMat^{-1}}{\truMat, \autoCor}$, i.e.,
\begin{equation}
    \err{\sampMat^{-1}}{\truMat, \autoCor} \geq \err{(1 - \truncFac{1}) \sampMat^{-1}}{\truMat, \autoCor} \geq ... \geq \err{(1 - \truncFac{N}) \sampMat^{-1}}{\truMat, \autoCor} \geq ... \geq \err{(1 - \optAugFac) \sampMat^{-1}}{\truMat, \autoCor} \,.
\end{equation}

Before we continue, note that one can rewrite the truncated shift factors $\truncFac{N}$ in a form more amenable for computation, namely
\begin{equation}
    \truncFac{N} =  \frac{\E\left[\sum_{k = 0}^\infty \windowFunc{N}{k} \, \sampleTrVec^T \truMat^{-1} (\errMat \truMat^{-1})^k \sampleTrVec \right]}{\E\left[\sum_{k = 0}^\infty \windowFuncDenom{N}{k} \, \sampleTrVec^T \truMat^{-1} (\errMat \truMat^{-1})^k \sampleTrVec \right]} \,,
\end{equation}
where $\sampleTrVec$ is sampled from a distribution with second moment matrix $\autoCor$ (perhaps $\mathcal{N}(\mat{0}, \autoCor)$), and is independent from $\sampMat \sim \truMatDist$. 

\subsection{Monotonic Estimates of the Shift Factor} \label{sec:taylorApprox}
Our analyses of the monotonicity of the $\truncFac{N}$ relies upon the following lemma,
\begin{restatable}{lemma}{monotonesequencelemma} \label{lem:monotonic}
Let $a_1, a_2, ..., a_k, ... \in \R_{\geq 0}$ and $b_1, b_2 ..., b_k, ... \in \R_{\geq 0}$ be two sequences of nonnegative real numbers with $b_1 > 0$, and consider the truncated sum ratios
\begin{equation} \label{eq:truncFacRatio}
    \truncFac{N} \equiv \frac{\sum_{k = 1}^N a_k }{\sum_{k = 1}^{N} b_k} \,,
\end{equation}
then, if it is the case that
\begin{equation} \label{eq:monotonicRequirement}
\frac{a_k}{b_k} \geq \frac{a_{k - 1}}{b_{k - 1}} \,,    
\end{equation} 
for all $k$ (e.g., the ratios $a_k / b_k$ are monotonically increasing), then the sequence $\beta_1, \beta_2, ..., \beta_k, ...$ is monotonically increasing.
\end{restatable}

To construct the discrete windowing functions $\windowFunc{N}{k}, \windowFuncDenom{N}{k}$, it is instructive to think of the generating polynomials corresponding to $\windowFunc{N}{k}, \windowFuncDenom{N}{k}$, i.e.,
\begin{equation}
    \genWindowPoly{N}(x) \equiv \sum_{k = 0}^\infty \windowFunc{N}{k} \, x^k, \qquad \genWindowPolyDenom{N}(x) \equiv \sum_{k = 0}^\infty \windowFuncDenom{N}{k} \, x^k \,.
\end{equation}
We can rewrite \cref{eq:truncFacDef} as
\begin{equation} \label{eq:truncFacSimple}
    \truncFac{N} = \frac{\langle \genWindowPoly{N}(\xMat) \rangle_{\sMat}}{ \langle\genWindowPolyDenom{N}(\xMat) \rangle_{\sMat}} \,. 
\end{equation}
Note that we have used the fact that $\genWindowPoly{N}(x), \genWindowPolyDenom{N}(x)$ are polynomial generating functions of bounded degree to interchange summation and expectation.

Our intent now is to find a sequence of polynomials $\genWindowPoly{N}(x), \genWindowPolyDenom{N}(x)$ with the properties
\begin{equation} \label{eq:polyconvergence}
    \genWindowPoly{N}(x) \nearrow \sum_{k = 0}^\infty k \, x^k \,, \qquad \genWindowPolyDenom{N}(x) \nearrow \sum_{k = 0}^\infty (k + 1) \, x^k \,, \qquad \text{for } |x| < 1, \text{ as } N \rightarrow \infty \,,
\end{equation}
such that the sequence in \cref{eq:truncFacSimple} allows us to invoke \cref{lem:monotonic}. We do this by constructing $\genWindowPoly{N}(x), \genWindowPolyDenom{N}(x)$ from smaller primitive polynomials $\primmPoly{j}(x)$, $\primmPolyDenom{j}(x)$ such that
\begin{gather}
    \genWindowPoly{N}(x) = \sum_{j = 0}^N \primmPoly{j}(x) \,, \qquad \genWindowPolyDenom{N}(x) = \sum_{j = 0}^N \primmPolyDenom{j}(x) \,, \label{eq:polydef} \\
        \E\left[\primmPoly{j}(\xMat)\right] \succeq 0 \,, \qquad \E\left[\primmPolyDenom{j}(\xMat)\right] \succ 0 \,, \label{eq:primmPsd}\\
        \primmPoly{j}(x) = \frac{2j - 1}{2j} \primmPolyDenom{j}(x)\,, \qquad \text{for } j \geq 1 \,, \label{eq:primmRatio} \\
        \primmPoly{0}(x) = 0 \,. 
\end{gather}
With this, we can expand \cref{eq:truncFacSimple} into the required form of \cref{lem:monotonic},
\begin{equation}
     \truncFac{N} = \frac{\sum_{j = 0}^N \langle \primmPoly{j}(\xMat) \rangle_{\sMat} }{\sum_{j = 0}^N \langle \primmPolyDenom{j}(\xMat) \rangle_{\sMat}} \equiv \frac{\sum_{j = 0}^N a_j}{\sum_{j = 0}^N b_j} \,.
\end{equation}
Note that property \cref{eq:primmPsd} implies $a_j \geq 0$ and $b_j > 0$, and the property \cref{eq:primmRatio} implies, for $j \geq 1$,
\begin{equation}
    \frac{a_j}{b_j} = \frac{ \langle \primmPoly{j}(\xMat) \rangle_{\sMat} }{\langle \primmPolyDenom{j}(\xMat) \rangle_{\sMat}} = \frac{2j - 1}{2j} \frac{\langle \primmPolyDenom{j}(\xMat) \rangle_{\sMat}}{\langle \primmPolyDenom{j}(\xMat) \rangle_{\sMat}} = \frac{2j - 1}{2j} \,,
\end{equation}
and for $j = 0$, we have $a_0 / b_0 = 0$. Hence, the ratio $a_j / b_j$ is monotonically increasing in $j$ and hence satisfies the requirement \cref{eq:monotonicRequirement} of \cref{lem:monotonic}. Therefore, the existence of such primitive polynomials $\primmPoly{N}(x)$, $\primmPolyDenom{N}(x)$ immediately implies that
\begin{gather}
    \truncFac{N} \rightarrow \optAugFac \text{ as } N \rightarrow \infty\,, \label{eq:truncFacConvergence} \\
    0 \leq \truncFac{1} \leq \truncFac{2} \leq \truncFac{3} \leq ... \leq \truncFac{N} \leq ... \leq \optAugFac \leq 1 \,,
\end{gather}
where \cref{eq:truncFacConvergence} follows from \cref{eq:polyconvergence}; the fact that $\optAugFac \leq 1$ follows from from $\truncFac{N} \rightarrow \optAugFac$ and the fact that $a_j \leq b_j$, and hence the numerator of $\sum_{j = 0}^N a_j / \sum_{j = 0}^N b_j$ is always bounded by the denominator, implying $\truncFac{N} \leq 1$ for all $N$; and the fact that $\truncFac{N} \geq 0$ for any $N$ comes from non-negativity of the numerator and denominator of $\truncFac{N}$.

To show that such primitive polynomials $\primmPoly{N}(x)$ and $\primmPolyDenom{N}(x)$ actually exist, we consider the following definition,
\begin{align}
    \primmPoly{0}(x) &\equiv 0 \,, \qquad \primmPolyDenom{0}(x) \equiv 1 \,, \\
    \primmPoly{1}(x) &\equiv x + \frac{1}{2} x^2 \,, \qquad \primmPolyDenom{1}(x) \equiv 2x + x^2 \,, \\
    \primmPoly{j}(x) &\equiv (2 j - 1) \left(\frac{1}{2} x^{2j - 2} + x^{2j - 1} + \frac{1}{2} x^{2j}\right) \,, \qquad \text{for } k \geq 2 \,, \\
   \primmPolyDenom{j}(x) &\equiv 2 j \, \left(\frac{1}{2} x^{2j - 2} + x^{2j - 1} + \frac{1}{2} x^{2j}\right) \,, \qquad \text{for } k \geq 2 \,.
\end{align}
To show this family of primitive polynomials satisfies the desired properties, note that, for $j \geq 2$,
\begin{equation}
    \primmPoly{j}(x) = \frac{2 j - 1}{2} x^{2j - 2} (x + 1)^2 \geq 0 \,.
\end{equation}
This implies $\primmPoly{j}(\xMat) \succeq 0$. Moreover, we can only have $\primmPoly{j}(\xMat) = \mat{0}$ if all of the eigenvalues of $\xMat$ are either $0$ or $-1$. Note that a $-1$ eigenvalue in $\xMat$ is impossible by virtue of the fact that $-\id \prec \xMat \prec \id$. Therefore, $\primmPoly{j}(\xMat) = \mat{0}$ is only possible if $\xMat = \mat{0}$. However, this cannot be the case almost surely, as $\xMat = \mat{0}$ implies $\sampMat = \truMat$. Therefore, with probability greater than $0$, we have that $\primmPoly{j}(\xMat) \succ 0$, implying
\begin{equation}
    \E[\primmPoly{j}(\xMat)] \succ \mat{0} \,, \qquad \E[\primmPolyDenom{j}(\xMat)] \succ \mat{0} \,.
\end{equation}
Furthermore, for $j = 1$, we have
\begin{equation}
     \E[\primmPoly{1}(\xMat)] = \E[ \xMat ] + \frac{1}{2} \E[ \xMat^2 ] \succ \mat{0} \,,
\end{equation}
where we have used the fact that $\E[\xMat] \succeq 0$ (from the fact that $\E[\sampMat] \preceq \truMat$) and the fact that $\xMat^2 \succ \mat{0}$ with probability greater than $0$ (unless $\sampMat = \truMat$ a.s.).

Finally, to show \cref{eq:polyconvergence}, we simply note that, for $k$ odd, and $N$ large enough, it is the case that
\begin{equation}
\begin{split}
    [x^k] \genWindowPoly{N}(x) &= [x^k] \primmPoly{(k + 1) / 2}(x) = k \,, \\
    [x^k] \genWindowPolyDenom{N}(x) &= [x^k]\primmPolyDenom{(k + 1) / 2}(x) = k + 1 \,,
\end{split}
\end{equation}
since $\primmPoly{(k + 1) / 2}(x)$ is the only primitive polynomial with a $x^k$ term in $\genWindowPoly{N}(x)$, and likewise for $\genWindowPolyDenom{N}(x)$. For $k \geq 2$ even, we have that 
\begin{equation}
\begin{split}
    [x^k] \genWindowPoly{N}(x) &= [x^k] (\primmPoly{k / 2}(x) + \primmPoly{k/2 + 1}(x)) = \frac{k - 1}{2} + \frac{k + 1}{2} = k \,, \\
    [x^k] \genWindowPolyDenom{N}(x) &= [x^k] (\primmPolyDenom{k / 2}(x) + \primmPolyDenom{k/2 + 1}(x)) = \frac{k}{2} + \frac{k + 2}{2} = k + 1 \,.
\end{split}
\end{equation}
Thus, the polynomials $\genWindowPoly{N}(x)$ and $\genWindowPolyDenom{N}(x)$ have all the desired properties. We restate the results of the past two sections in a theorem,
\begin{theorem} \label{thm:eag}
Under the assumptions in \cref{sec:assumptions}, consider operator shifting with shift $\augMat = \sampMat$ in energy norm $\|\cdot\|_{\truMat, \autoCor}$. Suppose that the random matrix $\sampMat \in \pdcone$ satisfies $\mat{0} \prec \sampMat \prec (2 - \varepsilon) \truMat$ almost surely. Then let $\truncFac{N}$ be the truncated approximations to the optimal shift factor $\optAugFac$, i.e.,
\begin{equation}
        \truncFac{N} = \frac{\sum_{k = 0}^{2N}\windowFunc{N}{k} \langle \xMat^k \rangle_{\sMat}}{ \sum_{k = 0}^{2N} \windowFuncDenom{N}{k} \langle \xMat^k \rangle_{\sMat}} = \frac{\E\left[\sum_{k = 0}^{2N} \windowFunc{N}{k} \, \sampleTrVec^T \truMat^{-1} (\errMat \truMat^{-1})^k \sampleTrVec \right]}{\E\left[\sum_{k = 0}^{2N} \windowFuncDenom{N}{k} \, \sampleTrVec^T \truMat^{-1} (\errMat \truMat^{-1})^k \sampleTrVec \right]} \,,
\end{equation}
where $\windowFunc{N}{k}$ and $\windowFuncDenom{N}{k}$ are given by
\begin{equation} \label{eq:softtruncwindows}
    \windowFunc{N}{k} = \begin{cases}
        k & k < 2N \\
        \frac{k - 1}{2} & k = 2N \\
        0 & \text{o.w.}
    \end{cases} \,, \qquad
    \windowFuncDenom{N}{k} = \begin{cases}
        k + 1 & k < 2N \\
        \frac{k}{2} & k = 2N \\
        0 & \text{o.w.}
    \end{cases} \,.
\end{equation}
Under these assumptions, we have that
\begin{gather*}
    \truncFac{N} \nearrow \optAugFac \text{ as } N \rightarrow \infty\,, \\
    0 \leq \truncFac{1} \leq \truncFac{2} \leq \truncFac{3} \leq ... \leq \truncFac{N} \leq ... \leq \optAugFac \leq 1 \,, \\
    \err{\sampMat^{-1}}{\truMat, \autoCor} \geq \err{(1 - \truncFac{1}) \sampMat^{-1}}{\truMat, \autoCor} \geq ... \geq \err{(1 - \truncFac{N}) \sampMat^{-1}}{\truMat, \autoCor} \geq ... \geq \err{(1 - \optAugFac) \sampMat^{-1}}{\truMat, \autoCor} \,.
\end{gather*}
\end{theorem}

\subsection{Hard Windowing} \label{sec:hardwindow}

The tradeoff for monotone convergence to the true shift factor $\optAugFac$ is that the windowing functions $\windowFunc{N}{k}$ and $\windowFuncDenom{N}{k}$ presented above --- which we will refer to as \emph{soft windowing functions} --- may be too conservative at low orders. When this is the case, one may instead choose to use \emph{hard windowing functions} that perform a hard truncation of the infinite Taylor series. That is, one may choose to instead use

\begin{equation} \label{eq:hardtruncwindows}
    \windowFunc{N}{k} = \begin{cases}
        k & k \leq 2N \\
        0 & \text{o.w.}
    \end{cases} \,, \qquad
    \windowFuncDenom{N}{k} = \begin{cases}
        k + 1 & k \leq 2N \\
        0 & \text{o.w.}
    \end{cases} \,.
\end{equation}

Under the conditions of \cref{thm:eag}, this choice of windowing function will still guarantee the convergence $\truncFac{N} \rightarrow \optAugFac$. However, we lose the monotonicity guarantees of the soft windowing functions unless one makes very stringent assumptions on the underlying distribution. That being said, in practice this technique can perform quite well, as indicated in our numerical experiments in \cref{sec:numerics}. To distinguish between truncated energy norm shifting with soft and hard windows, we will use the abbreviations \textbf{ES-T-S} and \textbf{ES-T-H} for truncated energy norm augmentation with soft and hard windows respectively. 

\subsection{Quick Start} \label{sec:quickstart}

For help readers with implementation, we provide explicit formulas for the shift factor $\augFac$ for low truncation orders, as well as a pseudo-code implementation of the different variants of energy norm augmentation.

\subsubsection{Explicit Formulas for Low Orders} \label{sec:explicitloworder}

 First, we provide formulas for low orders of the algorithm presented in the previous section. In the subsequent formulas, we let
\begin{equation}
         \errMat \equiv \sampMat - \truMat, \qquad
         \sampleTrVec \sim \mathcal{N}(\mat{0}, \autoCor), \qquad \sampMat \sim \truMatDist, \qquad \sampleTrVec \independent \sampMat \,.
\end{equation}

\begin{enumerate}
    \item \textbf{ES-T-S, Order 2}:
     \begin{equation}
        \truncFac{1}^{\text{ES-T-S}} = \frac{\E\left[ \sampleTrVec^T( \frac{1}{2} \truMat^{-1} \errMat \truMat^{-1} \errMat \truMat^{-1} + \truMat^{-1} \errMat \truMat^{-1}) \sampleTrVec\right]}{\E \left[ \sampleTrVec^T ( \truMat^{-1} \errMat \truMat^{-1} \errMat \truMat^{-1} + 2\truMat^{-1} \errMat \truMat^{-1} + \truMat^{-1})  \sampleTrVec \right]} \,.
    \end{equation}
    
    \item \textbf{ES-T-S, Order 2, Mean-Zero Error}:
    
    In many cases, the error matrix $\errMat$ may be mean zero, i.e., $\E[\errMat] = \mat{0}$. When this happens, the above expression has an even simpler form,
    
    \begin{equation}
        \truncFac{1}^{\text{ES-T-S}} = \frac{1}{2} \frac{\E\left[ \sampleTrVec^T( \truMat^{-1} \errMat \truMat^{-1} \errMat \truMat^{-1}) \sampleTrVec\right]}{\E \left[ \sampleTrVec^T ( \truMat^{-1} \errMat \truMat^{-1} \errMat \truMat^{-1} + \truMat^{-1})  \sampleTrVec \right]} \,.
    \end{equation}
    
    \item \textbf{ES-T-H, Order 2}:
    
    \begin{equation}
        \truncFac{1}^{\text{ES-T-H}} = \frac{\E\left[ \sampleTrVec^T( 2 \truMat^{-1} \errMat \truMat^{-1} \errMat \truMat^{-1} + \truMat^{-1} \errMat \truMat^{-1}) \sampleTrVec\right]}{\E \left[ \sampleTrVec^T ( 3 \truMat^{-1} \errMat \truMat^{-1} \errMat \truMat^{-1} + 2\truMat^{-1} \errMat \truMat^{-1} + \truMat^{-1})  \sampleTrVec \right]} \,.
    \end{equation}

    \item \textbf{ES-T-H, Order 2, Mean-Zero Error}:
        In many cases, the error matrix $\errMat$ may be mean zero, i.e., $\E[\errMat] = \mat{0}$. When this happens, the above expression has an even simpler form,
    \begin{equation}
        \truncFac{1}^{\text{ES-T-H}} =  \frac{\E\left[ \sampleTrVec^T( 2 \truMat^{-1} \errMat \truMat^{-1} \errMat \truMat^{-1}) \sampleTrVec\right]}{\E \left[ \sampleTrVec^T ( 3 \truMat^{-1} \errMat \truMat^{-1} \errMat \truMat^{-1} + \truMat^{-1})  \sampleTrVec \right]} \,.
    \end{equation}

\end{enumerate}

\subsection{Algorithm}

We give the full meta algorithm for all favors of energy norm augmentation in \cref{alg:augMetaAlgorithm}. Note that in \cref{alg:augMetaAlgorithm}, like in \cref{alg:genSemiBayesian}, we replace expectations with bootstrapped Monte Carlo estimators. If one wants to use the simplified expressions provided above in \cref{sec:explicitloworder}, one must similarly replace the expectations with sampled and bootstrapped versions. This process is fairly straightforward, for example, for ES-T-H, Order 2, Mean-Zero Error, we get
\begin{equation}
    \bootstrapMonteCarloTruncAugFac{1}^{\text{ES-T-H}} = \frac{ \sum_{i = 0}^M \, \sampSampleTrVec{i}^T ( 2 \sampMat^{-1} (\sampBootstrapSampMat{i} - \sampMat) \sampMat^{-1} (\sampBootstrapSampMat{i} - \sampMat) \sampMat^{-1}) \sampSampleTrVec{i}}{\sum_{i = 0}^M \, \sampSampleTrVec{i}^T ( 3 \sampMat^{-1} (\sampBootstrapSampMat{i} - \sampMat) \sampMat^{-1} (\sampBootstrapSampMat{i} - \sampMat) \sampMat^{-1} + \sampMat^{-1}) \sampSampleTrVec{i}} \,,
\end{equation}
where $\sampBootstrapSampMat{i}$ and $\sampSampleTrVec{i}$ are defined as in \cref{alg:augMetaAlgorithm}.

\begin{algorithm}[t] 
\caption{Energy-Norm Operator Shfiting Meta-algorithm}
\hspace*{\algorithmicindent} \textbf{Input}: A right hand side $\rhs$, an operator sample $\sampMat \sim \matDist$ with corresponding parameters $\sampParam \in \paramSpace$, a choice of second moment matrix $\autoCor$, a choice of matrix $\modMat$ satisfying the compatibility conditions, sample count $M$. \\
\hspace*{\algorithmicindent} \textbf{Output}: An estimate $\impSol$ of $\truMat^{-1} \rhs$. \\
\begin{algorithmic}[1] \label{alg:augMetaAlgorithm}
\State Factorize/preprocess $\sampMat$ to precompute $\sampMat^{-1}$ if necessary.
\State Draw $M$ i.i.d. bootstrap samples $\sampBootstrapSampMat{1}, ..., \sampBootstrapSampMat{M} \sim \bootstrapDist$.
\State Draw $M$ i.i.d. bootstrap samples $\sampSampleTrVec{1}, ..., \sampSampleTrVec{M} \sim \mathcal{N}(\vc{0}, \autoCor)$.
\If{using Truncated Energy-Norm Shifting \textbf{(ES-T)}}
\If{using Soft Truncation \textbf{(ES-T-S)}}
\State Let
\begin{equation*}
    \windowFunc{N}{k} = \begin{cases}
        k & k < 2N \\
        \frac{k - 1}{2} & k = 2N \\
        0 & \text{o.w.}
    \end{cases} \,, \qquad
    \windowFuncDenom{N}{k} = \begin{cases}
        k + 1 & k < 2N \\
        \frac{k}{2} & k = 2N \\
        0 & \text{o.w.}
    \end{cases} \,.
\end{equation*}
\ElsIf{using Hard Truncation \textbf{(ES-T-H)}}
\State Let
\begin{equation*} 
    \windowFunc{N}{k} = \begin{cases}
        k & k \leq 2N \\
        0 & \text{o.w.}
    \end{cases} \,, \qquad
    \windowFuncDenom{N}{k} = \begin{cases}
        k + 1 & k \leq 2N \\
        0 & \text{o.w.}
    \end{cases} \,.
\end{equation*}
\EndIf
\State Assign
\begin{equation*}
    \bootstrapMonteCarloAugFac \gets \frac{\sum_{i = 0}^M \sum_{k = 0}^\infty \windowFunc{N}{k} \, \sampSampleTrVec{i}^T \modMat^T \sampMat^{-1} ((\sampBootstrapSampMat{i} - \sampMat) \sampMat^{-1})^k \sampSampleTrVec{i}}{\sum_{i = 0}^M \sum_{k = 0}^\infty \windowFuncDenom{N}{k} \, \sampSampleTrVec{i}^T \modMat^T \sampMat^{-1} ((\sampBootstrapSampMat{i} - \sampMat) \sampMat^{-1})^k \modMat \sampSampleTrVec{i}} \,,
\end{equation*}
where
\ElsIf{using Untruncated Energy-Norm Shifting \textbf{(ES)}}
\State Assign
\begin{equation*}
    \bootstrapMonteCarloAugFac \gets \frac{\sum_{i = 0}^M \, \sampSampleTrVec{i}^T \modMat^T (\sampBootstrapSampMat{i}^{-1} \sampMat  \sampBootstrapSampMat{i}^{-1} -\sampBootstrapSampMat{i}^{-1}) \sampSampleTrVec{i}}{\sum_{i = 0}^M \, \sampSampleTrVec{i}^T \modMat^T (\sampBootstrapSampMat{i}^{-1} \sampMat  \sampBootstrapSampMat{i}^{-1}) \modMat \sampSampleTrVec{i}} \,,
\end{equation*}
\EndIf
\State Clamp $\bootstrapMonteCarloAugFac \gets \max(0, \bootstrapMonteCarloAugFac)$.
\State Assign $\impSol \gets (\sampMat^{-1} - \bootstrapMonteCarloAugFac  \sampMat^{-1} \modMat) \rhs$
\State Return $\impSol$.
\end{algorithmic}
\end{algorithm}

\section{Shifted Base-point Estimation} \label{sec:variableBasepoint}

Obviously, the issue with the above theorem is that the restriction that $\supp(\truMatDist) \subset \{ \sampMat \prec (2 - \varepsilon) \truMat \}$ is quite restrictive from a problem standpoint; there are many natural problems that do not fall into this setting. Recall that this assumption comes from the fact that the infinite Taylor series for $\sampMat^{-1}$ about base-point $\truMat$ only converges when $\sampMat \prec (2- \varepsilon) \truMat$.

We address this issue with a technique we call \emph{shifted base-point estimation}. The key idea is to grow the region of convergence of the infinite Taylor series by changing the base-point of the Taylor series expansion. If we make the assumption that the distribution $\truMatDist$ is bounded, then there must exist some $\basepointfactor \geq 1$ such that $\sampMat \prec \basepointfactor \truMat$ for every $\sampMat$ in the support of $\truMatDist$.  \cref{lem:regionofconvergence} then tells us that we are justified in taking an infinite Taylor expansion about base-point $\basepointfactor \truMat$,
\begin{equation}
    \sampMat^{-1} = \truMat^{-1/2} \left[ \sum_{k = 0}^\infty \frac{1}{\basepointfactor^{k+1}} (- \truMat^{-1/2} \basepointerrMat \truMat^{-1/2})^k \right] \truMat^{-1/2} \,.
\end{equation}
where $\basepointerrMat \equiv \sampMat - \basepointfactor \truMat$. In general, the best values of $\basepointfactor$ are those that are as small as possible while maintaining that the support of $\truMatDist$ lies within $\{ \sampMat \prec \basepointfactor \truMat \}$, as the accuracy of a truncated series becomes less farther away from the base-point.

With the above, one can repeat the calculations of \cref{sec:energyNorm} practically verbatim to derive the infinite series expression for the optimal shift factor,
\begin{equation} \label{eq:normA_optBeta_new}
    \optAugFac = \frac{\sum_{k = 0}^\infty (k + 1 - \basepointfactor) \, \basepointfactor^{-k} \, \langle ( - \truMat^{-1/2} \basepointerrMat \truMat^{-1/2} )^k \rangle_{\sMat}}{\sum_{k = 0}^\infty (k + 1) \, \basepointfactor^{-k} \, \langle ( - \truMat^{-1/2} \basepointerrMat \truMat^{-1/2} )^k\rangle_{\sMat}} \,.
\end{equation}
for notational simplicity, define
\begin{equation}
    \basepointxMat \equiv \basepointfactor^{-1} (-\truMat^{-1/2} \basepointerrMat \truMat^{-1/2}) \,.
\end{equation}
Note that
\begin{equation}
    \basepointxMat = \id - \basepointfactor^{-1} \truMat^{-1/2} \sampMat \truMat^{-1/2} \,.
\end{equation}
From the fact that $0 \prec \sampMat \prec \basepointfactor \truMat$, it follows that
\begin{equation}
    0 \prec \basepointxMat \prec \id \, .
\end{equation}
Therefore, the expression for the optimal shift factor becomes
\begin{equation} 
    \optAugFac = \frac{\sum_{k = 0}^\infty (k + 1 - \basepointfactor) \, \basepointfactor^{-k} \, \langle \basepointxMat^k \rangle_{\sMat}}{\sum_{k = 0}^\infty (k + 1) \, \basepointfactor^{-k} \, \langle \basepointxMat^k\rangle_{\sMat}} \,.
\end{equation}

From here, we follow the same schema to define the truncation of the infinite series above,
\begin{equation} 
    \truncFac{N} = \frac{\sum_{k = 0}^\infty \windowFuncBP{N}{k} \, \basepointfactor^{-k} \, \langle \basepointxMat^k \rangle_{\sMat}}{\sum_{k = 0}^\infty  \windowFuncBPDenom{N}{k} \, \basepointfactor^{-k} \, \langle \basepointxMat^k\rangle_{\sMat}} \,.
\end{equation}
and the form we will use for Monte Carlo,
\begin{equation}
    \truncFac{N} =  \frac{\E\left[\sum_{k = 0}^\infty \windowFuncBP{N}{k}  \, \basepointfactor^{-k}\, \sampleTrVec^T \truMat^{-1} (\basepointerrMat \truMat^{-1})^k \sampleTrVec \right]}{\E\left[\sum_{k = 0}^\infty \windowFuncBPDenom{N}{k} \, \basepointfactor^{-k} \, \sampleTrVec^T \truMat^{-1} (\basepointerrMat \truMat^{-1})^k \sampleTrVec \right]} \,,
\end{equation}
where $\windowFuncBP{N}{k}$ and $\windowFuncBPDenom{N}{k}$ are new window functions that converge to $k + 1 - \basepointfactor$ and $k + 1$ respectively. To show that this has the same properties as the truncated shift factors in the previous section, we simply repeat the proof from the previous section, but with a few small changes. First, 
Now, we repeat the previous section to obtain the polynomial expression for $\truncFac{N}$,
\begin{equation}
    \truncFac{N} = \frac{\langle \genWindowBPPoly{N}(\basepointxMat) \rangle_{\sMat}}{\langle \genWindowDenomBPPoly{N}(\basepointxMat) \rangle_{\sMat}} \,,
\end{equation}
where, once again
    \begin{gather}
    \genWindowBPPoly{N}(x) = \sum_{j = 0}^N \primmPolyBP{j}(x) \,, \qquad \genWindowDenomBPPoly{N}(x) = \sum_{j = 0}^N \primmPolyDenomBP{j}(x) \,, \label{eq:polydef2}
\end{gather}
For brevity of notation, we define the quantity,
\begin{equation}
    \transformedMean \equiv 1 - \basepointfactor^{-1} \,.
\end{equation}

Now, our monotonicity analysis in this section is based upon the observation that for $x \geq 0$, it is the case that
\begin{equation}
    x^j (x - \transformedMean) \geq \transformedMean^j (x - \transformedMean) \,,
\end{equation}
and therefore, it is also the case that, for $x \geq 0$,
\begin{equation}
    x^j (x^k - \transformedMean^k) = x^j (x^{k - 1} + \transformedMean x^{k - 2} + ... + \transformedMean^{k - 2} x + \transformedMean^{k - 1}) (x - \transformedMean) \geq k \transformedMean^{j + k - 1} (x - \transformedMean) \,.
\end{equation}
Whereas the analysis in the previous section built a monotonic sequence of polynomials that were positive everywhere, the above formula allows us to build a monotonic sequence of polynomials that are positive in expectation, but not necessarily positive everywhere. To do this, we first note that by our $\E[\sampMat] \preceq \truMat$ assumption,
\begin{equation}
    \E [\basepointxMat] = \E[\id - \basepointfactor^{-1} \truMat^{-1/2} \sampMat \truMat^{-1/2}] \succeq \id - \basepointfactor^{-1} \truMat^{-1/2} \truMat \truMat^{-1/2} = \transformedMean \id \,.
\end{equation}
Hence, the above polynomial inequalities imply that
\begin{equation} \label{eq:posPoly}
    \E[ \basepointxMat^j (\basepointxMat^k - \transformedMean^k \id)] \succeq \E[k \transformedMean^{j + k - 1} (\basepointxMat - \transformedMean \id)] \succeq \mat{0} \,.
\end{equation}
This allows us to use the matrix polynomials $\basepointxMat^j (\basepointxMat^k - \transformedMean^k \id)$ as building blocks for a series that converges monotonically to the desired $\optAugFac$. The final observation that one needs to build the series is the fact that
\begin{equation}
    \sum_{k = 0}^\infty \transformedMean^k = \frac{1}{1 - \transformedMean} = \basepointfactor \,.
\end{equation}

With this established, we finally define the primitive polynomials
\begin{equation}
    \primmPolyBP{k}(x) \equiv k x^k - \transformedMean x^{k - 1} - \transformedMean^2 x^{k - 2} - ... - \transformedMean^{k - 1} x - \transformedMean^k = (k + 1) x^k - \sum_{j = 0}^k \transformedMean^{j} x^{k - j} \,.
\end{equation}
By \cref{eq:posPoly}, we have that
\begin{equation}
    \E[\primmPolyBP{k}(\basepointxMat)] = \E\left[(k + 1) \basepointxMat^k - \sum_{j = 0}^k \transformedMean^{j} \basepointxMat^{k - j}\right] = \sum_{j = 0}^k \E[\basepointxMat^{k - j} (\basepointxMat^j - \transformedMean^j)] \succeq \mat{0} \,.
\end{equation}
However, if one examines the individual terms of the composite sum $\sum_{k = 0}^\infty \primmPolyBP{k}(x)$ by powers of $x$, one observes that, for $x \in [0, 1)$,
\begin{equation}
    [x^j] \sum_{k = 0}^\infty \primmPolyBP{k}(x) = j - \sum_{k = 1}^\infty \transformedMean^k = j + 1 - \basepointfactor \,.
\end{equation}
Ergo, for $x \in [0, 1)$, we have that
\begin{equation}
     \genWindowBPPoly{N}(x) = \sum_{k = 0}^N \primmPolyBP{k}(x) \nearrow \sum_{k = 0}^\infty (k + 1 - \basepointfactor) x^k \,.
\end{equation}
And therefore, if we also take
\begin{equation}
    \primmPolyDenomBP{k}(x) = (k + 1) x^k \,,\qquad \genWindowDenomBPPoly{N}(x) = \sum_{k = 0}^N \primmPolyDenomBP{k}(x) = \sum_{k = 0}^N (k + 1) x^k \,,
\end{equation}
and note that
\begin{equation}
    \optAugFac = \frac{\sum_{k = 0}^\infty (k + 1 - \basepointfactor) \, \langle \basepointxMat^k \rangle_{\sMat}}{\sum_{k = 0}^\infty (k + 1) \, \langle \basepointxMat^k \rangle_{\sMat}} \,,
\end{equation}
we can conclude that
\begin{equation} \label{eq:truncFacPoly}
    \truncFac{N} = \frac{\langle \genWindowBPPoly{N}(\basepointxMat) \rangle_{\sMat}}{\langle \genWindowDenomBPPoly{N}(\basepointxMat)\rangle_{\sMat}} \rightarrow \optAugFac \,.
\end{equation}
and from the fact that the numerator is a sum of positive terms, and the fact that $\genWindowBPPoly{N}(\basepointxMat) \preceq \genWindowDenomBPPoly{N}(\basepointxMat)$ by construction, it therefore follows that
\begin{equation}
    0 \leq \truncFac{N} \leq 1\,,\qquad 0 \leq \optAugFac \leq 1 \,.
\end{equation}

To achieve a proof of monotonicity of the $\truncFac{N}$, we appeal to \cref{lem:monotonic}, which necessitates that we verify the inequality
\begin{equation}
\begin{split}
    a_k b_{k - 1} = \langle \primmPolyBP{k}(\basepointxMat) \rangle_{\sMat} \, \langle \primmPolyDenomBP{k - 1}(\basepointxMat) \rangle_{\sMat} \geq     \langle \primmPolyBP{k - 1}(\basepointxMat) \rangle_{\sMat} \, \langle \primmPolyDenomBP{k}(\basepointxMat) \rangle_{\sMat} = a_{k - 1} b_k\,.
\end{split}
\end{equation}
To do this, let us subtract and expand the above terms
\begin{equation} \label{eq:boundingstuff}
\begin{split}
    a_k b_{k - 1} - a_{k - 1} b_k &= \langle \primmPolyBP{k}(\basepointxMat) \rangle_{\sMat} \, \langle \primmPolyDenomBP{k - 1}(\basepointxMat) \rangle_{\sMat} - \langle \primmPolyBP{k - 1}(\basepointxMat) \rangle_{\sMat} \, \langle \primmPolyDenomBP{k}(\basepointxMat) \rangle_{\sMat} \\
    &= k \, \langle \basepointxMat^{k - 1} \rangle_{\sMat} \left((k + 1) \, \langle \basepointxMat^k \rangle_{\sMat} - \sum_{j = 0}^k \transformedMean^j \langle \basepointxMat^{k - j} \rangle_{\sMat}\right) \\
    &\qquad - (k + 1) \, \langle \basepointxMat^{k} \rangle_{\sMat} \left(k\, \langle \basepointxMat^{k - 1} \rangle_{\sMat} - \sum_{j = 0}^{k - 1} \transformedMean^j \langle \basepointxMat^{k - j - 1} \rangle_{\sMat} \right) \\
    &= (k + 1) \sum_{j = 0}^{k - 1} \transformedMean^j \, \langle \basepointxMat^{k} \rangle_{\sMat} \langle \basepointxMat^{k  - j - 1} \rangle_{\sMat} - k \sum_{j = 0}^{k} \transformedMean^j \, \langle \basepointxMat^{k - 1} \rangle_{\sMat} \, \langle \basepointxMat^{k - j} \rangle_{\sMat} \\
    &= k \sum_{j = 0}^{k - 1} \transformedMean^j \, \langle \basepointxMat^{k} \rangle_{\sMat} \langle \basepointxMat^{k  - j - 1} \rangle_{\sMat} - k \sum_{j = 0}^{k - 1} \transformedMean^j \, \langle \basepointxMat^{k - 1} \rangle_{\sMat} \, \langle \basepointxMat^{k - j} \rangle_{\sMat} \\
    &\qquad +\sum_{j = 0}^{k - 1} \transformedMean^j \, \langle \basepointxMat^{k} \rangle_{\sMat} \, \langle \basepointxMat^{k  - j - 1} \rangle_{\sMat} - k\, \transformedMean^k \, \langle \basepointxMat^{k - 1} \rangle_{\sMat} \, \langle \id \rangle_{\sMat} \,.
\end{split}
\end{equation}

We now appeal to the following lemma, which allows us to compare terms across the two sums above,

\begin{restatable}{lemma}{inequalitylemma}
Let $\hat{\mat{X}}$ be a random matrix such that $\hat{\mat{X}} \succeq \mat{0}$ a.s. For $i \geq j$ and $r \geq 0$, and any symmetric positive semi-definite matrix $\sMat \succeq \mat{0}$, we have that
\begin{equation} \label{eq:transformationinequality}
    \langle \hat{\mat{X}}^{i + r} \rangle_{\sMat} \, \langle \hat{\mat{X}}^{j - r} \rangle_{\sMat} \geq \langle \hat{\mat{X}}^i \rangle_{\sMat} \, \langle \hat{\mat{X}}^j \rangle_{\sMat}  \,.
\end{equation}
\end{restatable}
The proof of this fact is relegated to the appendix. However, applying this lemma to the above \cref{eq:boundingstuff} gives us
\begin{equation} \label{eq:dafinalstraw}
\begin{split}
     a_k b_{k - 1} - a_{k - 1} b_k &\geq \sum_{j = 0}^{k - 1} \transformedMean^j \, \langle \basepointxMat^{k} \rangle_{\sMat} \, \langle \basepointxMat^{k  - j - 1} \rangle_{\sMat} - k \, \transformedMean^k \, \langle \basepointxMat^{k - 1} \rangle_{\sMat} \langle \id \rangle_{\sMat} \\
     &=\sum_{j = 0}^{k - 1} \langle \basepointxMat^{k} \rangle_{\sMat} \, \left(\transformedMean^j \langle \basepointxMat^{k  - j - 1} \rangle_{\sMat} \right) - \sum_{j = 0}^{k - 1} \left(\transformedMean\, \langle \basepointxMat^{k - 1} \rangle_{\sMat} \right) \left(\transformedMean^{k - 1} \langle \id \rangle_{\sMat} \right) \,.
\end{split}
\end{equation}
Finally, we note that \cref{eq:posPoly} gives us $\transformedMean^j \E[\basepointxMat^{k - j - 1}] \succeq \transformedMean^{k - 1} \id $ and therefore 
\begin{equation}
    \transformedMean^j \, \langle \basepointxMat^{k - j - 1} \rangle_{\sMat} \geq \transformedMean^{k - 1} \langle \id \rangle_{\sMat} \,.
\end{equation}
Moreover, \cref{eq:posPoly} also gives us that $\E[\basepointxMat^k] \succeq \transformedMean \E[\basepointxMat^{k - 1}]$ and therefore
\begin{equation}
    \langle \basepointxMat^{k} \rangle_{\sMat} \geq \transformedMean \, \langle \basepointxMat^{k - 1} \rangle_{\sMat}
\end{equation}
Noting the the above two inequalities are between positive numbers and then substituting the above two inequalities into \cref{eq:posPoly} gives the desired result
\begin{equation}
    a_k b_{k - 1} - a_{k - 1} b_k \geq 0 \,.
\end{equation}
Thus, the truncated estimators $\truncFac{N}$ form a positive monotonic sequence that converges to $\optAugFac$. To summarize, we restate the results we have just proved into a theorem,

\begin{theorem} \label{thm:steag}
Under the assumptions in \cref{sec:assumptions}, consider operator shifting with shift $\augMat = \sampMat$ in energy norm $\|\cdot\|_{\truMat, \autoCor}$. Suppose that the random matrix $\sampMat \in \pdcone$ satisfies $\mat{0} \prec \sampMat \prec \basepointfactor \truMat$ almost surely. Then let $\truncFac{N}$ be the truncated approximations to the optimal shift factor $\optAugFac$, i.e.,
\begin{equation} \label{eq:shifted}
            \truncFac{N} = \frac{\sum_{k = 0}^N \windowFuncBP{N}{k} \, \basepointfactor^{-k} \, \langle \basepointxMat^k \rangle_{\sMat}}{\sum_{k = 0}^N  \windowFuncBPDenom{N}{k} \, \basepointfactor^{-k} \, \langle \basepointxMat^k\rangle_{\sMat}} =  \frac{\E\left[\sum_{k = 0}^N \windowFuncBP{N}{k}  \, \basepointfactor^{-k}\, \sampleTrVec^T \truMat^{-1} (\basepointerrMat \truMat^{-1})^k \sampleTrVec \right]}{\E\left[\sum_{k = 0}^N \windowFuncBPDenom{N}{k} \, \basepointfactor^{-k} \, \sampleTrVec^T \truMat^{-1} (\basepointerrMat \truMat^{-1})^k \sampleTrVec \right]} \,,
\end{equation}
where $\windowFuncBP{N}{k}$ and $\windowFuncBPDenom{N}{k}$ are given by
\begin{equation} \label{eq:windowfuncdefs}
    \windowFuncBP{N}{k} = \begin{cases}
        (k + 1) - \sum_{j = k}^N \transformedMean^{j - k} & k \leq N \\
        0 & \text{o.w.}
    \end{cases} \,, \qquad
    \windowFuncBPDenom{N}{k} = \begin{cases}
        k + 1 & k \leq N \\
        0 & \text{o.w.}
    \end{cases} \,,
\end{equation}
and $\transformedMean = 1 - \basepointfactor^{-1}$. Under these assumptions, we have that
\begin{gather*}
    \truncFac{N} \nearrow \optAugFac \text{ as } N \rightarrow \infty\,, \\
    0 \leq \truncFac{1} \leq \truncFac{2} \leq \truncFac{3} \leq ... \leq \truncFac{N} \leq ... \leq \optAugFac \leq 1 \,, \\
    \err{\sampMat^{-1}}{\truMat, \autoCor} \geq \err{(1 - \truncFac{1}) \sampMat^{-1}}{\truMat, \autoCor} \geq ... \geq \err{(1 - \truncFac{N}) \sampMat^{-1}}{\truMat, \autoCor} \geq ... \geq \err{(1 - \optAugFac) \sampMat^{-1}}{\truMat, \autoCor} \,.
\end{gather*}
\end{theorem}

\section{Accelerating Shifted Base-point Estimation} \label{sec:accel}

In practice, while the formula \cref{eq:shifted} provides a positive, monotonically increasing series of estimates $\truncFac{N}$ for the optimal $\optAugFac$ which only use $N$ powers of the matrix $\basepointerrMat \truMat^{-1}$, note that the larger one takes the factor $\basepointfactor$, the poorer the accuracy of the truncated approximation near the matrix $\truMat$, where most of the probability distribution is concentrated. Therefore, while we get a guarantee of an estimate that will decrease the value of the objective $\errBayes{\cdot}{\truMat}$, the convergence to the optimal factor $\optAugFac$ might be very slow as a result, necessitating larger and larger powers of $\basepointerrMat \truMat^{-1}$. Thus, in practice it is often a good idea to let the quantity $\basepointfactor$ be a function of the sample $\sampMat$ such that $\sampMat \preceq \basepointfactorFunc \truMat$. This means that, instead of using the estimator $\truncFac{N}$ in \cref{eq:shifted} above, we use
\begin{equation} \label{eq:accelTruncFac}
     \accelTruncFac{N} =  \frac{\E\left[\basepointfactorFunc^{-2} \sum_{k = 0}^\infty \accelWindowFuncBP{N}{k}  \, \basepointfactorFunc^{-k}\, \sampleTrVec^T \truMat^{-1} (\accelBasepointerrMat \truMat^{-1})^k \sampleTrVec \right]}{\E\left[\basepointfactorFunc^{-2} \sum_{k = 0}^\infty \accelWindowFuncBPDenom{N}{k} \, \basepointfactorFunc^{-k} \, \sampleTrVec^T \truMat^{-1} (\accelBasepointerrMat \truMat^{-1})^k \sampleTrVec \right]} \,,
\end{equation}
where one choice of $\basepointfactorFunc$ is 
\begin{equation}
    \basepointfactorFunc = \|\truMat^{-1/2} \sampMat \truMat^{-1/2}\|_2 \,,
\end{equation}
i.e., the smallest value for which $\sampMat \preceq \basepointfactorFunc \truMat$, and the windowing functions $\windowFuncBP{N}{k}$ and $\windowFuncBPDenom{N}{k}$ are defined as in \cref{eq:windowfuncdefs}. In practice, one may choose to approximate $\basepointfactorFunc$ instead of computing it exactly. Note in \cref{eq:accelTruncFac} the reintroduction of the $\basepointfactorFunc^{-2}$ terms in the numerator and denominator; originally these terms passed out of the expectation and cancelled, but now the explicit dependence on $\sampMat$ prevents this cancellation from happening. 

Computing $\|\truMat^{-1/2} \sampMat \truMat^{-1/2}\|_2$ can be done with power method. In particular, with probability 1, if $\trvec \in \Rdim$ is sampled from a distribution continuous with respect to the Lebesgue measure on its support, we have that
\begin{equation}
\begin{split}
    \basepointfactorFunc &= \lim_{k \rightarrow \infty} \frac{\|(\truMat^{-1/2} \sampMat \truMat^{-1/2})^k \trvec \|_2}{\|(\truMat^{-1/2} \sampMat \truMat^{-1/2})^{k - 1} \trvec \|_2} \\
    &= \lim_{k \rightarrow \infty} \sqrt{\frac{\trvec^T \truMat^{1/2} (\truMat^{-1} \sampMat)^k \truMat^{-1} (\sampMat \truMat^{-1} )^k \truMat^{1/2} \trvec }{\trvec^T \truMat^{1/2} (\truMat^{-1} \sampMat)^{k - 1} \truMat^{-1} (\sampMat \truMat^{-1} )^{k - 1} \truMat^{1/2} \trvec } }  \,.
\end{split}
\end{equation}
Since $\truMat$ is non-singular, transforming the random variable $\trvec$ by $\truMat^{1/2}$ transforms the corresponding distribution into a distribution continuous with respect to the Lebesgue measure on its support. Therefore, it is sufficient to compute/approximate
\begin{equation}
    \begin{split}
        \basepointfactorFunc &= \lim_{k \rightarrow \infty} \sqrt{\frac{\trvec^T (\truMat^{-1} \sampMat)^k \truMat^{-1} (\sampMat \truMat^{-1} )^k \trvec }{\trvec^T (\truMat^{-1} \sampMat)^{k - 1} \truMat^{-1} (\sampMat \truMat^{-1} )^{k - 1} \trvec } } = \lim_{k \rightarrow \infty} \frac{\| (\sampMat \truMat^{-1} )^k \trvec \|_{\truMat^{-1}}}{\| (\sampMat \truMat^{-1} )^{k - 1} \trvec \|_{\truMat^{-1}}} \,,
    \end{split}
\end{equation}
and as a result, we do not actually need to know $\truMat^{1/2}$ or $\truMat^{-1/2}$ to be able to compute the correct value of $\basepointfactor$. 

Now, to produce an algorithm, we follow the template of \cref{sec:genaugalg} -- we bootstrap $\truMat$ by replacing it with our sampled $\sampMat$ and bootstrap the expectation by using the distribution $\bootstrapDist$ instead of the true distribution $\truMatDist$. This nets us the approximate estimator
\begin{equation}
    \bootstrapTruncAugFac{N}(\sampMat) = \frac{\E_{\sampleTrVec  \sim \mathcal{N}(\vc{0},\combA), \bootstrapSampMat \sim \bootstrapDist}\left[\sum_{k = 0}^\infty \bootAccelWindowFuncBP{N}{k}{\bootstrapSampMat}  \, \basepointfactorFuncOf{\bootstrapSampMat}^{-k-2}\, \sampleTrVec^T \sampMat^{-1} (\bootAccelBasepointerrMat{\bootstrapSampMat} \sampMat^{-1})^k \sampleTrVec \right]}{\E_{\sampleTrVec  \sim \mathcal{N}(\vc{0},\combA), \bootstrapSampMat \sim \bootstrapDist}\left[\sum_{k = 0}^\infty \bootAccelWindowFuncBPDenom{N}{k}{\bootstrapSampMat} \, \basepointfactorFuncOf{\bootstrapSampMat}^{-k-2} \, \sampleTrVec^T \sampMat^{-1} (\bootAccelBasepointerrMat{\bootstrapSampMat} \sampMat^{-1})^k \sampleTrVec \right]} \,,
\end{equation}
where $\bootAccelBasepointerrMat{\bootstrapSampMat} = \sampMat - \basepointfactorFuncOf{\bootstrapSampMat} \bootstrapSampMat$. The above quantity can be estimated by Monte Carlo by computing
\begin{equation}
    \bootstrapMonteCarloTruncAugFac{N}(\sampMat) = \frac{\sum_{i = 0}^M \sum_{k = 0}^\infty \bootAccelWindowFuncBP{N}{k}{\sampBootstrapSampMat{i}}  \, \basepointfactorFuncOf{\sampBootstrapSampMat{i}}^{-k-2}\, \sampSampleTrVec{i}^T \sampMat^{-1} (\bootAccelBasepointerrMat{\sampBootstrapSampMat{i}} \sampMat^{-1})^k \sampSampleTrVec{i} }{\sum_{i = 0}^M \sum_{k = 0}^\infty \bootAccelWindowFuncBPDenom{N}{k}{\sampBootstrapSampMat{i}} \, \basepointfactorFuncOf{\sampBootstrapSampMat{i}}^{-k-2} \, \sampSampleTrVec{i}^T \sampMat^{-1} (\bootAccelBasepointerrMat{\sampBootstrapSampMat{i}} \sampMat^{-1})^k \sampSampleTrVec{i}} \,,
\end{equation}
where
\begin{equation}
\begin{aligned}[c]
\sampBootstrapSampMat{1}, ..., \sampBootstrapSampMat{M} &\sim \bootstrapDist\\
\sampSampleTrVec{1}, ..., \sampSampleTrVec{M} &\sim \mathcal{N}(\vc{0}, \autoCor)
\end{aligned}
\qquad
\begin{aligned}[c]
&\text{ i.i.d.}, \\
&\text{ i.i.d.}
\end{aligned}
\end{equation}
The full algorithm is presented in \cref{alg:truncatedEnergyAug}. 

Unfortunately, we do not believe that the monotonic guarantees of the previous two sections carry over when acceleration is applied. While it is not difficult to prove that the terms underneath the expectations in \cref{eq:accelTruncFac} become more accurate point-wise in $\sampMat$ (as we are shifting the base-point of the Taylor expansion closer to the point we are evaluating), it may be possible to construct contrived examples where this produces less accurate estimates of the expectations. However, we strongly believe that in almost all practical use cases, one should expect a significant improvement in accuracy in using this technique, as the reduction in truncation error is extremely substantial. 

\begin{algorithm}[t] 
\caption{Accel. Shifted Truncated En.-Norm Augmentation (\textbf{ES-TRA})}
\hspace*{\algorithmicindent} \textbf{Input}: A right hand side $\rhs$, an operator sample $\sampMat \sim \matDist$ with corresponding parameters $\sampParam \in \paramSpace$, a choice of second moment matrix $\autoCor$, a choice of matrix $\modMat$ satisfying the compatibility conditions, sample count $M$. \\
\hspace*{\algorithmicindent} \textbf{Output}: An estimate $\impSol$ of $\truMat^{-1} \rhs$. \\
\begin{algorithmic}[1] \label{alg:truncatedEnergyAug}
\State Factorize/preprocess $\sampMat$ to precompute $\sampMat^{-1}$ if necessary.
\State Draw $M$ i.i.d. bootstrap samples $\sampBootstrapSampMat{1}, ..., \sampBootstrapSampMat{M} \sim \bootstrapDist$.
\State For each $\sampBootstrapSampMat{i}$, perform power method to assign
\begin{equation*}
    \basepointfactorFuncOf{\sampBootstrapSampMat{i}} \gets \lim_{k \rightarrow \infty} \frac{\| (\sampBootstrapSampMat{i} \sampMat^{-1} )^k \trvec \|_{\sampMat^{-1}}}{\| (\sampBootstrapSampMat{i} \sampMat^{-1} )^{k - 1} \trvec \|_{\sampMat^{-1}}} \,.
\end{equation*}
\State Draw $M$ i.i.d. bootstrap samples $\sampSampleTrVec{1}, ..., \sampSampleTrVec{M} \sim \mathcal{N}(\vc{0}, \autoCor)$.
\State Assign 
\begin{equation*}
    \bootstrapMonteCarloAugFac \gets \frac{\sum_{i = 0}^M \sum_{k = 0}^\infty \bootAccelWindowFuncBP{N}{k}{\sampBootstrapSampMat{i}}  \, \basepointfactorFuncOf{\sampBootstrapSampMat{i}}^{-k-2}\, \sampSampleTrVec{i} ^T \modMat^T \sampMat^{-1} (\id - \basepointfactorFuncOf{\sampBootstrapSampMat{i}} \sampBootstrapSampMat{i} \sampMat^{-1})^k \sampSampleTrVec{i} }{\sum_{i = 0}^M \sum_{k = 0}^\infty \bootAccelWindowFuncBPDenom{N}{k}{\sampBootstrapSampMat{i}} \, \basepointfactorFuncOf{\sampBootstrapSampMat{i}}^{-k-2} \, \sampSampleTrVec{i}^T \modMat^T \sampMat^{-1} (\id - \basepointfactorFuncOf{\sampBootstrapSampMat{i}} \sampBootstrapSampMat{i} \sampMat^{-1})^k \modMat \sampSampleTrVec{i}} \,,
\end{equation*}
where
\begin{equation*}
    \windowFuncBP{N}{k} = \begin{cases}
        (k + 1) - \sum_{j = k}^N (1 - \basepointfactor^{-1})^{j - k} & k \leq N \\
        0 & \text{o.w.}
    \end{cases} \,, \quad
    \windowFuncBPDenom{N}{k} = \begin{cases}
        k + 1 & k \leq N \\
        0 & \text{o.w.}
    \end{cases} \,,
\end{equation*}
\State Clamp $\bootstrapMonteCarloAugFac \gets \max(0, \bootstrapMonteCarloAugFac)$.
\State Assign $\impSol \gets (\sampMat^{-1} - \bootstrapMonteCarloAugFac \sampMat^{-1} \modMat) \rhs$
\State Return $\impSol$.
\end{algorithmic}
\end{algorithm}

\section{Numerical Experiments} \label{sec:numerics}

In this section, we present numerical experiments to benchmark the above methods. We compare a number of different variations of operator shifting:

\begin{enumerate}
    \item \textbf{Naive}: Naive solve of the system $\sampMat \sampSol = \rhs$, by inverting the system directly without modifying the operator $\sampMat$.
    \item \textbf{GS} \emph{(General Operator Shifting)}: The method presented in \cref{sec:semiBayesian} and \cref{alg:genSemiBayesian}, where we take $\autoCor = \normMat = \id$ and let the prior on $\rhs$ be the standard normal distribution.
    \item \textbf{ES} \emph{(Energy-Norm Operator Shifting)}: The method presented in \cref{sec:energyNorm} without any truncation (i.e., computing the shift factor $\optAugFac$ directly using bootstrap and Monte-Carlo), where we take $ \autoCor  = \id$ and let the distribution of $\rhs$ be the standard normal distribution.
    \item \textbf{ES-T} \emph{(Truncated Energy Operator Shifting)}: The method presented in \cref{sec:taylorApprox}. In the numerical results, we test different orders of truncation. The order here denotes the highest power of a bootstrapped matrix sample which appears in the computation for the approximate shift factor. Furthermore, we will also test both \textbf{soft} (ES-T-S) and \textbf{hard} (ES-T-H) truncation windows, as discussed in \cref{sec:hardwindow}. 
    \item \textbf{ES-TRA} \emph{(Truncated Rebased Accelerated Energy Operator Shifting)}: The method presented in \cref{sec:accel} and \cref{alg:truncatedEnergyAug}. The order of truncation denotes the highest power of a bootstrapped matrix sample which appears in the computation. Unlike with ES-T, we will only benchmark the windowing function presented in \ref{eq:windowfuncdefs}. Like above, we take we take $ \autoCor  = \id$ and let the distribution of $\rhs$ be the standard normal distribution.
\end{enumerate}
In our numerical experiments, we measure two metrics of error:
\begin{enumerate}
    \item \textbf{R. MSE} \emph{(Relative Mean Squared Error)}: This is a normalized version of the error function $\errId{\cdot}$ with norm matrix $\normMat = \id$,
    \begin{equation}
    \begin{split}
        \text{R. MSE} &\equiv \frac{\errId{\impSol}}{\| \truMat^{-1} \|_F^2} = \frac{\E \|(\sampMat^{-1} - \augFac \augMat) - \truMat^{-1}\|_F^2}{\|\truMat^{-1}\|_F^2} \,.
    \end{split}
    \end{equation}
    Therefore, this quantity measures both the relative error of $\impSol$ from the true solution $\truSol$ in $L^2$, as well as the relative error from our augmented operator $\sampMat^{-1} - \augFac \augMat$ to the true operator $\truMat^{-1}$ in the Frobenius norm. We evaluate this quantity with Monte-Carlo and provide a $2\sigma$ estimate of the error of the Monte-Carlo procedure.
    \item \textbf{Rel. EMSE} \emph{(Relative Energy-Norm Mean Squared Error)}: This is defined like the above, except it is defined using the Energy norm $\|\cdot\|_\truMat$,
    \begin{equation}
        \text{Rel. EMSE} \equiv \frac{\err{\impSol }{\truMat}}{\| \truMat^{-1} \|_{\truMat, \id}^2} \,,
    \end{equation}
    this quantity may be of more interest than Rel. MSE in many problems, as for many elliptic systems, it more heavily penalizes high-frequency noise.
\end{enumerate}

\subsection{1D and 2D Poisson Equation on a Noisy Background}

Our first benchmark will be the Poisson equation, given by
\begin{equation} \label{eq:poisson}
\begin{split}
     \nabla \cdot (a(x) \nabla u(x)) = b(x), \qquad &\text{on } \mathcal{D} \,, \\
     u(x) = 0, \qquad &\text{on } \partial\mathcal{D} \,,
\end{split}
\end{equation}
where $a(x) > 0$ is a function determined by the physical background of the system. We discretize this equation using finite differences as follows: let $G_\mathcal{D} = (V, E)$ be a regular grid on the domain $\mathcal{D}$, with vertices $V$ and edges $E$. Let $\incidenceMat \in \R^{V \times E}$ be the (arbitrarily oriented) incidence operator of the grid, i.e.,
\begin{equation}
    \incidenceMat_{v,e} = \begin{cases} \pm 1 & v \text{ is incident to } e \\ 0 & \text{otherwise} \end{cases} \,,
\end{equation}
$\incidenceMat_{v,e}$ is positive for one of the $v$ incident to $e$ and negative for the other. The discrete approximation for the differential operator in \cref{eq:poisson} is given by
\begin{equation}
    \laplacian = -\incidenceMat \weightMat \incidenceMat^T \,,
\end{equation}
where $\weightMat \in \mathbb{R}^{E \times E}$ is a diagonal matrix whose $e, e$-th entry is the function $a$ evaluated at the midpoint of $e$. 

We suppose that we only have noisy measurements of the physical background, i.e. that the matrix $\weightMat$ is subject to some randomness, hence, in practice, we only have access to an approximate
\begin{equation}
    \sampLaplacian = -\incidenceMat  \sampWeightMat \incidenceMat^T \,,
\end{equation}
where $\sampLaplacian$ is drawn from a distribution $\matDist$, where $\truParam = (a(x_e))_{e \in E}$, i.e., the background $a$ evaluated at all the edge midpoints $x_e$. Note, to use the operator shifting method, one must prescribe a class of distributions $\matDistAt{\params}$ that we may sample from given background samples $\params$. 

In particular, the noisy background model we use for this benchmark perturbs every observation with independent multiplicative noise,
\begin{equation}
     \sampWeightMat_{e, e} = \sampParam_e = \hat{z}_e \params_e \,,
\end{equation}
where $\hat{z}_e \sim \mathcal{Z}$ i.i.d. for some positive distribution $\mathcal{Z}$ to be specified. To enforce Dirichlet boundary conditions we solve
\begin{equation}
    \begin{split}
        \sampLaplacian_{\text{int}(G_\mathcal{D}), \text{int}(G_\mathcal{D})} \vc{u}_{\text{int}(G_\mathcal{D})} + \sampLaplacian_{\text{int}(G_\mathcal{D}), \partial G_\mathcal{D}} \vc{u}_{\partial G_\mathcal{D}} &= \vc{b} \,, \\
        \vc{u}_{\partial G_\mathcal{D}} &= 0 \,,
    \end{split}
\end{equation}
where $\text{int}(G_\mathcal{D}) \subset V$ denotes the interior of the grid $G_\mathcal{D}$ and $\partial G_\mathcal{D} \subset V$ denotes the boundary, and $\sampLaplacian_{A, B}$ for $A \subset V$ and $B \subset V$ denotes the $A,B$-minor of $\sampLaplacian$. Hence, this becomes
\begin{equation} \label{eq:solvethispoisson}
    \sampMat \sampSol = \rhs \,,
\end{equation}
where $\sampMat = \sampLaplacian_{\text{int}(G_\mathcal{D}), \text{int}(G_\mathcal{D})}$ and $\rhs$ is the function $b(x)$ sampled at the interior vertices of $G_\mathcal{D}$. 

\begin{table}[]
    \centering
\begin{tabular}{r|ccllll}
Method & Order & Window & R. MSE & $\pm 2\sigma$ & R. EMSE & $\pm 2\sigma$ \\
\hline \hline
\textbf{Naive} & --- & --- & $12\%$ & $\pm 0.0352\%$ & $55.1\%$ & $\pm 0.1\%$ \\
\hline
\textbf{GS} & --- & --- & $0.59\%$ & $\pm 0.0203\%$ & $24.6\%$ & $\pm 0.448\%$ \\
\hline
\textbf{ES} & --- & --- & $4.32\%$ & $\pm 0.124\%$ & $20\%$ & $\pm 0.362\%$ \\
\hline
\textbf{ES-T} & 2 & Soft & $4.77\%$ & $\pm 0.155\%$ & $39.7\%$ & $\pm 0.723\%$ \\
\textbf{ES-T} & 4 & Soft & $1.26\%$ & $\pm 0.044\%$ & $21.5\%$ & $\pm 0.39\%$ \\
\textbf{ES-T} & 6 & Soft & $3.11\%$ & $\pm 0.0946\%$ & $20.1\%$ & $\pm 0.364\%$ \\
\hline
\textbf{ES-T} & 2 & Hard & $0.79\%$ & $\pm 0.0319\%$ & $22.9\%$ & $\pm 0.42\%$ \\
\textbf{ES-T} & 4 & Hard & $2.71\%$ & $\pm 0.0855\%$ & $20.3\%$ & $\pm 0.367\%$ \\
\hline
\textbf{ES-TRA} & 2 & --- & $0.798\%$ & $\pm 0.029\%$ & $22.7\%$ & $\pm 0.406\%$ \\
\textbf{ES-TRA} & 4 & --- & $2.88\%$ & $\pm 0.0913\%$ & $20.2\%$ & $\pm 0.366\%$ \\
\textbf{ES-TRA} & 6 & --- & $4.01\%$ & $\pm 0.117\%$ & $20\%$ & $\pm 0.354\%$ \\
\end{tabular}
    \caption{Comparison of augmentation methods for a \textbf{1D Poisson problem} on 128 grid points, where $a(x) = 1$ and $\hat{z}_e \sim \mathcal{U} \{ 0.5, 1.5 \}$. }
    \label{tab:onedpoissondiscrete}
\end{table}

\begin{table}[]
    \centering
\begin{tabular}{r|ccllll}
Method & Order & Window & R. MSE & $\pm 2\sigma$ & R. EMSE & $\pm 2\sigma$ \\
\hline \hline
\textbf{Naive} & --- & --- & $7.52\%$ & $\pm 0.0247\%$ & $58\%$ & $\pm 0.145\%$ \\
\hline
\textbf{GS} & --- & --- & $0.802\%$ & $\pm 0.0317\%$ & $32.5\%$ & $\pm 0.803\%$ \\
\hline
\textbf{ES} & --- & --- & $6.5\%$ & $\pm 0.189\%$ & $24.9\%$ & $\pm 0.546\%$ \\
\hline
\textbf{ES-T} & 2 & Soft & $3.28\%$ & $\pm 0.161\%$ & $46.6\%$ & $\pm 2.75\%$ \\
\textbf{ES-T} & 4 & Soft & $1.15\%$ & $\pm 0.0384\%$ & $29.4\%$ & $\pm 0.721\%$ \\
\textbf{ES-T} & 6 & Soft & $5.22\%$ & $\pm 0.177\%$ & $25.1\%$ & $\pm 0.524\%$ \\
\hline
\textbf{ES-T} & 2 & Hard & $1.07\%$ & $\pm 0.0385\%$ & $29.7\%$ & $\pm 0.716\%$ \\
\textbf{ES-T} & 4 & Hard & $6.27\%$ & $\pm 0.183\%$ & $25\%$ & $\pm 0.541\%$ \\
\hline
\textbf{ES-TRA} & 2 & --- & $1.28\%$ & $\pm 0.0435\%$ & $29\%$ & $\pm 0.801\%$ \\
\textbf{ES-TRA} & 4 & --- & $7.04\%$ & $\pm 0.203\%$ & $24.7\%$ & $\pm 0.507\%$ \\
\textbf{ES-TRA} & 6 & --- & $22.9\%$ & $\pm 0.633\%$ & $31.8\%$ & $\pm 0.586\%$ \\
\end{tabular}
   \caption{Comparison of augmentation methods for a \textbf{1D Poisson problem} on 128 grid points, where $a(x) = 1$ and $\hat{z}_e \sim \Gamma(\mu = 1, \sigma = 0.45)$. }
    \label{tab:onedpoissongamma}
\end{table}

\begin{table}[]
    \centering
    \begin{tabular}{r|ccllll}
Method & Order & Window & R. MSE & $\pm 2\sigma$ & R. EMSE & $\pm 2\sigma$ \\
\hline \hline
\textbf{Naive} & --- & --- & $6.43\%$ & $\pm 0.105\%$ & $45.3\%$ & $\pm 0.11\%$ \\
\hline
\textbf{GS} & --- & --- & $0.234\%$ & $\pm 0.00974\%$ & $25\%$ & $\pm 0.64\%$ \\
\hline
\textbf{ES} & --- & --- & $4.31\%$ & $\pm 0.772\%$ & $20\%$ & $\pm 0.456\%$ \\
\hline
\textbf{ES-T} & 2 & Soft & $2.57\%$ & $\pm 0.465\%$ & $35.6\%$ & $\pm 0.908\%$ \\
\textbf{ES-T} & 4 & Soft & $0.876\%$ & $\pm 0.133\%$ & $21.8\%$ & $\pm 0.504\%$ \\
\textbf{ES-T} & 6 & Soft & $2.59\%$ & $\pm 0.515\%$ & $20.3\%$ & $\pm 0.561\%$ \\
\hline
\textbf{ES-T} & 2 & Hard & $0.422\%$ & $\pm 0.039\%$ & $23.1\%$ & $\pm 0.464\%$ \\
\textbf{ES-T} & 4 & Hard & $2.13\%$ & $\pm 0.353\%$ & $20.5\%$ & $\pm 0.519\%$ \\
\hline
\textbf{ES-TRA} & 2 & --- & $0.845\%$ & $\pm 0.117\%$ & $22\%$ & $\pm 0.503\%$ \\
\textbf{ES-TRA} & 4 & --- & $3.62\%$ & $\pm 0.586\%$ & $20.1\%$ & $\pm 0.499\%$ \\
\textbf{ES-TRA} & 6 & --- & $5.61\%$ & $\pm 0.887\%$ & $20.2\%$ & $\pm 0.471\%$ \\
\end{tabular}
        \caption{Comparison of augmentation methods for a \textbf{2D Poisson problem} on 128 x 128 grid points, where $a(x) = 1$ and $\hat{z}_e \sim \mathcal{U} \{ 0.4, 1.6 \}$. }
    \label{tab:twodpoissondiscrete}
\end{table}

In \cref{tab:onedpoissondiscrete}, \cref{tab:onedpoissongamma}, and \cref{tab:twodpoissondiscrete}, we see the results of operator shifting applied to the above Poisson equation problem. As we can see, all our methods produce a substantial improvement in both relative MSE and relative energy-norm MSE --- with GS obtaining the largest reduction in $L^2$ error and ES obtaining the largest reduction in energy-norm error, as is to be expected. Moreover, note that the truncated methods ES-T and ES-TRA quickly approach the efficacy of ES as one increases the truncation order, with an order of 6 usually being enough to obtain an error comparable to baseline ES (which requires significantly more computation for large scale problems). Note also, that the energy error of ES-T is always monotonically decreasing, which agrees with \cref{thm:eag}. Moreover, note that the error of ES-TRA is not always monotonically decreasing. The unfortunate reality is that, while ES-TRA is guaranteed to converge to ES as the order becomes large, this convergence may be uneven, and is not guaranteed to be monotonic like ES-T with a soft window. We also note that the performance of our technique is comparable across different problems (i.e., 1D vs. 2D), as well as across different models of randomness (i.e., discrete vs. gamma). 

\subsection{Graph Laplacian Systems with Noisy Edge Weights}

One may extend the model in the above section to general graphs $G = (V, E)$. However, convention typically dictates that the Laplacian should be positive definite instead of negative definite, i.e.,
\begin{equation}
    \sampLaplacian = \incidenceMat \sampWeightMat \incidenceMat^T \,.
\end{equation}
In this model, we suppose that we are given a weighted graph $G$, however that the true edge weights of the graph, denoted by $\truEdge$, are unknown to us -- but we have access to a noisy observation $\sampTruEdge$ of $\truEdge$. Like in the previous setting, we suppose that the observations are independent. Therefore, the diagonal weight matrix $\sampWeightMat$ again has entries given by
\begin{equation}
    \sampWeightMat_{e, e} = \sampParam_e = \hat{z}_e \params_e \,,
\end{equation}
where $\hat{z}_e \sim \mathcal{Z}$ i.i.d. for some distribution $\mathcal{Z}$ to be specified. Like in the previous example, we solve a Dirichlet problem, arbitrarily selecting approximately six vertices as our boundary $\partial G$, whose values we set to zero. Thus, $\sampMat = \sampLaplacian_{\text{int}(G), \text{int}(G)}$ with $\text{int}(G) = V \setminus \partial G$ like before, and we again solve \cref{eq:solvethispoisson} with operator shifting.

\begin{figure}
    \centering
    \includegraphics[scale=0.4]{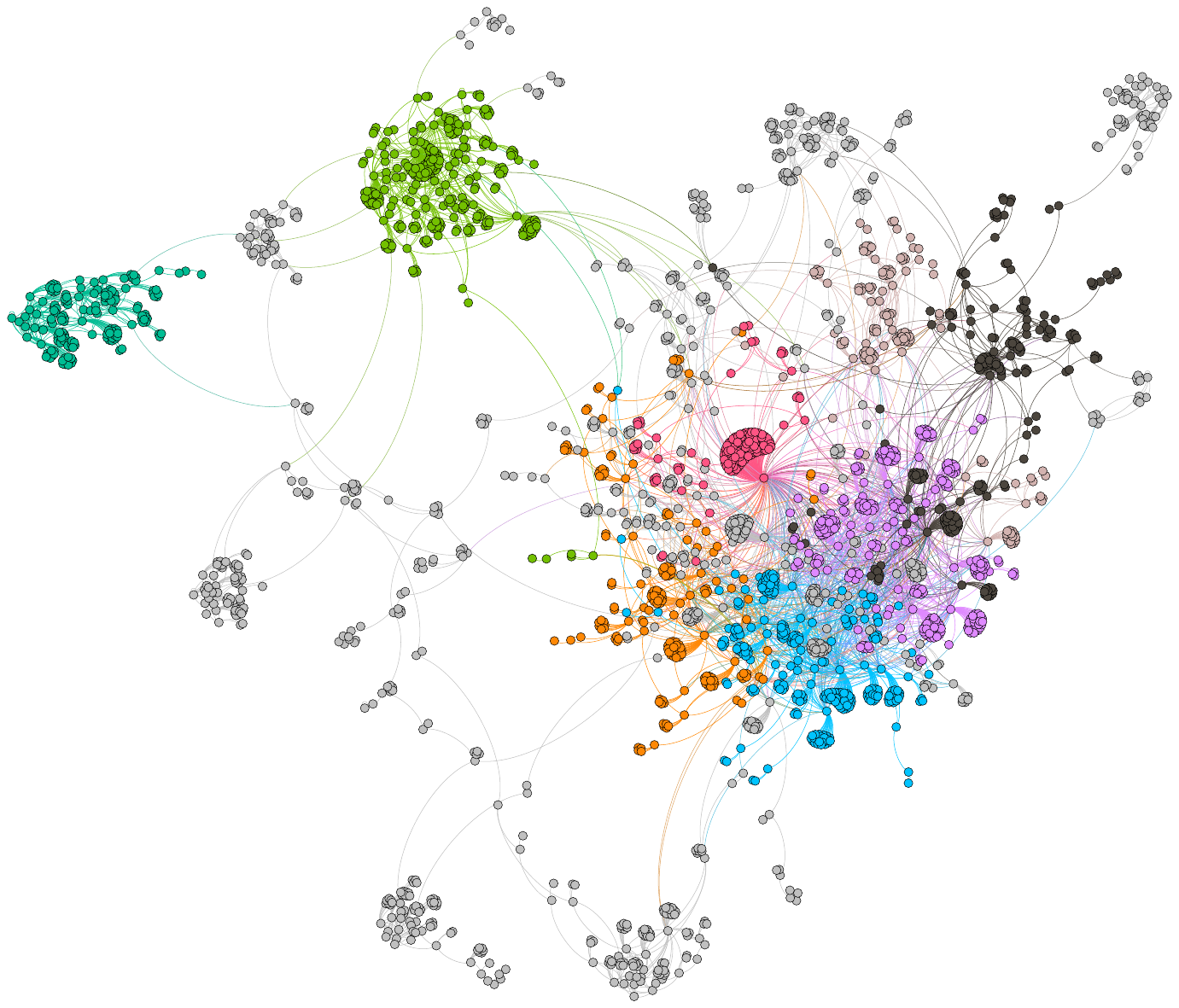}
    \caption{A visualization of the \emph{fb-pages-food} graph used in our numerical experiments. We give our performance results on this graph in \cref{tab:lap} and \cref{tab:edgedrop}. }
    \label{fig:fb-pages-food}
\end{figure}

\begin{table}[]
    \centering
\begin{tabular}{r|ccllll}
Method & Order & Window & R. MSE & $\pm 2\sigma$ & R. EMSE & $\pm 2\sigma$ \\
\hline \hline
\textbf{Naive} & --- & --- & $46.9\%$ & $\pm 0.386\%$ & $46.2\%$ & $\pm 0.14\%$ \\
\hline
\textbf{GS} & --- & --- & $17.5\%$ & $\pm 1.27\%$ & $19.1\%$ & $\pm 0.464\%$ \\
\hline
\textbf{ES} & --- & --- & $18.1\%$ & $\pm 1.37\%$ & $19.1\%$ & $\pm 0.502\%$ \\
\hline
\textbf{ES-T} & 2 & Soft & $34.9\%$ & $\pm 2.79\%$ & $34.6\%$ & $\pm 0.999\%$ \\
\textbf{ES-T} & 4 & Soft & $18.9\%$ & $\pm 1.1\%$ & $20\%$ & $\pm 0.433\%$ \\
\textbf{ES-T} & 6 & Soft & $16.9\%$ & $\pm 1.03\%$ & $18.6\%$ & $\pm 0.407\%$ \\
\hline
\textbf{ES-T} & 2 & Hard & $20.1\%$ & $\pm 1.11\%$ & $21\%$ & $\pm 0.438\%$ \\
\textbf{ES-T} & 4 & Hard & $17.8\%$ & $\pm 1.2\%$ & $19\%$ & $\pm 0.452\%$ \\
\hline
\textbf{ES-TRA} & 2 & --- & $20.6\%$ & $\pm 1.18\%$ & $21.6\%$ & $\pm 0.451\%$ \\
\textbf{ES-TRA} & 4 & --- & $18\%$ & $\pm 1.13\%$ & $19.3\%$ & $\pm 0.431\%$ \\
\textbf{ES-TRA} & 6 & --- & $18.2\%$ & $\pm 1.17\%$ & $19.3\%$ & $\pm 0.449\%$ \\
\end{tabular}
    \caption{Comparison of augmentation methods for a \textbf{graph Laplacian system}. This particular weighted graph is the \emph{fb-pages-food} graph \cite{rozemberczki2019gemsec}, visualized in \cref{fig:fb-pages-food}. In this benchmark we have $\hat{z}_e \sim \mathcal{U} \{ 0.5, 1.5 \}$.}
    \label{tab:lap}
\end{table}

\begin{figure}
    \centering
    \includegraphics[scale=0.6]{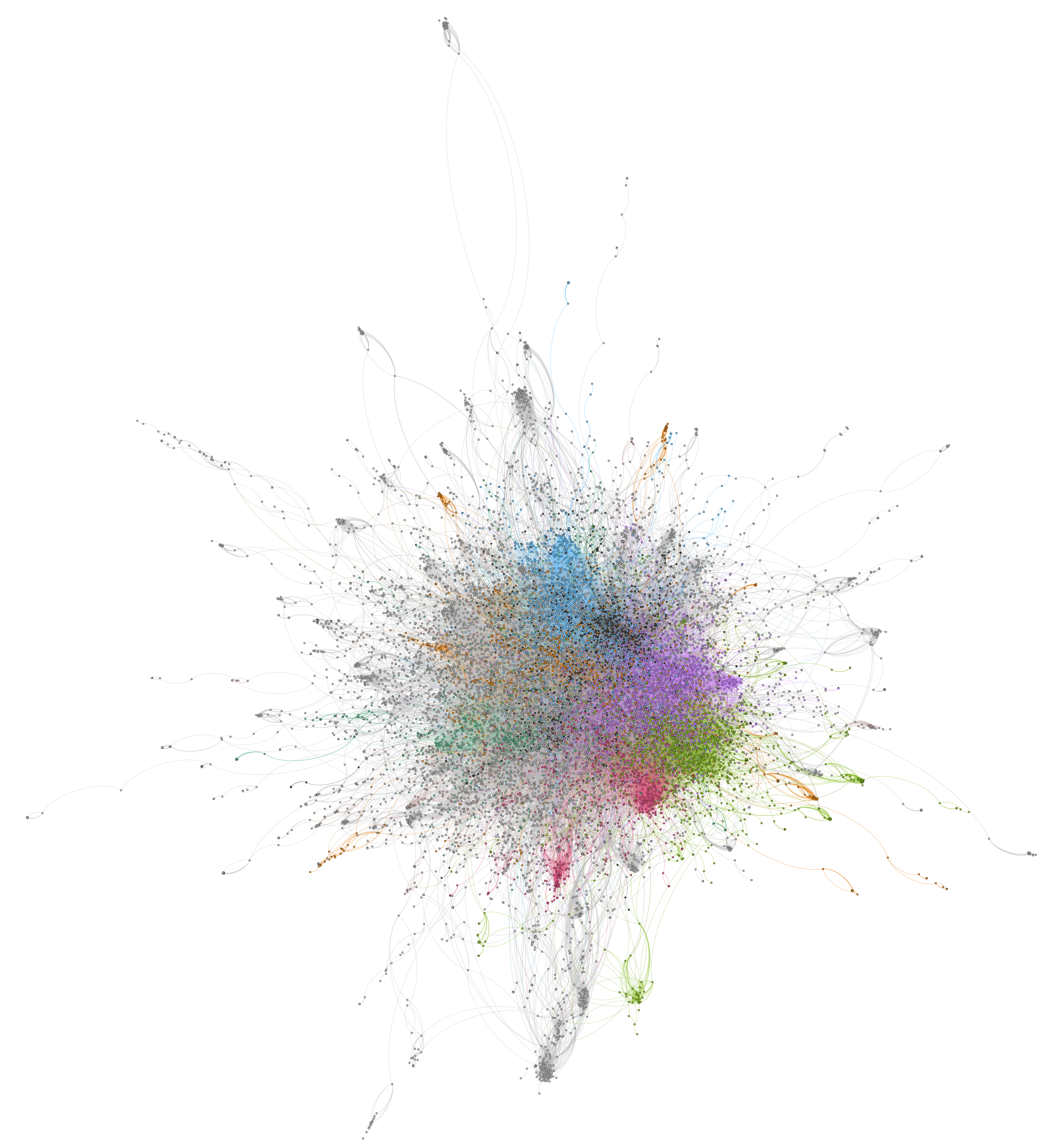}
    \caption{A visualization of the \emph{fb-pages-company} graph used in our numerical experiments. We give our performance results on this graph in \cref{tab:lap} and \cref{tab:edgedrop}. }
    \label{fig:fb-pages-company}
\end{figure}

\begin{table}[]
    \centering
    \begin{tabular}{r|ccllll}
Method & Order & Window & R. MSE & $\pm 2\sigma$ & R. EMSE & $\pm 2\sigma$ \\
\hline \hline
\textbf{Naive} & --- & --- & $33.8\%$ & $\pm 3.16\%$ & $32.4\%$ & $\pm 1.28\%$ \\
\hline
\textbf{GS} & --- & --- & $13.3\%$ & $\pm 4.92\%$ & $16.8\%$ & $\pm 2.22\%$ \\
\hline
\textbf{ES} & --- & --- & $13.7\%$ & $\pm 5.7\%$ & $16.1\%$ & $\pm 2.15\%$ \\
\hline
\textbf{ES-T} & 2 & Soft & $32.5\%$ & $\pm 17.9\%$ & $27.8\%$ & $\pm 7.67\%$ \\
\textbf{ES-T} & 4 & Soft & $18\%$ & $\pm 11.3\%$ & $18.1\%$ & $\pm 4.78\%$ \\
\textbf{ES-T} & 6 & Soft & $22.3\%$ & $\pm 12\%$ & $20.1\%$ & $\pm 4.76\%$ \\
\hline
\textbf{ES-T} & 2 & Hard & $11.8\%$ & $\pm 4.59\%$ & $15.3\%$ & $\pm 2.18\%$ \\
\textbf{ES-T} & 4 & Hard & $13.1\%$ & $\pm 4.49\%$ & $14.9\%$ & $\pm 1.7\%$ \\
\hline
\textbf{ES-TRA} & 2 & --- & $11.6\%$ & $\pm 4.08\%$ & $15.8\%$ & $\pm 1.81\%$ \\
\textbf{ES-TRA} & 4 & --- & $19.8\%$ & $\pm 5.88\%$ & $22\%$ & $\pm 2.13\%$ \\
\textbf{ES-TRA} & 6 & --- & $35.4\%$ & $\pm 11.3\%$ & $34.7\%$ & $\pm 4.17\%$ \\
\end{tabular}
    \caption{Comparison of augmentation methods for a \textbf{graph Laplacian system}. This particular weighted graph is the \emph{fb-pages-company} graph \cite{rozemberczki2019gemsec}, visualized in \cref{fig:fb-pages-company}. In this benchmark we have $\hat{z}_e \sim \mathcal{U} \{ 0.5, 1.5 \}$.}
    \label{tab:lap2}
\end{table}

We see the results of this computation in \cref{tab:lap} and \cref{tab:lap2}. The graphs shown are from the Network Repository \cite{nr}. We note that the method performs quite similarly on this problem as it does on the grid Laplacian case --- this shows that the performance of the method is consistent across different types of problems.

\subsection{Heat Steady-State with Sparsified Graph Laplacians}

In many areas of computer science, one can use graph sparsification techniques to reduce the complexity of a Laplacian system solve if one is able to tolerate some degree of approximation. These graph sparsification techniques work by randomly selecting some subset of the edges of the graph $G$ to remove and then re-weighting the remaining edges to obtain a sparsifed graph $\hat{G}$. We consider the problem of approximating the solution to a Laplacian system on $G$ using the Laplacian of $\hat{G}$. In particular, suppose we are interested in the steady-state heat distribution given by
\begin{equation}
    (\laplacian + \gamma \id) \vc{u} = \rhs \,,
\end{equation}
where $\gamma > 0$ is the coefficient of heat decay and $\rhs$ is the vector describing heat introduced to the system per unit time. However, we only have access to the topology of the sparsified $\hat{G}$ and its Laplacian $\sampLaplacian$. Naively, one could solve
\begin{equation}
    (\sampLaplacian + \gamma \id) \hat{\vc{u}} = \rhs \,.
\end{equation}
Of course, this naive solution carries a certain amount of error. Note that we can apply operator shifting to $\sampLaplacian + \gamma \id$ to obtain a more accurate solution.

In particular, for this numerical experiment, we use the sparsification model
\begin{equation}
     \sampWeightMat_{e, e} = \sampParam_e = \hat{z}_e \params_e \,
\end{equation}
where $\hat{z}_e \sim p^{-1} \text{Ber}(p)$ i.i.d., for $p \in (0, 1)$. 

\begin{table}[]
    \centering
\begin{tabular}{r|ccllll}
Method & Order & Window & R. MSE & $\pm 2\sigma$ & R. EMSE & $\pm 2\sigma$ \\
\hline \hline
\textbf{Naive} & --- & --- & $18.5\%$ & $\pm 0.0843\%$ & $26.2\%$ & $\pm 0.125\%$ \\
\hline
\textbf{GS} & --- & --- & $12.6\%$ & $\pm 0.386\%$ & $16.9\%$ & $\pm 0.613\%$ \\
\hline
\textbf{ES} & --- & --- & $13.8\%$ & $\pm 0.258\%$ & $16.9\%$ & $\pm 0.38\%$ \\
\hline
\textbf{ES-T} & 2 & Soft & $16.9\%$ & $\pm 0.771\%$ & $23.7\%$ & $\pm 1.18\%$ \\
\textbf{ES-T} & 4 & Soft & $12.7\%$ & $\pm 0.448\%$ & $17.5\%$ & $\pm 0.737\%$ \\
\textbf{ES-T} & 6 & Soft & $13\%$ & $\pm 0.379\%$ & $17.2\%$ & $\pm 0.611\%$ \\
\hline
\textbf{ES-T} & 2 & Hard & $14.1\%$ & $\pm 0.528\%$ & $20.1\%$ & $\pm 0.802\%$ \\
\textbf{ES-T} & 4 & Hard & $12.2\%$ & $\pm 0.326\%$ & $16.4\%$ & $\pm 0.498\%$ \\
\hline
\textbf{ES-TRA} & 2 & --- & $12.6\%$ & $\pm 0.404\%$ & $17\%$ & $\pm 0.627\%$ \\
\textbf{ES-TRA} & 4 & --- & $13.7\%$ & $\pm 0.315\%$ & $17\%$ & $\pm 0.493\%$ \\
\textbf{ES-TRA} & 6 & --- & $15.3\%$ & $\pm 0.318\%$ & $17.8\%$ & $\pm 0.367\%$ \\
\end{tabular}
\caption{Comparison of augmentation methods for a \textbf{sparsified graph Laplacian system}. This particular weighted graph is the \emph{fb-pages-food} graph \cite{rozemberczki2019gemsec}, visualized in \cref{fig:fb-pages-food}. In this benchmark we have $\hat{z}_e \sim \frac{1}{0.75} \text{Ber}(0.75)$ with $\gamma = 1$.}
    \label{tab:edgedrop}
\end{table}

We see in \cref{tab:edgedrop} that our methods allow for a substantial reduction in energy-norm mean squared error like in the previous two scenarios. However, this scenario seems to be more difficult for the augmentation process. Particularly, the $L^2$ reduction is not as high as with previous examples. Regardless, the fact that operator shifting functions under this regime of noise shows us that operator shifting is a technique which can be broadly applied to various problems.

\section{Conclusion}

In this paper we have presented a novel method for reducing error in elliptic systems corrupted by noise that requires only a single sample of a corrupted system. We have introduced the GS and ES methods, as well as the ES-T and ES-TRA methods for efficiently approximating ES. Moreover, we have proved multiple important theorems that underlie our methods -- this includes the error reduction bounds in \cref{thm:fullOpAug_revision} and \cref{thm:fullOpAugEnergy} for the GS and ES methods respectively, as well as monotone convergence guarantees \cref{thm:eag} and \cref{thm:steag} that provide justification and intuition for the ES-T and ES-TRA methods.

Furthermore, we have demonstrated in our numerical experiments that the operator shifting methods we presented are effective in many different scenarios and different noise models -- consistently providing a $2\times$ reduction in energy mean-squared error, and often a significantly higher reduction in $L^2$ error. We have also shown that ES-T and ES-TRA converge relatively quickly to ES, which makes these truncated method good alternatives when solving a large number of matrix systems is computationally intractable.

Our numerical results also make clear the relative benefits and trade-offs of the different augmentation methods; these are seen in \cref{tab:methodcomp}. As per these trade-offs, we recommend using ES if computation is not an issue. If computation is an issue, we recommend using hard-window (or soft-window) ES-T, depending on the scenario, and if this approximation seems not to be performing well, or the noise distribution is heavy tailed, then we recommend using ES-TRA.

\begin{table}[]
    \centering
    \begin{tabular}{r|ccccc}
Method & Computation & $L^2$ & Energy & Convergence & Monotone \\
\hline \hline
\textbf{Naive} & Lowest & --- & --- & --- & --- \\
\textbf{GS} & High & Best & Good & --- & --- \\
\textbf{ES} & High & Good & Best & --- & --- \\
\textbf{ES-T-S} & Low & Good & Better & When $\sampMat \prec 2\truMat$ & Always \\
\textbf{ES-T-H} & Low & Good & Better+ & When $\sampMat \prec 2\truMat$  & Empirically \\
\textbf{ES-TRA} & Moderate & Good- & Better- & Pointwise & No 
\end{tabular}
        \caption{Comparison of pros/cons of different augmentation methods presented in this paper. $L^2$ and energy denote reduction in $L^2$ and energy-norm error respectively. (S) and (H) denote hard and soft windows respectively. Convergence denotes whether or not the method converges to ES when the order is taken to be large, monotone denotes whether or not the truncated shift factors $\truncFac{N}$ of the method are monotonic.}
    \label{tab:methodcomp}
\end{table}

While the operator shifting framework offers a new approach to reducing error in noisy elliptic systems, there are still a number of interesting avenues for further exploration. The most obvious is, of course, the extension of the operator shifting framework machinery to the case of asymmetric systems. Unfortunately, while there is nothing preventing one from using the same approach for asymmetric systems, the question of how one would analyze such an algorithm remains open. The machinery developed within does not relatively apply, since the move from symmetric to asymmetric systems breaks a number of core tools used throughout. Since many systems of interest are indeed asymmetric, this is an important direction for future research. In addition, while we leave the optional choice of matrices $\normMat, \autoCor, \modMat, \modMatTwo$ up to the reader -- it is yet unclear how one should approach making a choice for these optional parameters in general. Finally, to judge the performance of the method in real-world problems, one could apply the techniques we've developed within to an application area where elliptic systems are corrupted by randomness --- possible aforementioned applications include structural dynamics \cite{soize2005comprehensive} or whether modelling \cite{palmer2005representing}, among a plethora of others.

\section{Acknowledgements}

The work of L.Y. is partially supported by the National Science Foundation under award DMS1818449 and DMS-2011699.


\appendix

\section{Proofs of Miscellaneous Lemmas} 

\loewnerinversionlemma*

\begin{proof}
Consider the exact second order Taylor expansion of the inverse functional on the space of positive definite matrices,
\begin{equation}
    \sampMat^{-1} = \truMat^{-1} - \truMat^{-1} (\sampMat - \truMat) \truMat^{-1} + \truMat_*^{-1} (\sampMat - \truMat) \truMat_*^{-1} (\sampMat - \truMat) \truMat_*^{-1} \,,
\end{equation}
where $\truMat_*$ is a matrix between $\truMat$ and $\sampMat$. Note that the last term is positive semi-definite because $\truMat_*$ is positive definite. Therefore,
\begin{equation}
    \sampMat^{-1} \succeq \truMat^{-1} - \truMat^{-1} (\sampMat - \truMat) \truMat^{-1} \,.
\end{equation}
Taking expectations of both sides and using the fact that $\E[\sampMat - \truMat] \preceq \mat{0}$ yields
\begin{equation}
    \E[\sampMat^{-1}] \succeq \truMat^{-1} - \truMat^{-1} \E[\sampMat - \truMat] \truMat^{-1} \succeq \truMat^{-1} \,.
\end{equation}
\end{proof}

\begin{lemma} \label{lem:momentslemma}
Let $\hat{\mat{W}}$ be a symmetric random matrix that satisfies $(1 - \varepsilon) \id \preceq \hat{\mat{W}} \prec \id$ almost surely and $\E [(\id - \hat{\mat{W}})^{-2}]$ exists. Then it is the case that:
\begin{equation}
    \E \| \hat{\mat{W}}^k \|^2_F = o(1/k^2) \,.
\end{equation}
\end{lemma}

\begin{proof}
    Note that $\hat{\mat{W}}$ is symmetric and hence can always has a spectral decomposition 
    \begin{equation}
        \hat{\mat{W}} = \hat{\mat{Q}} \hat{\mat{\Lambda}} \hat{\mat{Q}}^T \,.
    \end{equation}
    Using the above decomposition, for any positive $\gamma > 0$, we can split the matrix $\hat{\mat{W}}$ into two matrices $\hat{\mat{W}}_{\geq \gamma} + \hat{\mat{W}}_{< \gamma}$ with the properties
    \begin{equation}
        \gamma \id \preceq \hat{\mat{W}}_{\geq \gamma} \prec \id \,, \qquad -(1 - \varepsilon) \preceq \hat{\mat{W}}_{< \gamma} \prec \gamma \id \,.
    \end{equation}
    We do this by defining $\hat{\mat{\Lambda}}_{\geq \gamma}$ and $\hat{\mat{\Lambda}}_{< \gamma}$ to be $\hat{\mat{\Lambda}}$ but with all entries zeroed that don't fall within the ranges $[\gamma, 1)$ and $[-1 + \varepsilon, \gamma)$ respectively. Then we have that $\hat{\mat{\Lambda}} = \hat{\mat{\Lambda}}_{\geq \gamma} + \hat{\mat{\Lambda}}_{< \gamma}$ and therefore, we can define:
    \begin{equation}
        \hat{\mat{W}}_{\geq \gamma} \equiv \hat{\mat{Q}} \hat{\mat{\Lambda}}_{\geq \gamma} \hat{\mat{Q}}^T \,, \qquad \hat{\mat{W}}_{< \gamma} \equiv \hat{\mat{Q}} \hat{\mat{\Lambda}}_{< \gamma} \hat{\mat{Q}}^T \,.
    \end{equation}
    Moreover, since $\hat{\mat{\Lambda}}^{k} = \hat{\mat{\Lambda}}^{k}_{\geq \gamma} + \hat{\mat{\Lambda}}^{k}_{< \gamma}$, we have that
    \begin{equation}
        \hat{\mat{W}}^{k} = \hat{\mat{W}}^{k}_{\geq \gamma} + \hat{\mat{W}}^{k}_{< \gamma} \,.
    \end{equation}
    Hence, it follows that:
    \begin{equation}
        k^2 \|\hat{\mat{W}}^{k}\|_F^2 \leq 2 k^2 \|\hat{\mat{W}}^{k}_{\geq \gamma}\|_F^2 + 2 k^2 \, \|\hat{\mat{W}}^{k}_{< \gamma}\|_F^2 \,.
    \end{equation}
    Furthermore, since $\|\hat{\mat{W}}^{k}_{< \gamma}\|_F^2$ is the sum of the eigenvalues of $\hat{\mat{W}}^{2k}_{< \gamma}$, which are all bounded by $\max(1 - \varepsilon, \gamma)^{2k}$, it follows again that:
    \begin{equation} \label{eq:appxfirstbound}
        k^2 \|\hat{\mat{W}}^{k}\|_F^2 \leq 2 k^2 \, \|\hat{\mat{W}}^{k}_{\geq \gamma}\|_F^2 + 2 \dimension k^2 \cdot \max(1 - \varepsilon, \gamma)^{2k}\,.
    \end{equation}
    For the remaining term $\|\hat{\mat{W}}^{k}_{\geq \gamma}\|_F^2$, we note that for $x \in [0, 1)$,
    \begin{equation}
        k^2 x^{k} \leq 2 \sum_{i = 1}^k i x^k \leq 2 \sum_{i = 1}^k i x^i \,,
    \end{equation}
    wheras the exact Taylor expansion for $1/(1-x)^2$ to $k+1$th order has the form:
    \begin{equation}
        \frac{1}{(1 - x)^2} = \sum_{i = 1}^k i x^i + (k + 1) y^{k + 1} \geq \sum_{i = 1}^k i x^i  \,,
    \end{equation}
    where $0 \leq y \leq x$. Thus, for $x \in [0, 1)$,
    \begin{equation}
        k^2 x^{k} \leq \frac{2}{(1 - x)^2} \,.
    \end{equation}
    Hence, since all eigenvalues of $\hat{\mat{W}}^{k}_{\geq \gamma}$ lie in the range $[\gamma, 1)$, it follows that
    \begin{equation}
        \begin{split}
        k^2 \hat{\mat{W}}_{\geq \gamma}^{2k} &= k^2 \hat{\mat{W}}_{\geq \gamma}^{2k} \, (1 - \mathds{1}(\hat{\mat{W}} \preceq \gamma \id)) \\
        &\preceq \frac{1}{2} (\id - \hat{\mat{W}}_{\geq \gamma})^{-2} \, (1 - \mathds{1}(\hat{\mat{W}} \preceq \gamma \id)) \\
        &\preceq \frac{1}{2} (\id - \hat{\mat{W}})^{-2} \, (1 - \mathds{1}(\hat{\mat{W}} \preceq \gamma \id))\,,
        \end{split}
    \end{equation}
    where $ \mathds{1}(\hat{\mat{W}} \preceq \gamma \id)$ is the indicator function for the event $\{ \hat{\mat{W}} \preceq \gamma \id \}$. The first line is by virtue of the fact that $\hat{\mat{W}}_{\geq \gamma}^{2k}$ is zero on the set $\{ \hat{\mat{W}} \preceq \gamma \id \}$. Thus, we may substitute this into \cref{eq:appxfirstbound} to obtain:
    \begin{equation}
        k^2 \E \|\hat{\mat{W}}^{k}\|_F^2 \leq 2 \E [\tr((\id - \hat{\mat{W}})^{-2}) \, (1 - \mathds{1}(\hat{\mat{W}} \preceq \gamma \id)))] + 2 \dimension k^2 \cdot \max(1 - \varepsilon, \gamma)^{2k}\,.
    \end{equation}
    Now, we choose $\gamma$ to be $\gamma = 1 - 1/\sqrt{k}$. This gives 
    \begin{equation} \label{eq:almostthere}
        k^2 \E \|\hat{\mat{W}}^{k}\|_F^2 \lesssim \E [\tr((\id - \hat{\mat{W}})^{-2}) \, (1 - \mathds{1}(\hat{\mat{W}} \preceq (1 - 1/\sqrt{k})\id))] + k^2 (1- 1/\sqrt{k})^{2k}\,.
    \end{equation}

    The two terms above are quite easy to bound, note that
    \begin{equation} \label{eq:appexfact1}
        k^2 (1- 1/\sqrt{k})^{2k} \leq k^2 \exp(-2k/\sqrt{k}) = k^2 \exp(-2\sqrt{k}) = o(1) \,.
    \end{equation}
    Conversely, we have 
    \begin{equation}
        \tr((\id - \hat{\mat{W}})^{-2}) \,\mathds{1}(\hat{\mat{W}} \preceq (1 - 1/\sqrt{k})\id) \nearrow \tr((\id - \hat{\mat{W}})^{-2}) \,.
    \end{equation}
    Therefore, it follows from monotone convergence theorem and the convergence of $\E[(\id - \hat{\mat{W}})^{-2}]$ that
    \begin{equation} \label{eq:appexfact2}
        \E [\tr((\id - \hat{\mat{W}})^{-2}) \, (1 - \mathds{1}(\hat{\mat{W}} \preceq (1 - 1/\sqrt{k})\id))] = o(1) \,.
    \end{equation}
    Plugging \cref{eq:appexfact1} and \cref{eq:appexfact2} into \cref{eq:almostthere} gives the desired result,
    \begin{equation}
        \E \|\hat{\mat{W}}^{k}\|_F^2 = o(1/k^2) \,.
    \end{equation}
\end{proof}

\convergencelemma*

\begin{proof}
Via a transformation of variables, it suffices to prove the statements 
\begin{equation}
    (\id - \hat{\mat{W}})^{-1} = \sum_{k = 0}^\infty \hat{\mat{W}}^k \,, \qquad (\id - \hat{\mat{W}})^{-2} = \sum_{k = 0}^\infty (k + 1) \hat{\mat{W}}^k \,,
\end{equation}
in the mean squared Frobenius norm when $-(1 - \varepsilon)\id \prec \hat{\mat{W}} \prec \id$ almost surely and $\E[(\id - \hat{\mat{W}})^{-2}] \prec \infty$. Let us write:
\begin{equation}
    \begin{split}
        \E \left\| (\id - \hat{\mat{W}})^{-1} - \sum_{k = 0}^N \hat{\mat{W}}^k \right\|_F^2 &\leq \left( \E \left\| (\id - \hat{\mat{W}})^{-1}\right\|_F^2 \right) \left( \E \left\| \id - (\id - \hat{\mat{W}}) \sum_{k = 0}^N \hat{\mat{W}}^k  \right\|_F^2 \right) \\
        &= \left( \E \left\| (\id - \hat{\mat{W}})^{-1}\right\|_F^2 \right) \left( \E \left\| \hat{\mat{W}}^{N + 1} \right\|_F^2 \right) \\
        &=  \tr (\E (\id - \hat{\mat{W}})^{-2}) \left( \E \left\| \hat{\mat{W}}^{N + 1} \right\|_F^2 \right) \\
        &\lesssim \E \left\| \hat{\mat{W}}^{N + 1} \right\|_F^2 \rightarrow 0 \,.
    \end{split}
\end{equation} 
The first inequality above is by Cauchy-Schwartz and convergence in the last line is by \cref{lem:momentslemma}. 
\begin{equation}
    \begin{split}
        \E \left\| (\id - \hat{\mat{W}})^{-2} - \sum_{k = 0}^N (k + 1) \hat{\mat{W}}^k \right\|_F^2 &\leq \left( \E \left\| (\id - \hat{\mat{W}})^{-1}\right\|_F^2 \right)^2 \left( \E \left\| \id - (\id - \hat{\mat{W}})^2 \sum_{k = 0}^N (k + 1) \hat{\mat{W}}^k  \right\|_F^2 \right) \\
        &= \left( \tr(\E(\id - \hat{\mat{W}})^{-2}) \right)^2 \left( \E \left\| \id - (\id - \hat{\mat{W}})^2 \sum_{k = 0}^N (k + 1) \hat{\mat{W}}^k  \right\|_F^2 \right) \\
        &\lesssim \left( \E \left\| \id - (\id - \hat{\mat{W}})^2 \sum_{k = 0}^N (k + 1) \hat{\mat{W}}^k  \right\|_F^2 \right) \\
        &= \E \left\| \id - (\id - 2 \hat{\mat{W}} + \hat{\mat{W}}^2) \sum_{k = 0}^N (k + 1) \hat{\mat{W}}^k  \right\|_F^2 \\
        &= \E \left\| (N + 1)\hat{\mat{W}}^{N + 1} - N \hat{\mat{W}}^{N + 2}  \right\|_F^2 \\
        &\leq 2 \E \|(N + 1) \hat{\mat{W}}^{N + 1}\|_F^2 + 2 \E \|N \hat{\mat{W}}^{N + 2} \|_F^2 \rightarrow 0 \,,
    \end{split}
\end{equation}
where we have once again invoked \cref{lem:momentslemma} for the convergence on the last line. Note that we use Cauchy-Schwartz twice on the first line above. 
\end{proof}

\monotonesequencelemma*

\begin{proof}
Consider the following series of equivalent inequalities,
\begin{equation}
    \begin{split}
        \truncFac{N} &\geq \truncFac{N - 1} \,, \\
        \frac{\sum_{k = 1}^N a_k}{\sum_{k = 1}^{N} b_k} &\geq \frac{\sum_{k = 1}^{N - 1} a_k}{\sum_{k = 1}^{N - 1} b_k} \,, \\
        \left( \sum_{k = 1}^N a_k \right) \left(\sum_{k = 1}^{N - 1} b_k\right) &\geq \left(\sum_{k = 1}^{N} b_k\right) \left( \sum_{k = 1}^{N - 1} a_k\right) \,,\\
        \left(a_N + \sum_{k = 1}^{N - 1} a_k\right) \left(\sum_{k = 1}^{N - 1} b_k\right) &\geq \left(b_N + \sum_{k = 1}^{N - 1} b_k \right) \left( \sum_{k = 1}^{N - 1} a_k\right) \,,\\
        \sum_{k = 1}^{N - 1} a_N b_k &\geq \sum_{k = 1}^{N - 1} b_N a_k
    \end{split}
\end{equation}
 The last inequality above is clearly true because the terms in the sum on the left dominate their corresponding terms on the right. Therefore, the first inequality is also true.
 \end{proof}

\inequalitylemma*

\begin{proof}
We make a series of simplifications. The first assumption is that $\hat{\mat{X}}$ is a uniform random variable over a set of (not necessarily distinct) outcomes $\{ \mat{X}_1, ..., \mat{X}_N \}$, i.e., has distribution
\begin{equation}
    \mathcal{D} = \frac{1}{N} \sum_{i = 1}^N \delta_{\mat{X}_i} \,,
\end{equation}
where $\delta_{\mat{X}_k}$ is the delta distribution supported at $\mat{X}_k$. Since any continuous distribution can be approximated by a series of discrete distributions of this above form, it suffices to prove the statement for discrete distributions of the form above. Under this assumption, the inequality \cref{eq:transformationinequality} becomes
\begin{equation}
    \sum_{k,l} \langle \mat{X}^i_l \rangle_{\sMat} \, \langle \mat{X}^j_k \rangle_{\sMat} \leq \sum_{i,l} \langle \mat{X}^{i + r}_l \rangle_{\sMat} \, \langle \mat{X}^{j - r}_k \rangle_{\sMat} \,.
\end{equation}
Therefore, it suffices to consider individual pairs $\{i, l\}$ under the sum and show that for any $\mat{A}, \mat{B} \succeq \mat{0}$,
\begin{equation} \label{eq:tracenightmare}
\begin{split}
    &\langle \mat{B}^i \rangle_{\sMat} \, \langle \mat{A}^j \rangle_{\sMat} +  \langle \mat{A}^i \rangle_{\sMat} \, \langle \mat{B}^j \rangle_{\sMat} \\
    &\leq \langle \mat{A}^{j - r} \rangle_{\sMat} \, \langle \mat{B}^{i + r} \rangle_{\sMat} +  \langle \mat{B}^{j - r} \rangle_{\sMat} \, \langle \mat{A}^{i + r} \rangle_{\sMat} \,.
\end{split}
\end{equation}
Let $\lambda_i(\mat{A})$ and $\lambda_i(\mat{B})$ denote the eigenvalues of $\mat{A}, \mat{B}$ respectively. Note that since $\langle \cdot \rangle_{\sMat}$ is a linear functional it satisfies
\begin{equation}
    \langle \mat{A}^j \rangle_{\sMat} = \sum_i s_i \lambda_i(\mat{A})^j \,,
\end{equation}
for some $s_i \geq 0$ that don't depend on $j$. Therefore, \cref{eq:tracenightmare} amounts to
\begin{equation}
\begin{split}
     &\sum_{i,l} s_i s_l ' \left[ \lambda_i(\mat{A})^j \lambda_l(\mat{B})^k + \lambda_i(\mat{B})^j \lambda_l(\mat{A})^k \right] \\
     &\leq \sum_{i,l} s_i s_l ' \left[ \lambda_i(\mat{A})^{j - r} \lambda_l(\mat{B})^{k + r} + \lambda_i(\mat{B})^{j - r} \lambda_l(\mat{A})^{k + r} \right] \,.
\end{split}
\end{equation}
Since $s_i$ and $s_i'$ are non-negative, it therefore suffices to prove that for any non-negative $a, b \geq 0$,
\begin{equation} \label{eq:lessnightmare}
    a^j b^k + b^j a^k \leq a^{j - r} b^{k + r} + b^{j - r} a^{k + r}
\end{equation}

To prove \cref{eq:lessnightmare}, define the function
\begin{equation}
    C_{a, b, s}(\Delta) = a^{s - \Delta} b^{s + \Delta} + b^{s - \Delta} a^{s + \Delta} = 2 a^s b^s \cosh\left(\Delta \log(a/b) \right) \,.
\end{equation}
If we take $s = (k + j) / 2$, the claim \cref{eq:lessnightmare} can be rephrased as
\begin{equation}
    C_{a, b, s}\left(\frac{k - j}{2}\right) \leq C_{a, b, s}\left(\frac{k - j}{2} + r\right) \,,
\end{equation}
so it suffices to prove $C_{a, b, s}$ is monotonic in $\Delta$ for $\Delta \geq 0$ --- and this follows from the fact that $\cosh(x)$ is monotonically increasing for $x \geq 0$ and monotonically decreasing for $x \leq 0$. 
\end{proof}

\section{Source Code} \label{sec:source}

For reproducibility and reference purposes, we provide an implementation of all the algorithms in this paper and the corresponding benchmarks, \newline
\begin{center}
    \url{https://github.com/UniqueUpToPermutation/OperatorShifting}. 
\end{center}

\bibliographystyle{siamplain}
\bibliography{references}

\begin{thebibliography}{10}

\bibitem{anderson2010introduction}
{\sc G.~W. Anderson, A.~Guionnet, and O.~Zeitouni}, {\em An introduction to
  random matrices}, vol.~118, Cambridge university press, 2010.

\bibitem{aspri2020data}
{\sc A.~Aspri, Y.~Korolev, and O.~Scherzer}, {\em Data driven regularization by
  projection}, Inverse Problems, 36 (2020), p.~125009.

\bibitem{bleyer2013double}
{\sc I.~R. Bleyer and R.~Ramlau}, {\em A double regularization approach for
  inverse problems with noisy data and inexact operator}, Inverse Problems, 29
  (2013), p.~025004.

\bibitem{buccini2018semiblind}
{\sc A.~Buccini, M.~Donatelli, and R.~Ramlau}, {\em A semiblind regularization
  algorithm for inverse problems with application to image deblurring}, SIAM
  Journal on Scientific Computing, 40 (2018), pp.~A452--A483.

\bibitem{candes2010matrix}
{\sc E.~J. Candes and Y.~Plan}, {\em Matrix completion with noise}, Proceedings
  of the IEEE, 98 (2010), pp.~925--936.

\bibitem{davis1970rotation}
{\sc C.~Davis and W.~M. Kahan}, {\em The rotation of eigenvectors by a
  perturbation. iii}, SIAM Journal on Numerical Analysis, 7 (1970), pp.~1--46.

\bibitem{derezinski2020debiasing}
{\sc M.~Derezinski, B.~Bartan, M.~Pilanci, and M.~W. Mahoney}, {\em Debiasing
  distributed second order optimization with surrogate sketching and scaled
  regularization}, Advances in Neural Information Processing Systems, 33
  (2020), pp.~6684--6695.

\bibitem{derezinski2019distributed}
{\sc M.~Derezinski and M.~W. Mahoney}, {\em Distributed estimation of the
  inverse hessian by determinantal averaging}, Advances in Neural Information
  Processing Systems, 32 (2019).

\bibitem{golub1980analysis}
{\sc G.~H. Golub and C.~F. Van~Loan}, {\em An analysis of the total least
  squares problem}, SIAM journal on numerical analysis, 17 (1980),
  pp.~883--893.

\bibitem{james1992estimation}
{\sc W.~James and C.~Stein}, {\em Estimation with quadratic loss}, in
  Breakthroughs in statistics, Springer, 1992, pp.~443--460.

\bibitem{keshavan2009matrix}
{\sc R.~Keshavan, A.~Montanari, and S.~Oh}, {\em Matrix completion from noisy
  entries}, in Advances in neural information processing systems, 2009,
  pp.~952--960.

\bibitem{lunz2021learned}
{\sc S.~Lunz, A.~Hauptmann, T.~Tarvainen, C.-B. Schonlieb, and S.~Arridge},
  {\em On learned operator correction in inverse problems}, SIAM Journal on
  Imaging Sciences, 14 (2021), pp.~92--127.

\bibitem{marzouk2016introduction}
{\sc Y.~Marzouk, T.~Moselhy, M.~Parno, and A.~Spantini}, {\em An introduction
  to sampling via measure transport}, arXiv preprint arXiv:1602.05023,  (2016).

\bibitem{palmer2005representing}
{\sc T.~Palmer, G.~Shutts, R.~Hagedorn, F.~Doblas-Reyes, T.~Jung, and
  M.~Leutbecher}, {\em Representing model uncertainty in weather and climate
  prediction}, Annu. Rev. Earth Planet. Sci., 33 (2005), pp.~163--193.

\bibitem{nr}
{\sc R.~A. Rossi and N.~K. Ahmed}, {\em The network data repository with
  interactive graph analytics and visualization}, in AAAI, 2015,
  \url{http://networkrepository.com}.

\bibitem{rozemberczki2019gemsec}
{\sc B.~Rozemberczki, R.~Davies, R.~Sarkar, and C.~Sutton}, {\em Gemsec: Graph
  embedding with self clustering}, in Proceedings of the 2019 IEEE/ACM
  International Conference on Advances in Social Networks Analysis and Mining
  2019, ACM, 2019, pp.~65--72.

\bibitem{soize2005comprehensive}
{\sc C.~Soize}, {\em A comprehensive overview of a non-parametric probabilistic
  approach of model uncertainties for predictive models in structural
  dynamics}, Journal of sound and vibration, 288 (2005), pp.~623--652.

\bibitem{stein1956inadmissibility}
{\sc C.~Stein et~al.}, {\em Inadmissibility of the usual estimator for the mean
  of a multivariate normal distribution}, in Proceedings of the Third Berkeley
  symposium on mathematical statistics and probability, vol.~1, 1956,
  pp.~197--206.

\bibitem{stein2011fourier}
{\sc E.~M. Stein and R.~Shakarchi}, {\em Fourier analysis: an introduction},
  vol.~1, Princeton University Press, 2011.

\bibitem{tao2012topics}
{\sc T.~Tao}, {\em Topics in random matrix theory}, vol.~132, American
  Mathematical Soc., 2012.

\bibitem{tikhonov1963solution}
{\sc A.~N. Tikhonov}, {\em On the solution of ill-posed problems and the method
  of regularization}, in Doklady Akademii Nauk, vol.~151, Russian Academy of
  Sciences, 1963, pp.~501--504.

\bibitem{xiu2005high}
{\sc D.~Xiu and J.~S. Hesthaven}, {\em High-order collocation methods for
  differential equations with random inputs}, SIAM Journal on Scientific
  Computing, 27 (2005), pp.~1118--1139.

\bibitem{xiu2002wiener}
{\sc D.~Xiu and G.~E. Karniadakis}, {\em The wiener--askey polynomial chaos for
  stochastic differential equations}, SIAM journal on scientific computing, 24
  (2002), pp.~619--644.

\end{thebibliography}
\end{document}